\documentclass[12pt,a4paper]{amsart}
\usepackage{amsfonts}
\usepackage{amsthm}
\usepackage{amsmath}
\usepackage{amscd}
\usepackage{color}
\usepackage[latin2]{inputenc}
\usepackage{t1enc}
\usepackage[mathscr]{eucal}
\usepackage{indentfirst}
\usepackage{graphicx}
\usepackage{graphics}
\usepackage{pict2e}
\usepackage{epic}
\numberwithin{equation}{section}
\usepackage[margin=2.9cm]{geometry}
\usepackage{epstopdf} 
\usepackage{mathtools}
\usepackage{tikz}
\usetikzlibrary{positioning}
\usepackage{url}
\usepackage{hyperref}
 
 \usepackage[capitalize,nameinlink,noabbrev,nosort]{cleveref}
\usepackage{amssymb}
\usepackage{array}
\usepackage{wrapfig}
\usepackage{multirow}
\usepackage{tabu}
\usepackage{MnSymbol}
\usepackage{mathtools}

\theoremstyle{plain}
\newtheorem{Th}{Theorem}[section]
\newtheorem{Lemma}[Th]{Lemma}
\newtheorem{Cor}[Th]{Corollary}
\newtheorem{Prop}[Th]{Proposition}

 \theoremstyle{definition}
\newtheorem{Def}[Th]{Definition}

\newtheorem{Rem}[Th]{Remark}
\newtheorem{Prob}[Th]{Problem}
\newtheorem{Ex}[Th]{Example}

\newcommand{\xyFact}{\operatorname{xyFact}}
\newcommand{\xzFact}{\operatorname{xzFact}}
\newcommand{\yzFact}{\operatorname{yzFact}}
\newcommand{\TF}{\operatorname{TF}}
\newcommand{\id}{\operatorname{id}}

\newcommand{\mul}{\operatorname{mul}}
\newcommand{\wt}{\operatorname{wt}}
\newcommand{\zwt}{\operatorname{wt_{z}}}

\newcommand{\St}{\operatorname{St}}
\newcommand{\type}{\operatorname{type}}

\newcommand{\LC}{\operatorname{LC_z}}

\newcommand{\Ess}{\rm{Ess}}

\newcommand{\Sym}{\mathfrak{S}}
\newcommand{\Val}{\operatorname{Val}}
\newcommand{\length}{\operatorname{length}}
\newcommand{\last}{\operatorname{last}}
\newcommand{\ParSeq}{\operatorname{ParSeq}}
\newcommand{\code}{\operatorname{code}}
\newcommand{\Z}{\mathbb{Z}}
\newcommand{\z}{\mathbf{z}}
\newcommand{\x}{\mathbf{x}}
\newcommand{\y}{\mathbf{y}}
\newcommand{\wdown}{w^{\downarrow}}

\begin{document}
\title
[Schubert polynomials, the TASEP, and  evil-avoiding
permutations]
{Schubert polynomials, the inhomogeneous TASEP, and evil-avoiding
permutations}
\author{Donghyun Kim}
\address{Department of Mathematics, University of California, Berkeley, CA 94720}
\email{donghyun\_kim@berkeley.edu}%
\author{Lauren K. Williams}%
\address{Department of Mathematics,
            Harvard University,
            Cambridge, MA
            USA
}
\email{williams@math.harvard.edu}

\begin{abstract}
Consider a lattice of n sites arranged
around a ring, with the $n$ sites occupied by 
particles of weights $\{1,2,\ldots,n\}$; the possible
arrangements of particles in sites thus corresponds to the 
$n!$ permutations in $S_n$.
The \emph{inhomogeneous totally asymmetric simple
exclusion process} (or TASEP) is a Markov chain 
on $S_n$,
in which two adjacent particles of weights $i<j$
swap places at rate $x_i - y_{n+1-j}$ if the particle of weight $j$ is to the right of the particle of weight $i$.  (Otherwise nothing happens.)  
When $y_i=0$ for all $i$, the stationary
distribution was conjecturally linked to Schubert polynomials
\cite{LW}, and explicit formulas for steady 
state probabilities were subsequently given in terms
of multiline queues \cite{AL, AM}.
In the case of general $y_i$,  Cantini \cite{C} showed that $n$ of the 
$n!$ states have probabilities proportional to double
Schubert polynomials.  In this paper
we introduce the class of \emph{evil-avoiding permutations},
which are the permutations avoiding the patterns
$2413, 4132, 4213$ and $3214$.
We show that 
there are $\frac{(2+\sqrt{2})^{n-1}+(2-\sqrt{2})^{n-1}}{2}$
evil-avoiding permutations in $S_n$, and for each 
evil-avoiding permutation $w$, we give an explicit formula
for the steady 
state probability $\psi_w$ as a product
of double Schubert polynomials. (Conjecturally all 
other probabilities are proportional to a positive
sum of at least two Schubert polynomials.) 
When $y_i=0$ for all $i$, 
	we  give multiline queue formulas for the $\mathbf{z}$-deformed steady 
	state probabilities, and use this to prove the 
	monomial factor conjecture from \cite{LW}.
Finally, we show that the Schubert polynomials arising in our formulas are flagged Schur functions, and we give a bijection in this case between 
multiline queues and semistandard Young tableaux.  
\end{abstract}

\maketitle
\setcounter{tocdepth}{1}
\tableofcontents

\section{Introduction}
In recent years, there has been a lot of work 
on interacting particle models such as the 
\emph{asymmetric simple exclusion process} (ASEP),
a model in which particles hop on a one-dimensional
lattice subject to the condition that at most one particle
may occupy a given site.  The ASEP on a 
one-dimensional lattice with open boundaries has 
been linked to Askey-Wilson polynomials and Koornwinder
polynomials \cite{USW, CW1, C2, CW2}, 
while the ASEP on a ring has been linked to 
Macdonald polynomials \cite{CGW, CMW}.
The \emph{inhomogeneous totally asymmetric simple exclusion 
process} (TASEP) is a variant of the exclusion process 
on the ring in which the hopping rate depends on the weight of 
the particles.  In this paper we build on works of 
Lam-Williams \cite{LW}, Ayyer-Linusson \cite{AL},
and especially Cantini \cite{C} to give formulas for 
many steady state probabilities of the inhomogeneous TASEP on a ring in terms of Schubert polynomials.

\begin{Def}\label{def:TASEP}
Consider a lattice with $n$ sites arranged in a ring.
Let $\St(n)$ denote the $n!$ labelings
of the lattice by distinct numbers $1,2,\ldots,n$, where each number $i$ is called a 
\emph{particle of weight $i$}.  
The \emph{inhomogeneous TASEP on a ring of size $n$} is a Markov chain with state space $\St(n)$
where at each time $t$ a swap of two adjacent particles may occur: a particle of weight $i$ on the left swaps its position with a particle of weight $j$ on the right with transition rate $r_{i,j}$ given by:
$$
r_{i,j} = 
\begin{cases}  
x_i-y_{n+1-j} \text{ if $i<j$} \\
0 \text{ otherwise.}
\end{cases}
$$
\end{Def}

 In what follows, we will identify each state with a permutation in $S_n$.   Following \cite{LW, C}, we multiply all steady state probabilities for $\St(n)$ by the same constant, obtaining 
 ``unnormalized'' steady state probabilities  $\psi_w$, so that 
  \begin{equation}\label{eq:normalization}
 \psi_{123 \ldots n} = \prod_{1 \leq i<j\leq n} (x_i-y_{n+1-j})^{j-i-1} \in 
	  \Z[\x; \y],
 \end{equation}
	  where $\x$ and $\y$ denote 
	  $x_1,\ldots,x_{n-1}$ and $y_1,\ldots, y_{n-1}.$
  \cref{fig:S3} shows the state diagram and
 unnormalized steady state probabilities for $n=3$.  
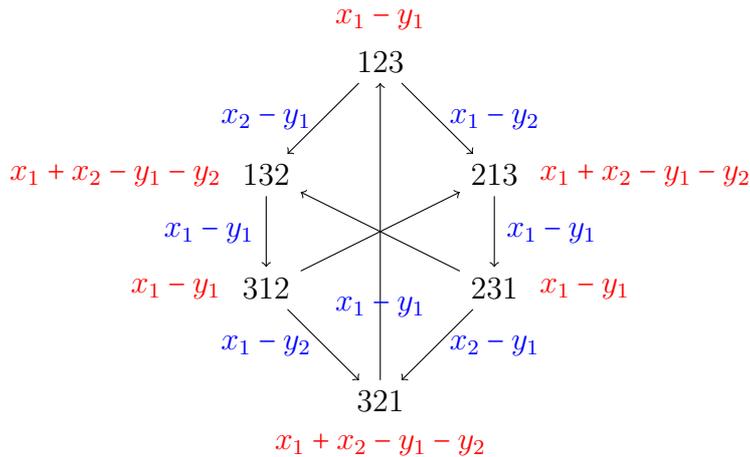
\begin{figure}[h]
\begin{center}

\begin{tikzpicture}[scale=1.5]
\node (0) [label={[red]below: $x_1+x_2-y_1-y_2$}] at (0,0) {$321$};
\node (12)  [label={[red]right: $x_1+x_2-y_1-y_2$}] at (1,2) {$213$};
\node (21)  [label={[red]left: $x_1+x_2-y_1-y_2$}] at (-1,2) {$132$};
\node (1) [label={[red]left: $x_1-y_1$}] at (-1,1) {$312$};
\node (2) [label={[red]right: $x_1-y_1$}] at (1,1) {$231$};
\node (121) [label={[red]above: $x_1-y_1$}] at (0,3) {$123$};

\draw[->] (0) -- node[near start,blue]{$x_1-y_1$} (121);
\draw[->](1)--node[near start,sloped, above=0.5pt,blue]{}(12);
\draw[->](2)--node[near start,sloped, above=0.5pt,blue]{} (21);
\draw[->](121)--node[left=0.5pt,blue]{$x_2-y_1$}(21);
\draw[->](121)--node[right=0.5pt,blue]{$x_1-y_2$}(12);
\draw[->](21)--node[left=0.5pt,blue]{$x_1-y_1$}(1);
\draw[->](12)--node[right=0.5pt,blue]{$x_1-y_1$}(2);
\draw[->](1) -- node[left=0.5pt,blue]{$x_1-y_2$}(0);
\draw[->](2)--node[right=0.5pt,blue]{$x_2-y_1$}(0);
\end{tikzpicture}
\caption{The state diagram for the inhomogeneous TASEP on $\St(3)$, with transition rates shown in blue, and unnormalized steady state probabilities $\psi_w$ in red.
Though not shown, the transition rate $312\to 213$ is $x_2-y_1$ and the 
transition rate $231\to 132$ is $x_1-y_2$.}
\label{fig:S3}
\end{center}
\end{figure}

In the case that $y_i=0$, Lam and Williams \cite{LW} 
studied this model\footnote{However the convention of \cite{LW} was slightly different; it corresponds to labeling states by the inverse of the permutations we use here.}
and conjectured that after a suitable normalization, each 
steady state probability $\psi_w$ can be written as a monomial
factor times a positive sum of Schubert polynomials, see 
\cref{table:1} and \cref{table:2}.  
They also gave an explicit formula for the monomial
factor, and conjectured that under certain conditions on 
$w$,  $\psi_w$ is a multiple of 
a particular Schubert polynomial. Subsequently  
Ayyer and Linusson \cite{AL} gave a conjectural combinatorial
formula for the stationary distribution 
in terms of \emph{multiline queues}, which was proved by 
Arita and Mallick \cite{AM}.
 In \cite{C}, Cantini introduced the version of 
 the model given in \cref{def:TASEP}\footnote{We
 note that in \cite{C}, the rate $r_{i,j}$  was $x_i-y_{j}$ rather than $x_i-y_{n+1-j}$ as we use in \cref{def:TASEP}.}
 with $y_i$ general, and, after introducing a further set of 
 spectral parameters $z_i$, gave a series of 
 \emph{exchange equations} relating the components $\psi_w(\z)$ of 
 the \emph{$\z$-deformed} stationary distribution.  This allowed him to give explicit formulas for the steady state probabilities for $n$ of the 
$n!$ states as products of double
Schubert polynomials. 

\begin{table}[h!]
\centering
\begin{tabular}{|c c |}
    \hline
    State $w$ & Probability $\psi_w$\\
    \hline 
    1234 & $(x_1-y_1)^2 (x_1-y_2)(x_2-y_1)$\\
    1324 & $(x_1-y_1) \Sym_{1432}$\\
    1342 & $(x_1-y_1)(x_2-y_1) \Sym_{1423}$\\
    1423 & $(x_1-y_1)(x_1-y_2)(x_2-y_1) \Sym_{1243}$\\
    1243 & $(x_1-y_2)(x_1-y_1) \Sym_{1342}$\\
    1432 & $\Sym_{1423} \Sym_{1342}$\\ 
    \hline
    \end{tabular}
    \caption{The unnormalized steady state probabilities for $n=4$.}
    \label{table:1}
    \end{table}

In this paper we build on \cite{C, AL, AM} to strengthen the connection 
between 
steady state probabilities
and Schubert polynomials: in particular, we give
a formula for $\psi_w$ as a ``trivial factor'' times a product of (double) Schubert polynomials whenever $w$ is \emph{evil-avoiding}, that is, it
avoids the patterns
$2413, 4132, 4213$ and $3214$.\footnote{We call these permutations \emph{evil-avoiding} because if one 
replaces $i$ by $1$, $e$ by $2$, $l$ by $3$, and $v$ by $4$, 
then \emph{evil} and its anagrams \emph{vile, veil} and \emph{leiv} become the four patterns $2413, 4132, 4213$ and $3214$.  (Leiv is a name of Norwegian origin meaning ``heir.'')}  
We show that 
there are $\frac{(2+\sqrt{2})^{n-1}+(2-\sqrt{2})^{n-1}}{2}$
evil-avoiding permutations in $S_n$, so this gives a 
substantial generalization of the previous result
\cite{C} linking probabilities to Schubert polynomials.
In the case that $y_i=0$ for all $i$,
we also give a formula for the $\mathbf{z}$-deformed steady state probabilities 
$\psi_w(\z)$ in terms of multiline queues, generalizing  the 
result of Arita and Mallick; we then use this result to 
prove the 
monomial factor conjecture from \cite{LW}.  
Finally, we show that the Schubert polynomials that arise in our formulas are flagged Schur functions, and give a bijection in this case between multiline queues and semistandard Young tableaux.

\begin{center}
\begin{table}[h!]
    \begin{tabular}{|c c |}
    \hline
    State $w$ & Probability $\psi_w$\\
    \hline 
    12345 & $\mathbf{x}^{(6,3,1)}$\\
    12354 & $\mathbf{x}^{(5,2,0)} \Sym_{13452}$\\
    12435 & $\mathbf{x}^{(4,1,0)} \Sym_{14532}$\\
    12453 & $\mathbf{x}^{(4,1,1)} \Sym_{14523}$\\
    12534 & $\mathbf{x}^{(5,2,1)} \Sym_{12453}$\\
    12543 & $\mathbf{x}^{(3,0,0)} \Sym_{14523} \Sym_{13452}$\\
    13245 & $\mathbf{x}^{(3,1,1)} \Sym_{15423}$\\
    13254 & $\mathbf{x}^{(2,0,0)} \Sym_{15423} \Sym_{13452}$\\
    13425 & $\mathbf{x}^{(3,2,1)} \Sym_{15243}$\\
    13452 & $\mathbf{x}^{(3,3,1)} \Sym_{15234}$\\
    13524 & $\mathbf{x}^{(2,1,0)} (\Sym_{164325}+\Sym_{25431})$\\
    13542 & $\mathbf{x}^{(2,2,0)} \Sym_{15234} \Sym_{13452}$\\
    14235 & $\mathbf{x}^{(4,2,0)} \Sym_{13542}$\\
    14253 & $\mathbf{x}^{(4,2,1)} \Sym_{12543}$\\
    14325 & $\mathbf{x}^{(1,0,0)} 
(\Sym_{1753246}+\Sym_{265314}+\Sym_{2743156}+\Sym_{356214}+\Sym_{364215}+\Sym_{365124})$\\
    14352 & $\mathbf{x}^{(1,1,0)} \Sym_{15234} \Sym_{14532}$\\
    14523 & $\mathbf{x}^{(4,3,1)} \Sym_{12534}$\\
    14532 & $\mathbf{x}^{(1,1,1)} \Sym_{15234} \Sym_{14523}$\\
    15234 & $\mathbf{x}^{(5,3,1)} \Sym_{12354}$\\
    15243 & $\mathbf{x}^{(3,1,0)}(\Sym_{146325}+\Sym_{24531})$\\
    15324 & $\mathbf{x}^{(2,1,1)} (\Sym_{15432}+\Sym_{164235})$\\
    15342 & $\mathbf{x}^{(2,2,1)} \Sym_{15234}\Sym_{12453}$\\
    15423 & $\mathbf{x}^{(3,2,0)} \Sym_{12534}\Sym_{13452}$\\
    15432 & $\Sym_{15234} \Sym_{14523}\Sym_{13452}$\\
    \hline\end{tabular}
    \caption{The unnormalized steady state probabilities for $n=5$, when
    each $y_i=0$.  In the table, $\mathbf{x}^{(a,b,c)}$ denotes $x_1^a x_2^b x_3^c$.}
    \label{table:2}
    \end{table}
\end{center}

In order to state our main results, we need a few definitions.
First, we say that two states $w$ and $w'$ are \emph{equivalent}, and write
$w\sim w'$, if one state is a cyclic shift of the other, e.g.   
$(w_1,\ldots,w_n) \sim (w_2,\ldots,w_n,w_1)$.  Because of the cyclic symmetry inherent in the definition of the TASEP on a ring, it is clear that the probabilities of states $w$ and $w'$ are equal whenever $w\sim w'$.  
We will therefore often assume, without loss of generality, 
that $w_1=1$.  Note that up to cyclic shift, $\St(n)$ contains $(n-1)!$ states.

In \cref{MFC} below we describe the 
monomial factors that appear in the steady state probabilities.
Suppose that $y_i=0$ for all $i$.

\begin{Def}\label{def:alpha}
Given $w=(w_1,\ldots,w_n)\in \St(n)$, let $r=w^{-1}(i+1)$ and $s=w^{-1}(i)$.  
Let $\alpha_i(w)$ be the number of integers greater than $i+1$ 
among $\{i+1=w_r, w_{r+1}, \ldots , w_s=i\}$, where the subscripts
are taken modulo $n$.  
Let $\eta(w)$ be the largest monomial that can be factored
out of $\psi_w$.   
\end{Def}

\begin{Th}\label{MFC} \cite[Conjecture 2]{LW}
Let $y_i=0$ for all $i$.  
For $w\in \St(n)$, we have
$$\eta(w) = \prod_{i=1}^{n-2} x_i^{\alpha_i(w)+\cdots+\alpha_{n-2}(w)}.$$
Moreover if two states $w, w'\in \St(n)$ have the same $\eta(w)=\eta(w')$, then $w$ and $w'$ are cyclically equivalent i.e $w \sim w'$.
\end{Th}

\begin{Rem}
For the inverse map from $(\alpha_1(w),\ldots,\alpha_{n-2}(w))$ to a cyclic
	equivalence class $[w]\in \St(n)$, see \cref{permutation indexing} and \cref{lemma: construction of ST(N)}.
\end{Rem}

    See \cref{table:2} for examples.
We now introduce some definitions needed to characterize
the Schubert polynomial factors that appear in the 
probabilities $\psi_w$.

\begin{Def}\label{def:kGrassmannian}
Let $w=(w_1,\ldots,w_n)\in \St(n)$. 
We say that $w$ is a 
\emph{$k$-Grassmannian permutation}, 
and we write
$w\in \St(n,k)$ if: $w_1=1$; $w$ is \emph{evil-avoiding}, i.e. 
    $w$ avoids
    the patterns $2413$, $3214$, $4132$, and $4213$; and 
    $w$ has $k$ \emph{recoils}, that is, letters $a$ in $w$ such that $a+1$ appears to the left of $a$ in $w$.
    (Equivalently, $w^{-1}$ has exactly $k$ \emph{descents}.)
\end{Def}

\begin{Def}\label{def:bijpartition}
We associate to each $w\in\St(n,k)$ a sequence of partitions
$\Psi(w)=(\lambda^1,\ldots,\lambda^k)$ 
 as follows. Write the  code (cf. \cref{def:code}) of $w^{-1}$ as 
 $c(w^{-1}) = c = (c_1,\ldots,c_n);$
since $w^{-1}$ has $k$ descents, 
$c$ has $k$ \emph{descents} in positions we denote by
$a_1,\ldots,a_k$.  Set $a_0=0$.  For $1 \leq i \leq k$,
define
$\lambda^i = (n-{a_i})^{a_i} - (\underbrace{0,\ldots,0}_\text{$a_{i-1}$},c_{a_{i-1}+1},
c_{a_{i-1}+2},\ldots, c_{a_i}).$
\end{Def}
See \cref{table:3} for examples of the map $\Psi(w)$.

\begin{Def}\label{def:g} Given a positive integer $n$ and a partition $\lambda$ properly contained in a $\length(\lambda) \times (n-\length(\lambda))$ rectangle (we will later use the notation $\lambda\in\Val(n)$), we define an integer vector 
$g_n(\lambda)=(v_1,\ldots,v_n)$ of length $n$ as follows.  Write $\lambda=(\mu_1^{k_1},\ldots,\mu_l^{k_l})$ where $k_i>0$ and $\mu_1>\cdots>\mu_l$.
We start with $(v_1,\ldots,v_n)=(0,\ldots,0)$ and perform the following step for $i$ from $1$ to $l$.
\begin{itemize}
    \item (Step $i$) Set $v_{n-\mu_i}$ equal to $\mu_i$. Moving to the left, assign the value $\mu_i$ to the first $(k_i-1)$ entries which are zero.
\end{itemize}
\end{Def}

\begin{Rem}\label{rem:g}
Note that in Step 1, we set $v_{n-\mu_1},v_{n-\mu_1-1},\ldots,v_{n-\mu_1-k_1+1}$ equal to $\mu_1$.
\end{Rem}

\begin{Ex}
	We have
\begin{equation*}
	g_5((2,1,1))=(0,1,2,1,0), \hspace{.1cm}
	g_6((3,2,2,1))=(0,2,3,2,1,0), \hspace{.1cm}
     g_6((3,1,1))=(0,0,3,1,1,0).
\end{equation*}
For example, let $g_6((3,2,2,1))=(v_1,\ldots,v_6)$. In Step 1, we set $v_3=3$. In Step 2, we first set $v_4=2$, then moving to the left we set $v_2=2$. In Step 3, we set $v_5=1$.
\end{Ex}

The main result of this paper is
\cref{thm:main1}.  The definition of Schubert polynomial 
can be found in \cref{sec:background}.

\begin{Def}\label{def:xy}
We write $a\to b \to c$ if the letters $a, b, c$ appear in cyclic order in $w$.
So for example, if $w=1423$, we have that $1\to 2\to 3$ and $2\to 3\to 4$,
but it is not the case that $3 \to 2 \to 1$ or $4\to 3 \to 2$.
We then define
\begin{equation}\label{eq:xy}
   \xyFact(w)=\prod\limits_{i=1}^{n-2}\prod_{\substack{i+2 \leq k \leq n \\ i \to i+1 \to k}}(x_1-y_{n+1-k})\cdots(x_i-y_{n+1-k}). 
\end{equation}
\end{Def}

\begin{Th}\label{thm:main1}
Let $w\in \St(n,k)$  be a 
\emph{$k$-Grassmannian} permutation, as 
in \cref{def:kGrassmannian}, and write  $\Psi(w)=(\lambda^{1},\ldots,\lambda^{k})$. Then the (unnormalized) steady state probability is given by  
\begin{equation}\label{eq:maintheorem1}
    \psi_w=\xyFact(w)  \prod_{i=1}^k \Sym_{c^{-1}(g_n(\lambda^i))},
    \end{equation}
where $\Sym_{c^{-1}(g_n(\lambda^i))}$ is the double Schubert polynomial
associated to the permutation with  code
$g_n(\lambda^i)$, and $g_n$ is given by \cref{def:g}.
\end{Th}

In the case that each $y_i=0$, \cref{thm:main1} becomes
\cref{thm:main0} below.

\begin{Th}\label{thm:main0}
Let $w\in \St(n,k)$, and   
let $\Psi(w)=(\lambda^1,\ldots,\lambda^k)$.
Let
$\mu$ be the vector $\mu:=({n-1\choose 2}, {n-2 \choose 2},\ldots ,{2 \choose 2})-\sum\limits_{i=1}^k \lambda^i,$ 
where we view each 
partition $\lambda^i$ as 
a vector in $\Z_{\geq 0}^{n-2}$,  
adding trailing $0$'s if necessary. 
Then when each $y_i=0$,
the unnormalized steady state probability 
$\psi_w$ is given by 
$$\psi_w = \mathbf{x}^{\mu} \prod_{i=1}^k \Sym_{c^{-1}(g_n(\lambda^i))},$$
where $\Sym_{c^{-1}(g_n(\lambda^i))}$ is the Schubert polynomial
of the permutation with  code
$g_n(\lambda^i)$.

Equivalently, writing $\lambda^i = (\lambda_1^i, \lambda_2^i,\ldots)$, we have that 
$$\psi_w = \mathbf{x}^{\mu} \prod_{i=1}^k
s_{\lambda^i}(X_{n-\lambda_1^i}, X_{n-\lambda_2^i},\ldots),$$
where $s_{\lambda^i}(X_{n-\lambda_1^i}, X_{n-\lambda_2^i},\ldots)$ denotes the flagged
Schur polynomial associated to shape $\lambda^i$,
where the semistandard tableaux entries in row $j$ are bounded above by 
$n-\lambda_j^i$.  
\end{Th}

\begin{Rem}
	It follows from \cite[Theorem 7.2]{AS} that if we set each $y_i=0$, 
	the \emph{partition function}
	$Z_n = \sum_{w\in S_n} \psi_w$ is equal to 
	$\prod_{i=1}^n h_{n-i}(x_1, x_2,\ldots, x_{i-1}, x_i, x_i),$
	where $h_i$ is the homogeneous symmetric polynomial.
	Note that $Z_n$  has 
	$\prod_{i=0}^{n} {n \choose i}$ terms.
	However, if the $y_i$ are not zero, then $Z_n$ does not factor.
\end{Rem}

We illustrate \cref{thm:main0} in \cref{table:3} in the case that $n=5$.
The quantity $s(w)$ is defined in \cref{def:shiftingvector}.

\begin{center}
\begin{table}[h!]
    \begin{tabular}{|c c c c c|}
    \hline
k &     $w\in \St(5,k)$ & $\Psi(w)$ & probability $\psi_w$ & $s(w)$\\
    \hline 
0 &    12345 & $\emptyset$ & $\x^{(6,3,1)}$ & $(0)$\\
\hline 
 1&   12354 & $(1,1,1)$ & $\x^{(5,2,0)} \Sym_{13452}$ & $(0)$\\
   1& 12435 & $(2,2,1)$ & $\x^{(4,1,0)} \Sym_{14532}$ & $(0)$\\
    1& 12453 & $(2,2)$ & $\x^{(4,1,1)} \Sym_{14523}$ &$(0)$\\
    1& 12534 & $(1,1)$ & $\x^{(5,2,1)} \Sym_{12453}$ & $(0)$\\
    1& 13245 & $(3,2)$ & $\x^{(3,1,1)} \Sym_{15423}$ & $(0)$\\
    1& 13425 & $(3,1)$& $ \x^{(3,2,1)} \Sym_{15243}$ &$(0)$\\
    1& 13452 & $(3)$ & $\x^{ (3,3,1)} \Sym_{15234}$ & $(0)$\\
    1& 14235 & $(2,1,1)$ & $\x^{(4,2,0)}  \Sym_{13542}$ & $(0)$\\
    1& 14253 & $(2,1)$ & $ \x^{(4,2,1)} \Sym_{12543}$ & $(0)$\\
    1& 14523 & $(2)$ & $ \x^{(4,3,1)} \Sym_{12534}$ & $(0)$\\
    1& 15234 & $(1)$ & $\x^{(5,3,1)} \Sym_{12354}$ & $(0)$\\
        \hline 
    2& 12543 & $(2,2), (1,1,1)$& $\x^{(3,0,0)} \Sym_{14523} \Sym_{13452}$ & $(0,-1)$\\
    2& 13254 & $(3,2), (1,1,1)$ & $\x^{(2,0,0)} \Sym_{15423} \Sym_{13452}$ & $(0,0)$\\
    2& 13542 & $(3), (1,1,1)$ & $\x^{(2,2,0)} \Sym_{15234} \Sym_{13452}$ & $(0,-1)$\\
    2& 14352 & $(3), (2,2,1)$ & $\x^{(1,1,0)} \Sym_{15234} \Sym_{14532}$ & $(0,-1)$\\  
    2& 14532 & $(3), (2,2)$ & $\x^{(1,1,1)} \Sym_{15234} \Sym_{14523}$ & $(0,-1)$\\
    2& 15342 & $(3), (1,1)$ & $\x^{(2,2,1)} \Sym_{15234}\Sym_{12453}$ & $(0,-1)$\\
    2& 15423 & $(2),(1,1,1)$ & $\x^{(3,2,0)} \Sym_{12534}\Sym_{13452}$ & $(0,-2)$\\
    \hline 
3 &    15432 & $(3), (2,2), (1,1,1)$ & $\Sym_{15234} \Sym_{14523} \Sym_{13452}$ & $(0,-1,-2)$\\
    \hline\end{tabular}
    \caption{Special states $w\in\St(5,k)$, the corresponding sequences of partitions $\Psi(w)$,
	 the steady state probabilities $\psi_w$, and vectors $s(w)$. }
    \label{table:3}
    \end{table}
\end{center}

\begin{Prop}\label{prop:enumeration}
The number of evil-avoiding permutation in 
$S_n$ satisfies the recurrence	
$e(1)=1, e(2)=2, 
e(n)=4 e(n-1)-2e(n-2)$ for $n \geq 3$, and is given 
by \begin{equation}
    \label{eq:a}
e(n)=\frac{(2+\sqrt{2})^{n-1}+(2-\sqrt{2})^{n-1}}{2}.
\end{equation}
This sequence begins as $1, 2, 6, 20, 68, 232$, and 
	occurs in \cite{Sloane}
	as sequence A006012.
The cardinalities 
	$|\St(n,k)|$ also occur in \cite{Sloane} as sequence 
A331969.
\end{Prop}

\begin{Rem}
Let $w(n,h): = (h, h-1,\ldots,2,1,h+1,h+2,\ldots,n)\in \St(n).$
In \cite[Corollary 16]{C}, Cantini gives a formula for the steady 
state probability of state $w(n,h)$, as a trivial factor times a product of certain (double) Schubert polynomials.  Our main result generalizes this one.
For example,
for $n=4$ and $n=5$, \cite[Corollary 16]{C} 
gives a formula for the probabilities of 
	states $(1,2,3,4)$, $(1,3,4,2)$, $(1,4,3,2)$,
	and the probabilities of states
$(1,2,3,4,5), (1,3,4,5,2), (1,4,5,3,2)$, and $(1,5,4,3,2).$
Meanwhile, \cref{thm:main0} gives a formula for all six states when $n=4$ (see \cref{table:1}) and $20$ of the $24$ states when $n=5$.
Using \eqref{eq:a}, we see that asymptotically
	\cref{thm:main0} 
gives a formula for roughly $\frac{(2+\sqrt{2})^{n-1}}{2}$ 
out of the 
$(n-1)!$ states of $\St(n).$

Note also that the Schubert polynomials that occur in the formulas of \cite{C} are all of the form $\Sym_{\sigma(a,n)}$, where 
$\sigma(a,n)$ denotes the permutation $(1, a+1, a+2,\ldots, n, 2, 3, \ldots,a).$  However, many of the Schubert polynomials arising as (factors) of steady probabilities are not of this form.  Already we see for $n=4$ the Schubert polynomials $\Sym_{1432}$ and $\Sym_{1243}$, which are not of this form.
\end{Rem}

This paper is organized as follows.
In \cref{sec:background} we provide background
on partitions, permutations, and Schubert polynomials.  
In \cref{sec:evil-avoiding} we explore the 
combinatorics of evil-avoiding
and $k$-Grassmannian permutations.
In \cref{sec:Cantini} we present  Cantini's results giving
recursive formulas for
  \emph{$\mathbf{z}$-deformed} steady state probabilities of the inhomogeneous TASEP.
In \cref{sec:z} we state  
\cref{thm:maintechnical}, which says that
for $w\in \St(n,k)$, the $\mathbf{z}$-deformed steady state probability $\psi_w(\mathbf{z})$
is equal to a ``trivial factor'' times a product of \emph{$\mathbf{z}$-Schubert polynomials} -- 
certain polynomials which reduce to double Schubert polynomials when 
$\mathbf{z}=\infty$ (as we prove in \cref{sec:zproperties}).
\cref{section:mainproof} is devoted to the proof of 
\cref{thm:maintechnical}, 
which in turn implies 
\cref{thm:main1} and \cref{thm:main0}. 
In \cref{sec:MLQ} we recall Arita-Mallick's formula
for steady state probabilities in terms of \emph{multiline queues} (when $y_i=0$); we then 
show that when $w^{-1}$ is a Grassmann permutation, the multiline queues of type $w$
are in bijection with semistandard tableaux.
In \cref{sec:zMLQ} we 
prove a multiline queue formula
for $\mathbf{z}$-deformed steady state probabilities (when $y_i=0$), which generalizes
Arita-Mallick's result.
Using this formula  we
prove \cref{MFC} (the monomial factor conjecture) in \cref{sec:monomial}.
\cref{sec:future} presents a few remarks on future directions, and 
 \cref{sec:technical} is an appendix containing some technical results which are used
in the  proof of \cref{thm:maintechnical}.

\textsc{Acknowledgements:}
The authors would like to thank the referee for their helpful comments,
which improved the exposition of this paper.
L.W. would like to thank Allen Knutson for interesting
discussions, and 
 would like to acknowledge the support of the National Science Foundation
under agreements No.\ DMS-1854316 and No.\ DMS-1854512, as well as the support of the
Radcliffe Institute for Advanced Study at Harvard University,
 where some of this work ``took place'' (virtually).  Any opinions,
findings and conclusions or recommendations expressed in this material
are those of the authors and do not necessarily reflect the
views of the National Science Foundation.

\section{Background on partitions, permutations and Schubert polynomials}\label{sec:background}

We let $S_n$ denote the symmetric group on $n$ letters, which is a Coxeter group 
generated by the simple reflections $s_1,\ldots, s_{n-1}$, where $s_i$ is the simple 
transposition exchanging $i$ and $i+1$.  We let 
$w_0=(n,n-1,\ldots, 2,1)$ denote the longest permutation.

For $1 \leq i < n$, we have the \emph{divided difference operator} $\partial_i$ which acts on polynomials 
$P(x_1,\ldots,x_n)$ as follows:
$$(\partial_i P)(x_1,\ldots,x_n) = \frac{P(\ldots, x_i, x_{i+1},\ldots ) - P(\ldots,x_{i+1},x_i,\ldots)}{x_i-x_{i+1}}.$$
If $s_{i_1}\cdots s_{i_m}$ is a reduced expression for a permutation $w$, then
$\partial_{i_1} \cdots \partial_{i_m}$ depends only on $w$, so we denote this operator
by $\partial_w$.  

\begin{Def}
Let $\mathbf{x}=(x_1,\ldots,x_n)$ and $\mathbf{y}=(y_1,\ldots,y_n)$ be two sets of variables,
and let $$\Delta(\mathbf{x},\mathbf{y}) = \prod_{i+j \leq n} (x_i - y_j).$$
To each permutation $w\in S_n$ we associate the \emph{double Schubert polynomial}
$$\Sym_w(\mathbf{x},\mathbf{y}) = \partial_{w^{-1} w_0} \Delta(\mathbf{x},\mathbf{y}),$$
where the \emph{divided difference operator} acts on the $x$-variables.
\end{Def}

\begin{Def}\label{def:partition}
A \emph{partition} $\lambda = (\lambda_1,\ldots,\lambda_r)$ is a weakly 
decreasing sequence of positive integers.
We set $\length(\lambda):=r$ and call it the \emph{length} of $\lambda$.
We let $\mul(\lambda)$ be the multiplicity of the largest part $\lambda_1$ in $\lambda$,
and let $\lambda_{\last}$ denote the smallest part of $\lambda$.
\end{Def}
For example, if $\lambda = (6,6,4,3,3)$ then $\length(\lambda)=5$, $\mul(\lambda)=2$,
and $\lambda_{\last} = 3$.

\begin{Def}\label{def:code}
%
	Given a permutation $w\in S_n$, 
	its \emph{(Lehmer) code} is the vector 
	$c(w)=(v_1,\ldots,v_n)$ where $v_i = |\{j>i \ \vert \ w_j<w_i\}|$.
If $(v_1,\ldots,v_n)$ is the code of a permutation, we let $c^{-1}(v_1,\ldots,v_n)$ denote the permutation with  code  $(v_1,\ldots,v_n)$. The 
	\emph{shape} $\lambda(w)$ of $w$ is the
	partition obtained by sorting the components of the code.
\end{Def}

\begin{Ex} If $w = (1,3,5,4,2)$, then $c(w)=(0,1,2,1,0)$ and 
$\lambda(w) = (2,1,1)$.
\end{Ex}

The following is well-known.
\begin{Lemma}\label{lem:code0}
 Given a vector $(c_1,\ldots,c_n)\in \Z_{\geq 0}^n$, there exists a permutation $w\in S_n$ such that $c(w)=(c_1,\ldots,c_n)$ if and only if $c_i+i\leq n$ for all $1\leq i \leq n$.
\end{Lemma}

\begin{Def}
We say that a permutation $w$ is \emph{vexillary} if and only if there does not 
exist a sequence $i<j<k<\ell$ such that $w(j)<w(i)<w(\ell)<w(k)$.  Such a permutation
is also called \emph{$2143$-avoiding.}
\end{Def}

\begin{Def}\label{def:flag}
The \emph{flag} of a vexillary permutation $w$ is defined in terms of 
	its code $c(w)$ as follows.  If $c_i(w) \neq 0$, let $e_i$ be the greatest integer $j\geq i$ such that 
$c_j(w) \geq c_i(w)$.  The flag $\phi(w)$ is then the sequence of integers $e_i$, ordered to be increasing.
\end{Def}

\begin{Def}\label{def:flaggedtableaux}
Let $X_i$ denote the family of indeterminates $x_1,\ldots,x_i$.  
For $d=(d_1,\ldots,d_n)$ a weakly increasing sequence of $n$ integers, we define $SSYT(\lambda,d)$ to be the set of semistandard tableaux $T$ with shape $\lambda$ for which the entries in the $i$th row are bounded above by $d_i$.
We also define the \emph{flagged Schur function} 
$$s_{\lambda}(X_{d_1},\ldots,X_{d_n}) = \sum_{T\in SSYT(\lambda,d)} \mathbf{x}^{\type(T)}.$$
\end{Def}

\begin{Rem}
	There are also flagged \emph{double} Schur functions, which can be defined in terms of 
	tableaux or via a Jacobi-Trudi type formula \cite[Section 2.6.5]{Manivel}. 
	Below we use Lascoux's multi-Schur function notation, see \cite[Definition 2.6.4]{Manivel}.
\end{Rem}

\begin{Th}\cite[Corollary 2.6.10]{Manivel}\label{th:flag}
If $w$ is a vexillary permutation with shape $\lambda(w)$ and with flags $\phi(w) = (f_1,\ldots,f_m)$ and $\phi(w^{-1})=(g_1,\ldots,g_m)$, then we have
$$\Sym_w(\mathbf{x}; \mathbf{y}) = s_{{\lambda}(w)}(X_{f_1}-Y_{g_m},\ldots,X_{f_m}-Y_{g_1}),$$
i.e. the double Schubert polynomial of $w$ is a flagged double Schur polynomial.
\end{Th}

\section{Combinatorics of evil-avoiding
and $k$-Grassmannian permutations}\label{sec:evil-avoiding}

In this section we study
the special states of our Markov chain whose probabilities are proportional to products of Schubert polynomials.  
Recall our definition of $k$-Grassmannian permutations
$\St(n,k)$
from \cref{def:kGrassmannian}. As we will see, the set $\St(n,k)$ is in bijection with a certain set $\ParSeq(n,k)$ of sequences $(\lambda^1,\ldots,\lambda^k)$ of $k$ partitions.
Recall that our main result (see \cref{thm:main1}) states that the probability of each state
in $\St(n,k)$ is proportional to a product of $k$ Schubert polynomials, which are determined by the corresponding 
sequence of partitions.

\begin{Rem}
Note that $w$ contains a pattern $p$ if and only if $w^{-1}$ contains the pattern $p^{-1}$.  So $w$ is evil-avoiding if and only if $w^{-1}$ avoids $3142$, $2431$, $3241$, and $3214$.
\end{Rem}

We say that a sequence $(w_1,w_2,\ldots,w_n)\in \Z^n$ has
a \emph{descent} in position $j$ if $w_j>w_{j+1}$.

\begin{Prop}\label{prop:pandemic}
Let $c=(c_1,c_2,\ldots,c_n)$ be the code of $w\in S_n$.  
Then $w$ avoids
the patterns $3142$, $3214$, $2431$, and $3241$ (equivalently, $w^{-1}$ is 
evil-avoiding) if and only if 
for each descent position $j$, if there is a $b\leq j$ such that:
\begin{itemize}
    \item
$w_b<w_{b+1} < \cdots < w_j$ and $0<c_b<n-j$, and $b$ is maximal with these properties,
\end{itemize}
then we must have $c_{j+1} = c_{j+2} = \cdots = c_{j+c_b}=0$.
\end{Prop}

\begin{Rem}\label{rem:order}
Let $w=(w_1,\ldots,w_n)\in S_n$.  Note that having the $b$th entry of the code $c=\code(w)$ equal to  $c_b$ means that there are precisely $c_b$ letters less than $w_b$ which do not occur in positions $1$ through $b$.  The condition $c_{j+1}=c_{j+2} = \cdots = c_{j+c_b}=0$ means that these letters must occur in increasing order in positions $j+1,\ldots, j+c_b$.
\end{Rem}

\begin{Rem}
If $w\in \St(n,1)$  then
 $\code(w^{-1})$ has a unique descent.
This means
that $w^{-1}$ is a \emph{Grassmannian permutation}
 \cite[Definition 2.2.3]{Manivel}.
Equivalently, there is only one $a$ 
such that $a+1$ appears to the left of letter $a$ in $w$. Conversely, if $w$ is an inverse of a Grassmannian permutation that starts with 1, we have $w\in \St(n,1)$.
\end{Rem}

\begin{proof}[Proof of \cref{prop:pandemic}]
We start by showing that if $w$ fails to satisfy the condition  of \cref{prop:pandemic}, then it must contain one of the patterns 
$3142$, $3214$, $2431$, and $3241$. 
If $w$ fails to satisfy the condition of \cref{prop:pandemic},
then there is a descent position $j$ and a $b\leq j$ with 
$w_b < w_{b+1} <\cdots < w_j$ and $0<c_b<n-j$, and $b$ is maximal with 
this property, but we do not have 
$c_{j+1} = c_{j+2} = \cdots = c_{j+c_b} = 0.$
\begin{enumerate}
\item We have 
$w_b < w_{b+1} < \cdots <w_{j}>w_{{j+1}}$.  
\item By the definition of code, there must be $c_b$ letters $\mathcal{C}$ smaller than
$w_b$ which appear to the right of $w_b$; they must therefore appear
to the right of $w_{j}$.  
\item \label{third:bigger} Let $\ell\in [j,n]$ be minimal such that 
		$w_{\ell}>w_b$, which exists since  
	$c_b<n-j$.
\item \label{fourth} 
The fact that we do not have $c_{j+1} = c_{j+2} = \cdots = c_{j+c_b} = 0$
implies that the letter $w_{\ell}$ appears to the left of 
some letter $w_r\in \mathcal{C}$, i.e. $\ell<r$.
\end{enumerate} 

Let us first consider the case that $b=j$.  If $c_{b+1} \neq 0$
then $w_{b+1}$ is not the smallest letter to appear in positions
$[b+1,n]$, so there exists some $m>b+1$ such that $w_{b+1}>w_m$.  
Therefore the letters $\{w_b, w_{b+1}, w_m, w_{\ell}\}$ give either
an instance of the pattern $3214$ or of the pattern $3241$,
depending on whether $m<\ell$ or $\ell< m$.

If $b=j$ and $c_{b+1}=0$ then $w_{b+1}$ is the smallest letter
to appear in positions $[b+1,n]$.  
But then the letters $\{w_b, w_{b+1}, w_{\ell}, w_r\}$
form the pattern $3142$.

If $b<j$, the fact that $b$ is maximal such that $j+c_b<n$ implies
that $j+c_{b+1} \geq n$.  Since $w_{b+1} < \cdots < w_{j}$,
the $c_{b+1}$ letters which are less than $w_{b+1}$ and to the 
right of position $b+1$ must  lie 
in positions $[j+1,n]$.  But now since $c_{b+1} \geq n-j$, \emph{all} the 
letters in positions $[j+1,n]$ must be less than $w_{b+1}$.  In particular
the letter
$w_{\ell}$ defined in \eqref{third:bigger} must be less than $w_{b+1}.$
But now the letters $\{w_b, w_{b+1}, w_{\ell}, w_r\}$ form the pattern
$2431.$

Therefore we have shown that if $w$ avoids the patterns 
$3142$, $3214$, $2431$, and $3241$, then it satisfies the
condition  of \cref{prop:pandemic}.

In the other direction, suppose that $w$ contains the pattern 
$3214$.  Let $i<k<\ell<m$ denote the positions of the letters of this pattern.
Then $w_i>w_k$ implies that there exists some $j$ with $i \leq j < k$
such that $w_i< w_{i+1}< \cdots < w_j>w_{j+1}$.  We have 
$c_i\geq 2$ since $w_k$ and $w_{\ell}$ are both less than $w_i$.
And all of the $c_i$ letters less than $w_i$ which occur to the right of $w_i$
must occur in positions $[j+1,n]$.  Moreover, the $4$ in the pattern,
representing the letter $w_m>w_i$, occurs in a position in $[j+1,n]$.
Therefore $c_i<n-j$.  

Therefore there exists some $b$ with $i \leq b \leq j$ with 
$w_i \leq w_b <w_{b+1} < \cdots < w_{j}$ and $0<c_b<n-j$, and we choose
$b$ to be maximal with these properties.  We have that $w_b$ is 
greater than both $w_k$ and $w_{\ell}$.  But then it is impossible
for $c_{j+1} = c_{j+2} = \cdots = c_{j+c_b}=0$: by \cref{rem:order}
this means that the $c_b$ letters less than $w_b$ that appear to the right of $w_b$ must occur in 
increasing order in positions $j+1, j+2,\ldots, j+{c_{b}}$, but this is false
since $w_k$ and $w_{\ell}$ (the $2$ and $1$ of the pattern) occur in the 
wrong order.  Exactly the same argument holds if $w$ contains the pattern
$3241$.  

Nearly the same argument holds if $w$ contains the pattern $3142$.
Again let $i<k<\ell<m$ denote the positions of the letters of this pattern.
As before, since $w_i>w_k$, we can find $i\leq j<k$ such that 
$w_i<w_{i+1}<\cdots < w_j>w_{j+1}$, and we have $2\leq c_i<n-j$.
We then choose $b$ maximal with $i\leq b\leq j$ such that 
$0<c_b<n-j$.  But again using \cref{rem:order} we see it is impossible 
to have
$c_{j+1} = c_{j+2} = \cdots = c_{j+c_b}=0$ -- the $142$ in the pattern 
$3142$ (i.e. the letters in positions $k, \ell, m$) means that the letters less than $w_b$ do not occur in increasing order in consecutive positions. 

Finally, suppose that $w$ contains the pattern $2431$.  
Let $h<i<\ell<m$ denote the positions of the letters of this pattern.
Since $w_i>w_{\ell}$, there must be a descent position $j$
with $j<\ell$.  Let $j$ be minimal such that 
$w_h < w_{h+1} < \cdots < w_j>w_{j+1}$.  
Because $w_m<w_h$, we know that $c_h\geq 1$.  Moreover, the letters less 
than $w_h$ which are to the right of it must appear in positions $[j+1,n]$.
Because $w_{\ell}>w_h$, we know that $c_h<n-j$.
Therefore there exists some $b\geq h$ with 
$w_b<w_{b+1}<\cdots < w_j$ and $0<c_b<n-j$;  choose $b$ to be maximal
with this property.  But then by \cref{rem:order}, the $c_b$
letters less than $w_b$ which appear to the right of $w_b$ must appear in increasing order in positions $j+1,j+2,\ldots, j+c_b$. This is impossible,
since $w_{\ell}$ and $w_m$ lie weakly to the right of position $j+1$ but in the wrong order (since $w_{\ell}>w_m$).
This completes the proof.
\end{proof}

\begin{Def}\label{def:valid}
We say that a partition $\lambda$ is \emph{valid} for $n$
if $\lambda$ is contained in a 
	$\length(\lambda)\times(n-\length(\lambda))$ rectangle and is 
neither the empty partition nor the full 
	$\length(\lambda)\times(n-\length(\lambda))$ rectangle.
	Let $\Val(n) =\{\lambda \ \vert \ \lambda \text{ is valid for }n\}.$
\end{Def}
\begin{Rem}
	It is easy to show that $|\Val(n)|=2^{n-1}-(n-1)-1$. The elements of $\Val(n)$ are in bijection with Grassmannian permutations in $S_{n}$ that starts with 1. 
\end{Rem}

\begin{Def}\label{def:parseq}
For $1 \leq k \leq n-2$,  let $\ParSeq(n,k)$ denote the set of all sequences
of partitions $(\lambda^1,\ldots,\lambda^k)$ such that
each $\lambda^i$ is valid for $n$, and for all $1\leq i \leq k-1$:
\begin{itemize}
    \item if $\ell$ is the smallest part of $\lambda^i$, then
      the first $(n-\ell)$ parts of $\lambda^{i+1}$ are equal.
\end{itemize}
If $k=0$ then $\ParSeq(n,k)$ consists of one element, the 
empty sequence.
\end{Def}

\begin{Ex}
If $n=6$, then $((3),(2,2,2,1))$ and $((4,2),(1,1,1,1))$ lie in $\ParSeq(6,2)$ but 
$((3),(2,2,1,1))$ and $((4,2),(1,1,1)$ do not lie in $\ParSeq(6,2).$
\end{Ex}

\begin{Rem}
It follows from \cref{def:parseq} that if $(\lambda^1,\ldots,\lambda^k)\in \ParSeq(n,k)$, then for all
$1\leq i \leq k-1$:
\begin{itemize}
    \item the number of parts of $\lambda^i$ is less than the number of parts of $\lambda^{i+1}$
    \item every nonzero part of $\lambda^i$ is greater than every part of $\lambda^{i+1}.$ 
\end{itemize}
\end{Rem}

\begin{Rem}
Let $v(n)$ denote the vector $v(n):=({n-1 \choose 2}, {n-2 \choose 2},\ldots, {2 \choose 2}).$  Then it follows from \cref{def:parseq}
that $\mu:=v-\sum_{i=1}^k \lambda^i$ is a partition. 
(Here we think of each partition $\lambda^i$ as a vector in 
$\Z_{\geq 0}^{n-2}$ by adding extra parts equal to $0$ if necessary.) Note that the vector
$\mu$ appears in the steady steady probability formula
from \cref{thm:main0}.
\end{Rem}

\begin{Prop}\label{prop:bij}
The map $\Psi: \St(n,k) \to \ParSeq(n,k)$ 
(cf. \cref{def:bijpartition}) is well-defined and bijective.  
The inverse map $\Psi^{-1}: \ParSeq(n,k) \to \St(n,k)$ 
can be described as follows.
Let $(\lambda^1,\ldots,\lambda^k)\in \ParSeq(n,k)$, and 
let $(f_1,\ldots,f_k)$ be the sequence of first parts of 
$\lambda^1,\ldots, \lambda^k$, i.e. $f_i = \lambda^i_1.$
Then $((f_1^{n-f_1}-\lambda^1)+ (f_2^{n-f_2}-\lambda^2) + \cdots +
(f_k^{n-f_k} - \lambda^k))$ is the code of a permutation $w^{-1}$ of $w\in \St(n,k)$.  We define $\Psi^{-1}(\lambda^1,\ldots,\lambda^k)=w$.
\end{Prop}

\begin{proof}
We first show that when we apply the map  $\Psi$, we obtain a
vector that satisfies the properties of \cref{def:parseq}.  
Write $\code(w^{-1})=c= (c_1,\ldots,c_n)$ and let $a_1,\ldots,a_k$ denote the 
positions of the descents of $c$.  We have $\lambda^{j}=(n-a_j)^{a_j}-(0,\ldots,0,c_{a_{j-1}+1},\ldots,c_{a_j}).$ If we take the maximal $b$ such that $a_{j-1}<b\leq a_j$ and $a_j+c_b<n$, then by \cref{prop:pandemic}, we have $c_{a_j+1} = c_{a_j+2} = \cdots = c_{a_j+c_b} = 0$. Since $\lambda^{j+1}=(n-a_{j+1})^{a_{j+1}}-(0,\ldots,0,c_{a_{j}+1},\ldots,c_{a_{j+1}}),$ t
he first $a_j+c_b$ parts of $\lambda^{j+1}$ are equal. Let $l$ be the smallest part of $\lambda_j$, then $l=(n-a_j)-c_b$. So the first $(n-l)=a_j+c_b$ parts of $\lambda^{j+1}$ are equal.

Now we show that when we apply the map $\Psi^{-1}$ to $(\lambda_1,\cdots,\lambda_k)\in\ParSeq(n,k)$, we get an element in $\St(n,k)$. Let $c=((f_1^{n-f_1}-\lambda^1)+ (f_2^{n-f_2}-\lambda^2) + \cdots +
(f_k^{n-f_k} -\lambda^k))$. Since the smallest part of $\lambda^{i}$ is no greater than $f_i$, the first $n-f_i$ parts of $\lambda_{i+1}$ are equal. So the first $n-f_i$ components of the vector $(f_{i+1}^{n-f_{i+1}}-\lambda_{i+1})$ are zero. We write $c=(c_1,\ldots,c_n)$ as follows
\begin{align*}
    c_j=\begin{cases*}f_1-\lambda^{1}_{j}, \hspace{5mm}\text{if $1\leq j\leq n-f_1$}\\
                      f_i-\lambda^{i}_j, \hspace{5mm}\text{if $n-f_{i-1}< j\leq n-f_i$, for $2\leq i\leq k$}\\
                      0, \hspace{15mm}\text{if $n-f_{k}< j$},
    \end{cases*}
\end{align*}
where we regard $\lambda^{i}_j=0$ if $j$ is bigger that the length of $\lambda^{i}$. 

We claim that $c_{n-f_i}>c_{n-f_i+1}$. If $\lambda^{i}_{n-f_i}=0$ then we have $c_{n-f_i}=f_i>f_{i+1}\geq f_{i+1}-\lambda^{i+1}_{n-f_i+1}=c_{n-f_i+1}.$ If  $\lambda^{i}_{n-f_i}>0$ then $\lambda^{i}_{n-f_i}<f_i$ since $\lambda^{i}\in\Val(n)$. The first $n-\lambda^{i}_{n-f_i}$ parts of $\lambda^{i+1}$ are equal so $\lambda^{i+1}_{n-f_i+1}=f_{i+1}$. Thus we have $c_{n-f_i}=f_i-\lambda^{i}_{n-f_i}>0=c_{n-f_i+1}$.

We see that the descents of $c$ are at $n-f_1,\ldots,n-f_k$. Now take the maximal $b$ such that $n-f_{i-1}<b\leq n-f_i$ and $n-f_i+c_b<n$. Then $c_b=f_i-l$ where $l$ is the smallest part of $\lambda^{i}$. So the first $n-l$ parts of $\lambda^{i+1}$ are equal which implies $c_{n-f_{i}+1}=\cdots=c_{n-f_{i}+c_b}$ as $n-f_{i}+c_b=n-l$. We conclude that $c$ is the code of $w^{-1}$ for some $w\in \St(n,k)$ by \cref{prop:pandemic}.
\end{proof}

\begin{Ex}
If $c=\code(w^{-1}) = (0,3,1,1,0)$ then $c$ has two descents in positions
$a_1=2$ and $a_2=4$.  We also have $a_0=0$.
	Therefore $\lambda^1 = (5-2)^2-(c_1,c_2) = (3,3)-(0,3)=(3)$, and 
	$\lambda^2 = (5-4)^4-(0,0,c_3,c_4) = (1,1,1,1)-(0,0,1,1)=(1,1)$, and so 
	$\Psi(w) = ((3), (1,1)).$
If $\code(w^{-1}) = (0,2,2,1,0)$ then $\Psi(w) = ((2), (1,1,1)).$
\end{Ex}

\begin{Rem}
While we have not found any previous works studying
evil-avoiding permutations, we note that the sequence 
$\{e(n)\}$ of cardinalities of evil-avoiding permutations
in $S_n$ has several other combinatorial interpretations
	listed in \cite{Sloane}:
\begin{itemize}
    \item 
$e(n)$ counts permutations $\pi\in S_n$ for which the pairs $(i, \pi(i))$ with $i < \pi(i)$, considered as closed intervals $[i+1,\pi(i)]$, do not overlap; equivalently, for each $i\in [n]$ there is at most one $j \leq i$ with $\pi(j) > i.$ 
\item 
$e(n)$ is the number of permutations on $[n]$ with no subsequence abcd such that (i) bc are adjacent in position and (ii) $\max(a,c) < \min(b,d)$. For example, the 4 permutations of $[4]$ not counted by $a(4)$ are $1324, 1423, 2314, 2413$.
\item 
$e(n)$ is the number of \emph{rectangular permutations} on $[n]$, i.e. those permutations
which avoid the four patterns $2413$, $2431$, $4213$, $4231$, see \cite{CFF}.
\end{itemize}

It would be interesting to find a bijection between 
the $k$-Grassmannian permutations in $S_n$ (where we let $k$ vary) 
and any of the above sets of permutations.
\end{Rem}

\begin{Prop}\label{prop:recurrence}
For $k \geq 1$, we have that $$|\ParSeq(n,k)| = 2 |\ParSeq(n-1,k)| + \sum_{i=k+1}^{n-1} |\ParSeq(i,k-1)|.$$
\end{Prop}
\begin{proof}
To prove \cref{prop:recurrence}, we define two different injective maps 
$\Psi_i: \ParSeq(n-1,k)\to \ParSeq(n,k)$ for $i=1,2$ as well as a family
of injective maps $\Phi_{i,k,n}:\ParSeq(i,k-1) \to \ParSeq(n,k)$ for $k+1 \leq i \leq n-1.$  The statement then follows from the claim that every 
element of $\ParSeq(n,k)$ lies in the image of precisely one of these maps.

We define $\Psi_1: \ParSeq(n-1,k) \to \ParSeq(n,k)$ to be the map which
takes $(\lambda^1,\ldots,\lambda^k)$ to $(\mu^1,\ldots,\mu^k)$, where 
$\mu^i$ is obtained from $\lambda^i$ by duplicating its first part.
That is, if $\lambda^i = (\lambda_1^i, \lambda_2^i,\ldots, \lambda_r^i)$,
then $\mu^i = (\lambda_1^i, \lambda_1^i, \lambda_2^i, \ldots, \lambda_r^i)$.
So for example, $$\Psi_1((2),(1,1)) = ((2,2),(1,1,1)).$$

We define $\Psi_2: \ParSeq(n-1,k) \to \ParSeq(n,k)$ to be the map which
takes $(\lambda^1,\ldots,\lambda^k)$ to $(\mu^1,\ldots,\mu^k)$, where 
for $i \geq 2$, $\mu^i$ is obtained from $\lambda^i$ by duplicating its first part. For $i=1$, let $\lambda^i = (\lambda^i_1,\ldots, \lambda^i_r)$.  If all 
parts of $\lambda^i$ are equal, then we define $\mu^i = (\lambda_1^i+1,\lambda_1^i,\ldots, \lambda_r^i)$.  Otherwise, 
we define $\mu^i = (\lambda^i_1 +1,\lambda^i_2,\ldots, \lambda^i_r)$. 
So if $\lambda^1 = (\lambda_1^1, \lambda_2^1,\ldots, \lambda_r^1)$,
then $\mu^1 = (\lambda_1^1+1, \lambda_1^1, \lambda_2^1, \ldots, \lambda_r^1)$.
	So for example, $$\Psi_2((2),(1,1)) = ((3,2),(1,1,1))\qquad \text{ and } \qquad
	\Psi_2((3,2),(1,1,1)=(4,2),(1,1,1,1).$$

For $k+1 \leq i \leq n-1$, we define $\Phi_{i,k,n}: \ParSeq(i,k-1) \to \ParSeq(n,k)$ to be the map which takes $(\lambda^1,\ldots, \lambda^k)$
to $((i-1), \mu^1, \ldots, \mu^{k-1})$, where $\mu^j$ is obtained from 
$\lambda^j$ by duplicating the first part of $\lambda^j$ $n-i$ times.
That is, if $\lambda^j = (\lambda_1^j, \lambda_2^j,\ldots, \lambda_r^j)$,
then $\mu^j = (\underbrace{\lambda_1^j, \lambda_1^j, \ldots, \lambda_1^j}_\text{$n-i+1$}, \lambda_2^j, \ldots, \lambda_r^j)$.
So for example, we have that 
\begin{align*}\Phi_{4,2,5}((1)) &= ((3), (1,1)),\\
\Phi_{4,2,5}((2)) &= ((3),(2,2,)),\\
\Phi_{4,2,5}((1,1)) &= ((3), (1,1,1)),\\
	\Phi_{4,2,5}((2,1)) &= ((3),(2,2,1)), \text{ and }\\
	\Phi_{3,2,5}((1)) &= ((2),(1,1,1)).
\end{align*}
The above examples express the seven elements of $\ParSeq(5,2)$ as 
images of elements of $\ParSeq(4,2)$, $\ParSeq(4,1)$, and $\ParSeq(3,1)$.
\end{proof}

\begin{Cor}
Define the number $T(n,k)$ by 
\begin{equation}\label{eq:Tnk}
    T(n,k)=\sum\limits_{i=0}^{n-k-2} 2^{i}\binom{i+k-1}{k-1}\binom{n-2-i}{k}.
\end{equation}
(These numbers appear in  
	\cite[A331969]{Sloane}.)
Then we have $$|\St(n,k)|=|\ParSeq(n,k)| = T(n,k).$$
Equivalently, the number of evil-avoiding 
permutations $w$ in $S_{n-1}$ such that $w^{-1}$
has exactly $k$ descents is $T(n,k).$
\end{Cor}
For example, $T(4,0)=1$, $T(4,1)=4$, and $T(4,2)=1$. Meanwhile the number of evil-avoiding
permutations in $S_3$ whose inverse has exactly $0$, $1$, and $2$ descents, respectively, is
$1$, $4$, and $1$.

\begin{proof} 
The formula \eqref{eq:Tnk} is equivalent to the 
generating function given in \href{https://oeis.org/A331969}{A331969}.
By \cref{prop:recurrence} and \cref{prop:bij}, it suffices 
to prove that \begin{equation*}
     T(n,k)=2T(n-1,k)+\sum\limits_{i=k+1}^{n-1}T(i,k-1).
\end{equation*}

A simple computation shows that 
\begin{equation*}
	T(n,k)-2T(n-1,k) 
	=\binom{n-2}{k}+2\sum\limits_{i=0}^{n-k-3} 2^{i+2}\binom{i+k-1}{k-2}\binom{n-3-i}{k}, \text{ which
	implies that }
\end{equation*}
\begin{equation*}
     (T(n,k)-2T(n-1,k))-(T(n-1,k)-2T(n-2,k)) 
	=T(n-1,k-1).
\end{equation*}
Now the proof follows from the induction on $n$.
\end{proof}

\begin{proof}[Proof of \cref{prop:enumeration}] Since $e(n)=\sum\limits_{k=0}^{n-2}|\St(n,k)|,$ by \cref{prop:recurrence} we have 
\begin{align*}
    e(n)=\sum\limits_{k=0}^{n-2}|\St(n,k)|=1+\sum\limits_{k=1}^{n-2}(2|\St(n-1,k)|+\sum\limits_{j=k+1}^{n-1}|\St(j,k-1)|)\\
    =(3\sum\limits_{k=0}^{n-3}|\St(n-1,k)|-1)+\sum\limits_{k=1}^{n-3}\sum\limits_{j=k+1}^{n-2}|\St(j,k-1)|\\=4e(n-1)-(\sum\limits_{k=0}^{n-3}|\St(n-1,k)|+1-\sum\limits_{k=1}^{n-3}\sum\limits_{j=k+1}^{n-2}|\St(j,k-1)|)\\
    =4e(n-1)-(1+\sum\limits_{k=1}^{n-3}(2|\St(n-2,k)|)=4e(n-1)-2e(n-2).
    \end{align*}

\end{proof}

\section{Cantini's $\mathbf{z}$-deformation of steady state probabilities}\label{sec:Cantini}

In \cite{C}, Cantini associated to each $w\in \St(n)$ a \emph{deformed steady state probability} $\psi_w(\mathbf{z})=\psi_w(z_1,\ldots,z_n)$  which recovers the usual steady state probability $\psi_w$ for $w$ in the inhomogeneous TASEP when ``$\mathbf{z}=\infty,$'' or in other words, when one 
reads off the coefficient $\LC(\psi_w(\mathbf{z}))$ of the largest monomial in $\mathbf{z}.$
In this section we 
define some operators on  states and  polynomials, and then recall some results of 
\cite{C}.  

Recall that indices
of states are considered modulo $n$, so that for $w=(w_1,\ldots,w_n) \in \St(n)$, 
we have $w_{n+1} = w_1$.  
In this section we will also consider the indices for variables $\mathbf{z}=(z_1,\ldots,z_n)$ modulo $n$.
\begin{Def}\label{def:operator}
For $1 \leq i \leq n-1$, the simple transposition $s_i$ acts on the state 
$w=(w_1,\ldots,w_n)\in \St(n)$ by  
$$s_i (w_1,\ldots,w_i, w_{i+1}, \ldots ,w_n) = (w_1,\ldots,w_{i+1},w_i,\ldots, w_n).$$ 
We also define the action of $s_n$  by 
$$s_n (w_1,w_2,\ldots,w_{n-1},w_n) = (w_n,w_2,\ldots,w_{n-1},w_1).$$
And the \emph{shift operator} $\sigma$ acts on the state $w$ by 
$$\sigma(w_1,\ldots, w_n) = (w_2,\ldots,w_n,w_{n+1}).$$

If $f$ is a multivariate polynomial
in $z_1,\ldots, z_n$, then $s_i$ acts on $f(z_1,\ldots,z_n)$ by permuting the variables, i.e.
$$s_i f(z_1,\ldots, z_i, z_{i+1},\ldots,z_n) = f(z_1,\ldots, z_{i+1},z_i,\ldots, z_n).$$
And if $\mathbf{z} = (z_1,z_2,z_3,\ldots)$ is an ordered set of variables, we 
let 
$$\sigma (\mathbf{z}) = (z_2, z_3, z_4,\ldots).$$

Given a polynomial
	$f(z_1,\dots,z_n) = \sum_{\mathbf{m}} c_{\mathbf{m}} 
	  \mathbf{z}^{\mathbf{m}}$, where 
	  $\mathbf{z}^{\mathbf{m}}:=
	     z_1^{m_1} \dots z_n^{m_n}$, we say that 
	     the \emph{degree} of the term 
	     $c_{\mathbf{m}} z_1^{m_1} \dots z_n^{m_n}$ is $m_1+ \dots +m_n$
	     and the degree of $f$ is the maximum degree achieved
	     by its terms.
If $f(z,\dots,z)$ has a unique  term
	     $c_{\mathbf{m}} z_1^{m_1} \dots z_n^{m_n}$
	of maximum degree, we say it has a \emph{unique leading term}, and we
define the 
	 \emph{leading coefficient} $\LC(f)$ to be the coefficient 
	$c_{\mathbf{m}}$.
\end{Def}
    For example, 
    $$\LC \left(\prod\limits_{1\leq i<j\leq n}(x_i-y_{n+1-j})^{j-i-1}\prod\limits_{i=1}^{n}\big(\prod\limits_{j=1}^{i-1}(z_i-x_j)\prod\limits_{j=i+1}^{n}(z_i-y_{n+1-j})\big) \right) = 
    \prod\limits_{1\leq i<j\leq n}(x_i-y_{n+1-j})^{j-i-1}.$$
\begin{Prop}\cite[Equations (2), (24), (27), (28), (34), Theorem 18]{C} 
\label{CantiniProp}
Given 
 $w\in \St(n)$, we define $\psi_w(\mathbf{z})=\psi_w(z_1,\ldots,z_n)$ recursively, via:
\begin{align*}
    \psi_{(1,2,\ldots,n)}(\z)&=\prod\limits_{1\leq i<j\leq n}(x_i-y_{n+1-j})^{j-i-1}\prod\limits_{i=1}^{n}\big(\prod\limits_{j=1}^{i-1}(z_i-x_j)\prod\limits_{j=i+1}^{n}(z_i-y_{n+1-j})\big),\\
    \psi_{s_l w}(\z)&=\pi_l(w_l,w_{l+1};n)\psi_w(\z) \hspace{5mm} \text{if $w_l>w_{l+1}$,}
\end{align*}
where $\pi_l(\beta,\alpha;n)$ is the \emph{isobaric divided
difference operator} defined by
\begin{equation*}
    \pi_l(\beta,\alpha;n)G(\textbf{z})=\frac{(z_l-y_{n+1-\beta})(z_{l+1}-x_\alpha)}{x_\alpha-y_{n+1-\beta}}\frac{G(\textbf{z})-s_l G(\textbf{z})}{z_l-z_{l+1}}.
\end{equation*}

Then we have that 
for cyclically equivalent states $w$ and $\sigma(w)$ in $St(n)$,
\begin{equation}\label{eq:cyclic}
	\psi_{\sigma(w)}(z_1,\ldots,z_n)=\psi_{w}(\sigma(z_{1},\ldots,z_{n})),
\end{equation}
where the indices on the $\mathbf{z}$-variables
are considered modulo $n$.

	Then $\psi_w(\z)$ is a polynomial in $\x, \y, \z$ with a 
	unique leading term (with respect to the $z$-variables).
	Its leading coefficient 
	$\LC(\psi_w(\mathbf{z}))$ (see 
	\cref{def:operator})
	is  
\begin{equation}\label{eq:LC}
\LC(\psi_w(\mathbf{z}))=\psi_w.
\end{equation}
Because of \eqref{eq:LC}, we refer to 
$\psi_w(\mathbf{z})$ as the \emph{deformed steady state probability}.
\end{Prop}

\begin{Rem}
For any $w\in S_n$, the total degree of 
$\psi_w(\mathbf{z})$ is ${n \choose 3} + n(n-1) = \frac{n(n-1)(n+4)}{6}$.
\end{Rem}

In what follows we will often omit $n$ in the operator $\pi_l(w_l,w_{l+1};n)$ and just write $\pi_l(w_l,w_{l+1})$ when $n$ is clear from the context.  
\begin{Ex}
For $n=3$, we have
\begin{align*}
	\psi_{(1,2,3)}(\z)&=(x_1-y_1)(z_1-y_2)(z_1-y_1)(z_2-x_1)(z_2-y_1)(z_3-x_1)(z_3-x_2),\\
    \psi_{(3,2,1)}(\z)&=\pi_3(3,1)\psi_{(1,2,3)}(\z)=(z_1-x_1) (z_2-x_1) (z_2-y_1) (z_3-y_1)\times\\ &((x_1+x_2-y_1-y_2)z_3z_1+(x_1x_2-y_1y_2)(z_3+z_1)-x_1 x_2 y_1 - x_1 x_2 y_2 +  x_1 y_1 y_2 + x_2 y_1 y_2)  
   ),\\
    \psi_{(2,3,1)}(\z)&=\pi_1(3,2)\psi_{(3,2,1)}(\z)=(x_1-y_1)(z_3-y_2)(z_3-y_1)(z_1-x_1)(z_1-y_1)(z_2-x_1)(z_2-x_2),\\
    \psi_{(3,1,2)}(\z)&=\pi_2(2,1)\psi_{(3,2,1)}(\z)=(x_1-y_1)(z_2-y_2)(z_2-y_1)(z_3-x_1)(z_3-y_1)(z_1-x_1)(z_1-x_2),\\
    \psi_{(1,3,2)}(\z)&=\pi_1(3,1)\psi_{(3,1,2)}(\z)=(z_2-x_1) (z_3-x_1) (z_3-y_1) (z_1-y_1)\times\\ &((x_1+x_2-y_1-y_2)z_1z_2+(x_1x_2-y_1y_2)(z_1+z_2)-x_1 x_2 y_1 - x_1 x_2 y_2 +  x_1 y_1 y_2 + x_2 y_1 y_2)  
   ),\\
    \psi_{(2,1,3)}(\z)&=\pi_2(3,1)\psi_{(2,3,1)}(\z)=(z_3-x_1) (z_1-x_1) (z_1-y_1) (z_2-y_1)\times\\ &((x_1+x_2-y_1-y_2)z_2z_3+(x_1x_2-y_1y_2)(z_2+z_3)-x_1 x_2 y_1 - x_1 x_2 y_2 +  x_1 y_1 y_2 + x_2 y_1 y_2)  
   ).
\end{align*}
Taking the leading coefficients with respect to $\mathbf{z}$-variables 
recovers the probabilities from \cref{fig:S3}, namely
\begin{align*}
    &\psi_{(1,2,3)}=\psi_{(2,3,1)}=\psi_{(3,1,2)}=x_1-y_1\\
    &\psi_{(1,3,2)}=\psi_{(2,1,3)}=\psi_{(3,2,1)}=x_1+x_2-y_1-y_2.
\end{align*}
\end{Ex}

\section{The $\mathbf{z}$-deformation of our main result and $\mathbf{z}$-Schubert polynomials}\label{sec:z}

In this section we present a $\mathbf{z}$-generalization of our main 
result (see \cref{thm:maintechnical}):  for 
$w\in \St(n,k)$, the $\mathbf{z}$-deformed steady state probability 
$\psi_w(\mathbf{z})$ is equal to a ``trivial factor'' $\TF(w)$ times a product of \emph{$\mathbf{z}$-Schubert polynomials} (see \cref{def:zSchubert}) -- 
certain polynomials in $\mathbf{x}, \mathbf{y}, \mathbf{z}$
which reduce to double Schubert polynomials when 
$\mathbf{z}=\infty$.


We note that the $\mathbf{z}$-Schubert polynomials 
$\Sym_{\lambda}^n(\mathbf{z}; \mathbf{x}; \mathbf{y})\in \mathbb{R}[\mathbf{z}; \mathbf{x}; \mathbf{y}]$ (where 
$\mathbf{z}=\{z_1,\ldots,z_{n}\}$,
$\mathbf{x}=\{x_1,\ldots,x_{n-1}\}$,
$\mathbf{y}=\{y_1,\ldots,y_{n-1}\}$)
are not defined
for any permutation but rather depend on a choice of positive integer $n$ and a partition $\lambda\in \Val(n)$.

Given $\lambda\in\Val(n)$, the polynomial $\Sym_{\lambda}^n(\mathbf{z}; \mathbf{x}; \mathbf{y})$ has the property
that:
\begin{itemize}
    \item when $\mathbf{z}=\infty$, this polynomial reduces
    to  the 
double Schubert polynomial of the permutation with  code
$g_n(\lambda)$, see \cref{Lemma:leadingcoefficient};
\item when we additionally
set each $y_i=0$, we obtain a flagged Schur polynomial
of shape $\lambda$, see \cref{cor:flaggedSchur}.
\end{itemize}

\begin{Def}\label{Def:key operators new}
Given ordered variables $\mathbf{x} = (x_1,x_2,\ldots)$, if $I$ is a set of positive integers 
we let $\x_{\hat{I}}$ denote $\x_{\hat{I}} = \x \setminus \{x_i \ \vert \ i\in I\}$, where we keep 
 the order on variables inherited from $\x$. We abuse notation and use  $\x_{\hat{i}}$ to denote $\x_{\hat{\{i\}}}$. 
\end{Def}
Recall from \cref{def:operator}
that if 
$\z=(z_1,z_2,\ldots)$ then 
$\sigma(\z)$ denotes $(z_2, z_3,\ldots)$.

\begin{Ex}
For $f(\mathbf{z};\mathbf{x};\mathbf{y})=x_2z_1z_2+z_1+z_2+x_1+x_2+x_3$, we have
\begin{align*}
    & f(\sigma^{2}(\mathbf{z});\mathbf{x};\mathbf{y})=x_2z_3z_4+z_3+z_4+x_1+x_2+x_3\\
    & f(\mathbf{z};\mathbf{x}_{\hat{1}};\mathbf{y})=x_3z_1z_2+z_1+z_2+x_2+x_3+x_4\\
    & f(\mathbf{z};\mathbf{x}_{\{\widehat{2,3}\}};\mathbf{y})=x_4z_1z_2+z_1+z_2+x_1+x_4+x_5.
\end{align*}
\end{Ex}

   \begin{Def}\label{def:zSchubert}
    Given a positive integer $n$ and a partition $\lambda\in\Val(n)$ 
	   (see \cref{def:valid}), we define the \emph{$\mathbf{z}$-Schubert polynomial} 
	   $\Sym_{\lambda}^n(\mathbf{z}; \mathbf{x}; \mathbf{y})$ 
	   recursively as follows:
	   $$\Sym_{\emptyset}^{n-1}(\z;\x;\y)=1,$$ and for $\lambda \neq \emptyset$, we define 
	   $\Sym_{\lambda}^n(\mathbf{z}; \mathbf{x}; \mathbf{y})$  to be 
    \begin{equation}\label{eq: def zschubert}
	    \partial_{n-\lambda_1-\mul(\lambda)}\cdots\partial_1 \big( \Sym_{\lambda'}^{n-1}(\sigma^{\lambda_1-\lambda_2+1} (\mathbf{z}); \x_{\hat{1}}; \y) \prod\limits_{l=1}^{n-\mul(\lambda)}(x_1-y_l)\prod\limits_{i=1}^{(\lambda_1-\lambda_2+1)}\prod\limits_{\substack{m=2}}^{n-\lambda_1-\mul(\lambda)+1}(z_i-x_m) \big),
    \end{equation}
	   where the divided difference operators 
    \emph{act on the $x$-variables}, and   
 where $\lambda'$ is the partition obtained by deleting the first part of $\lambda$. 
	   (If $\lambda'$ is empty we 
	   regard $\lambda_2=0$.)  
\end{Def}

\begin{Rem}
Using induction on $n$, it is straightforward to show that the only $\mathbf{z}$-variables appearing in 
	$\Sym_{\lambda}^{n}(\z;\x;\y)$ 
	are $z_1,\ldots,z_{\lambda_1+\length(\lambda)}$.  It is also straightforward to show that 
	$\Sym_{\lambda}^{n}(\z;\x;\y)$ has a unique leading term with 
	respect to the $z$-variables so that 
	$\LC(\Sym_{\lambda}^n(\z; \x; \y))$ is well-defined.
\end{Rem}
\begin{Rem}
For positive integers $r$ and $s$, Cantini introduced a polynomial that he denoted $\Sym^{r,s}$ in \cite{C}. 
In our notation this is the same as $\Sym^{r+s-1}_{((r-1)^{s-1})}(\z;\x;\y)$.
\end{Rem}
\begin{Ex}We have
\begin{align*}
    \Sym^3_{(1)}(\z; \x;\y)&=\partial_1\big((x_1-y_1)(x_1-y_2)(z_1-x_2)(z_2-x_2)\big)\\
    &=\frac{(x_1-y_1)(x_1-y_2)(z_1-x_2)(z_2-x_2)-(x_2-y_1)(x_2-y_2)(z_1-x_1)(z_2-x_1)}{x_1-x_2} \\&=z_1 z_2 (x_1+x_2-y_1-y_2) - (z_1+z_2)(x_1 x_2+y_1 y_2) + (x_1 x_2 (y_1+y_2) - y_1 y_2(x_1+x_2)).
\end{align*}
Observe that $\LC(\Sym^3_{(1)}(\z; \x;\y)) = x_1+x_2-y_1-y_2,$ the double Schubert polynomial $\Sym_{(1,3,2)}(\x, \y)$.  It also
appears as a unnormalized steady state probability in \cref{fig:S3}.
\end{Ex}

Recall the definition of 
$g_n(\lambda)$ from
 \cref{def:g}.

\begin{Prop}\label{Lemma:leadingcoefficient}
Fix $n$ and choose a partition $\lambda\in \Val(n)$.  Then 
the leading coefficient 
	$\LC(\Sym_{\lambda}^n(\z; \x; \y))$ of $\Sym_{\lambda}^n(\z; \x; \y)$ is the double Schubert polynomial $\Sym_{c^{-1}(g_n(\lambda))}$ of 
the permutation $c^{-1}(g_n(\lambda))\in S_n$ whose code is $g_n(\lambda).$
\end{Prop}
We will prove \cref{Lemma:leadingcoefficient} in the
next section.

In \cref{thm:maintechnical} below we give a $\mathbf{z}$-deformation of our main result (\cref{thm:main1}); it says that for 
$w\in \St(n,k)$, the $\mathbf{z}$-deformed steady state probability 
$\psi_w(\mathbf{z})$ is equal to a ``trivial factor'' $\TF(w)$ times a product of $\mathbf{z}$-Schubert polynomials.  

Recall that we defined $\xyFact(w)$ in \cref{def:xy}. 
We also define
\begin{equation}\label{eq:xz}
\xzFact(w)=\prod_{i=1}^n \prod_{\substack{j \neq w_i \\ \min\{w_i, w_{i-1},w_{i-2},\ldots,j\}=j}} (z_i-x_j)
\end{equation} and 
\begin{equation}\label{eq:yz}
\yzFact(w)=\prod_{i=1}^n \prod_{\substack{j \neq n+1-w_i \\ \max\{w_i, w_{i+1},w_{i+2},\ldots,n+1-j\}=n+1-j}} (z_i-y_j).    
    \end{equation}
Finally we set 
\begin{equation}\label{eq:TF}
    \TF(w) = \xyFact(w) \xzFact(w) \yzFact(w).
\end{equation}

We associate to a partition sequence $(\lambda^1,\ldots,\lambda^k)$ 
a vector $(a_1,\ldots,a_k)$ as follows.

\begin{Def}\label{def:shiftingvector}
Let $(\lambda^1,\ldots,\lambda^k)$ be a sequence of partitions.  We 
denote the parts of $\lambda^j$ by $\lambda^j_1, \lambda^j_2$, etc.
We define $s((\lambda^1,\ldots,\lambda^k);n)=(a_1,\ldots,a_k)\in \Z^k$ by setting
$a_1=0$, and for each $2 \leq i \leq k$,
$$a_{i} = a_{i-1}+\lambda_1^{i-1} + \length(\lambda^{i-1}) - n.$$
If $w\in \St(n,k)$ such that $\Psi(w)=(\lambda^1,\ldots,\lambda^k)$, then abusing notation, we also refer to $s((\lambda^1,\ldots,\lambda^k);n)$ as 
$s(w)$.
\end{Def}
See \cref{table:3} for examples of $s(w)$.
The following theorem will be 
proved in \cref{section:mainproof}.

\begin{Th}\label{thm:maintechnical}
Let $w\in \St(n,k)$, and write  $\Psi(w)=(\lambda^{1},\ldots,\lambda^{k})$ and $s(w)=(a_1,\ldots,a_{k})$. Then we have 
\begin{equation}\label{eq:maintheorem}
    \psi_w(\z)=\TF(w)\prod_{i=1}^{k}\Sym_{\lambda^{i}}^{n}(\sigma^{a_i}(\z);\x;\y)
\end{equation}
where subscripts for $z$ variables are considered modular $n$.
\end{Th}

\section{Properties of $\mathbf{z}$-Schubert polynomials}\label{sec:zproperties}

\subsection{Double Schubert polynomials}
We review an algorithmic formula for computing double Schubert polynomials in terms of rc-graphs, based on \cite{BNB}.

\begin{Def}
Given a finite subset $D\subsetneq \{1,2,\ldots\} \times \{1,2,\ldots\}$ we define the \emph{weight} of $D$ to be 
\begin{equation*}
    \wt(D)(\x,\y)=\prod\limits_{(i,j)\in D}(x_i-y_j).
\end{equation*}
The \emph{initial diagram} of the permutation $w$ is
\begin{equation*}
    D_{in}(w) = \{(i,j) \ \vert \ 1 \leq j\leq c(w)_i\}.
\end{equation*}
\end{Def}

\begin{Def}
For a finite subset $D\subsetneq \{1,2,\ldots\} \times \{1,2,\ldots\}$, assume the following conditions are satisfied for some $i$ and $j$:
\begin{itemize}
    \item $(i,j)\in D, (i,j+1)\notin D$,
    \item $(i-m,j), (i-m,j+1) \notin D$ for some $0<m<i$,
    \item $(i-k,j), (i-k,j+1) \notin D$ for each $1\leq k<m$.
\end{itemize}
Then we define the \emph{ladder move} $L_{i,j}$ to be $L_{i,j}(D)=D\cup \{(i-m,j+1)\}\setminus \{(i,j)\}$.  We represent diagrams $D$ as above as collections of $+$'s, see \cref{fig:ladder move}.
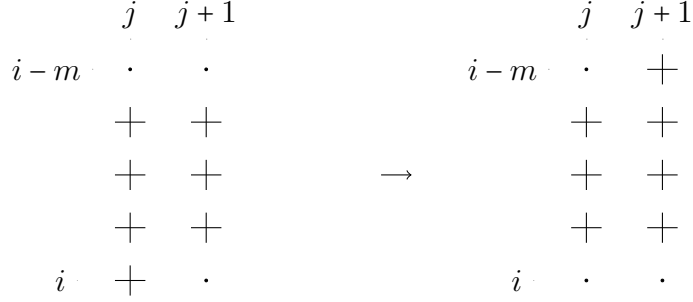
\begin{figure}\centering
\begin{tikzpicture}[scale=1]
\filldraw[black] (-0.7,0) circle (0.000001pt) node[anchor=east] {$i$};
\filldraw[black] (-0.5,2.8) circle (0.000001pt) node[anchor=east] {$i-m$};
\filldraw[black] (0,3.2) circle (0.000001pt) node[anchor=south] {$j$};
\filldraw[black] (1,3.2) circle (0.000001pt) node[anchor=south] {$j+1$};

\draw[-] (-0.2,0)--(0.2,0);
\draw[-] (0,-0.2)--(0,0.2);
\filldraw [black] (1,0) circle (0.5pt);

\draw[-] (-0.2,0.7)--(0.2,0.7);
\draw[-] (0,-0.2+0.7)--(0,0.2+0.7);

\draw[-] (1-0.2,0.7)--(1+0.2,0.7);
\draw[-] (1+0,-0.2+0.7)--(1+0,0.2+0.7);

\draw[-] (-0.2,1.4)--(0.2,1.4);
\draw[-] (0,-0.2+1.4)--(0,0.2+1.4);

\draw[-] (1-0.2,1.4)--(1+0.2,1.4);
\draw[-] (1,-0.2+1.4)--(1,0.2+1.4);

\draw[-] (-0.2,2.1)--(0.2,2.1);
\draw[-] (0,2.1-0.2)--(0,2.1+0.2);

\draw[-] (1-0.2,2.1)--(1.2,2.1);
\draw[-] (1,2.1-0.2)--(1,2.1+0.2);
\filldraw [black] (0,2.8) circle (0.5pt);
\filldraw [black] (1,2.8) circle (0.5pt);

\draw[->] (3.3,1.4)--(3.7,1.4);

\begin{scope}[shift={(6,0)}]
\filldraw[black] (-0.7,0) circle (0.000001pt) node[anchor=east] {$i$};
\filldraw[black] (-0.5,2.8) circle (0.000001pt) node[anchor=east] {$i-m$};
\filldraw[black] (0,3.2) circle (0.000001pt) node[anchor=south] {$j$};
\filldraw[black] (1,3.2) circle (0.000001pt) node[anchor=south] {$j+1$};

\filldraw [black] (0,0) circle (0.5pt);
\filldraw [black] (1,0) circle (0.5pt);

\draw[-] (-0.2,0.7)--(0.2,0.7);
\draw[-] (0,-0.2+0.7)--(0,0.2+0.7);

\draw[-] (1-0.2,0.7)--(1+0.2,0.7);
\draw[-] (1+0,-0.2+0.7)--(1+0,0.2+0.7);

\draw[-] (-0.2,1.4)--(0.2,1.4);
\draw[-] (0,-0.2+1.4)--(0,0.2+1.4);

\draw[-] (1-0.2,1.4)--(1+0.2,1.4);
\draw[-] (1,-0.2+1.4)--(1,0.2+1.4);

\draw[-] (-0.2,2.1)--(0.2,2.1);
\draw[-] (0,2.1-0.2)--(0,2.1+0.2);

\draw[-] (1-0.2,2.1)--(1.2,2.1);
\draw[-] (1,2.1-0.2)--(1,2.1+0.2);
\filldraw [black] (0,2.8) circle (0.5pt);

\draw[-] (1-0.2,2.8)--(1+0.2,2.8);
\draw[-] (1,2.8-0.2)--(1,3);
\end{scope}
\end{tikzpicture}
\caption{The figure shows the ladder move $L_{i,j}$.} \label{fig:ladder move}
\end{figure}
We also define $\mathcal{L}(D)$ to be the set of all $D'$ that can be obtained by applying ladder moves to $D$.
\end{Def}

Billey-Bergeron 
\cite[Theorem 3.7]{BNB} showed that $\mathcal{L}(D_{in}(w))$ gives the set of 
\emph{rc-graphs} for $w$.  
Combining this with the formula of Fomin-Kirillov \cite{FK}
produces the following formula for Schubert polynomials 
(see also \cite[Corollary 2.1.5]{KM}).
\begin{Th}\cite{FK}, \cite[Theorem 3.7 and Lemma 3.2]{BNB}\label{thm:rc}
Let $w$ be a permutation.
\begin{enumerate}
\item[(a)] We have $\Sym_w(\x,\y)=\sum\limits_{D'\in \mathcal{L}(D_{in}(w))}\wt(D')(\x,\y)$.
\item[(b)] The map sending $D'$ to its transpose $(D')^{t}$ is a bijection between $\mathcal{L}(D_{in}(w))$ and $\mathcal{L}(D_{in}(w^{-1}))$.
\end{enumerate}
\end{Th}
\begin{Ex}
The diagrams  $\mathcal{L}(D_{in}(1,4,2,3))$ are shown in \cref{fig:doubleschubert computation}. So we have
\begin{equation*}
    \Sym_{(1,4,2,3)}(\x,\y)=(x_2-y_1)(x_2-y_2)+(x_2-y_1)(x_1-y_3)+(x_1-y_2)(x_1-y_3).
\end{equation*}
\end{Ex}
\begin{figure}[h]\centering
\begin{tikzpicture}[scale=1]
\filldraw[black] (-0.5,2.1) circle (0.000001pt) node[anchor=east] {$2$};
\filldraw[black] (-0.5,2.8) circle (0.000001pt) node[anchor=east] {$1$};
\filldraw[black] (0,3.2) circle (0.000001pt) node[anchor=south] {$1$};
\filldraw[black] (1,3.2) circle (0.000001pt) node[anchor=south] {$2$};
\filldraw[black] (2,3.2) circle (0.000001pt) node[anchor=south] {$3$};

\filldraw [black] (0,2.8) circle (0.5pt);
\filldraw [black] (1,2.8) circle (0.5pt);
\filldraw [black] (2,2.8) circle (0.5pt);

\draw[-] (-0.2,2.1)--(0.2,2.1);
\draw[-] (0,2.1-0.2)--(0,2.1+0.2);
\draw[-] (1-0.2,2.1)--(1.2,2.1);
\draw[-] (1,2.1-0.2)--(1,2.1+0.2);

\filldraw [black] (2,2.1) circle (0.5pt);
\draw[->] (2.7,2.45)--(3.3,2.45);
\draw[->] (7.7,2.45)--(8.3,2.45);

\filldraw[black] (0.9,1.6) circle (0.000001pt) node[anchor=north] {$D_{in}((1,4,2,3))$};

\begin{scope}[shift={(5,0)}]
\filldraw[black] (-0.5,2.1) circle (0.000001pt) node[anchor=east] {$2$};
\filldraw[black] (-0.5,2.8) circle (0.000001pt) node[anchor=east] {$1$};
\filldraw[black] (0,3.2) circle (0.000001pt) node[anchor=south] {$1$};
\filldraw[black] (1,3.2) circle (0.000001pt) node[anchor=south] {$2$};
\filldraw[black] (2,3.2) circle (0.000001pt) node[anchor=south] {$3$};

\filldraw [black] (0,2.8) circle (0.5pt);
\filldraw [black] (1,2.8) circle (0.5pt);
\draw[-] (2-0.2,2.8)--(2+0.2,2.8);
\draw[-] (2,2.8-0.2)--(2,2.8+0.2);

\draw[-] (-0.2,2.1)--(0.2,2.1);
\draw[-] (0,2.1-0.2)--(0,2.1+0.2);
\filldraw [black] (1,2.1) circle (0.5pt);
\filldraw [black] (2,2.1) circle (0.5pt);
\end{scope}
\begin{scope}[shift={(10,0)}]
\filldraw[black] (-0.5,2.1) circle (0.000001pt) node[anchor=east] {$2$};
\filldraw[black] (-0.5,2.8) circle (0.000001pt) node[anchor=east] {$1$};
\filldraw[black] (0,3.2) circle (0.000001pt) node[anchor=south] {$1$};
\filldraw[black] (1,3.2) circle (0.000001pt) node[anchor=south] {$2$};
\filldraw[black] (2,3.2) circle (0.000001pt) node[anchor=south] {$3$};

\filldraw [black] (0,2.8) circle (0.5pt);
\draw[-] (1-0.2,2.8)--(1+0.2,2.8);
\draw[-] (1,2.8-0.2)--(1,2.8+0.2);
\draw[-] (2-0.2,2.8)--(2+0.2,2.8);
\draw[-] (2,2.8-0.2)--(2,2.8+0.2);

\filldraw [black] (0,2.1) circle (0.5pt);
\filldraw [black] (1,2.1) circle (0.5pt);
\filldraw [black] (2,2.1) circle (0.5pt);
\end{scope}
\end{tikzpicture}
\caption{The three diagrams in $\mathcal{L}(D_{in}(1,4,2,3))$.}\label{fig:doubleschubert computation}
\end{figure}

Note that rc-graphs are in bijection with \emph{reduced pipe dreams}
(replace each $+$ with a crossing of two wires and each empty position with 
a pair of ``elbows'').

We will need the following result about linear factors of certain 
double Schubert polynomials.  If  $y_i=0$ for all $i$, 
\cref{prop:doubleschubert property} follows from \cite[Corollary 3.11]{BNB}.

\begin{Prop}\label{prop:doubleschubert property}
Let $w\in S_n$ be a permutation, and let 
$c(w)=(c_1,\ldots,c_n)$ and $c(w^{-1})=(\tilde c_1,\ldots, \tilde c_n)$ denote the codes of $w$ and $w^{-1}.$
Suppose there is some $l\geq 0$ such that 
$\tilde c_{m}=0$ for all $m>l$, and let $w'$ be the permutation
with code $c(w')=(l, c_1,\ldots, c_n)$ (whose existence follows from \cref{lem:code0}).  Then 
\begin{equation*}
    \Sym_{w'}(\x,\y)=\Sym_{w}(\x_{\hat{1}},\y)\prod\limits_{k=1}^{l}(x_1-y_k).
\end{equation*}
\end{Prop}
\begin{Rem}\label{rem:codesequence}
The condition that 
$\tilde c_{m}=0$ for all $m>l$ is equivalent to the condition that $w$ has 
an increasing subsequence of the form $l+1, l+2,\ldots, n$.
\end{Rem}
\begin{proof}
We first claim that no $D' \in \mathcal{L}(D_{in}(w))$ contains 
a $+$ in a column greater than $l$, i.e. there is no $(i,j)\in D'$ with $j>l$. 
If there were such a $D'$, then we'd have $(j,i) \in (D')^{t}$.  But $D_{in}(w^{-1})$ does not have an element whose $x$-coordinate is bigger than $l$,
and ladder moves never
increase the $x$-coordinates of the $+$'s involved. 

By the claim, $D_{in}(w)$ does not contain 
a $+$ in a column greater than $l$, which implies the same is true for $D_{in}(w')$
and hence for $\mathcal{L}(D_{in}(w'))$.

Now we define a map $f: \mathcal{L}(D_{in}(w)) \rightarrow \mathcal{L}(D_{in}(w'))$ by $$f(D')= \{(i+1,j) \ \vert \ (i,j)\in D'\}\cup \{(1,1),(1,2),\ldots,(1,l)\}.$$
This map is clearly injective, and is well defined since $f(D_{in}(w))=D_{in}(w')$. 
We claim that $f$ is surjective.  Assume not. Then we can find $D_1, D_2 \in \mathcal{L}(D_{in}(w')$ such that $D_2=L_{i,j}(D_1)$, and $D_1$ is in the image of $f$ but $D_2$ is not. 

Clearly the only way that there would be a viable ladder move $L_{ij}$ on $D_1$
which does not have a counterpart for $f^{-1}(D_1)$ is if $L_{ij}$ adds a $+$
in row $1$, necessarily in some column $j>l$ (since the first component of $c(w')$ is $l$).
But we know that no diagram in $\mathcal{L}(D_{in}(w'))$ can have a $+$ in a column
greater than $l$.  
Therefore the map $f$ is surjective and hence bijective. 

We conclude that 
\begin{align*}
    \Sym_{w'}(\x,\y)&=\sum\limits_{D''\in \mathcal{L}(D_{in}(w'))}\wt(D'')(\x,\y)=\sum\limits_{D'\in \mathcal{L}(D_{in}(w))}\wt(f(D'))(\x,\y)\\
    &=\sum\limits_{D'\in \mathcal{L}(D_{in}(w))}\prod\limits_{k=1}^{l}(x_1-y_k) \ \wt(D')(\x_{\hat{1}},\y)\\
    &=\prod\limits_{k=1}^{l}(x_1-y_k)\sum\limits_{D'\in \mathcal{L}(D_{in}(w))}\wt(D')(\x_{\hat{1}},\y)=\prod\limits_{k=1}^{l}(x_1-y_k)\Sym_{w}(\x_{\hat{1}},\y).
\end{align*}
\end{proof}

\subsection{The proof of \cref{Lemma:leadingcoefficient}.}\label{sec:proof}

The following lemma is easy to verify.

\begin{Lemma}\label{lem:code}
Let $w\in S_n$ with code $c(w)=(v_1,\ldots,v_n)$. Then $w_i>w_{i+1}$ if and only if $v_i>v_{i+1}$.  In this case we have 
\begin{equation} \label{eq:divided0}
\partial_i \Sym_w = \Sym_{w s_i},
    \end{equation}
    and $c(w s_i)=(v'_1,\ldots,v'_n)$,
    where $v'_i = v_{i+1}$, $v'_{i+1}=v_i - 1$, and $v'_j=v_j$ for $j\notin \{i,i+1\}$.

If we iterate \eqref{eq:divided0}, we find that if  
$v_1 - i \geq v_{i+1}$ for all $1 \leq i \leq r$, then
\begin{equation}\label{eq:divided}
\partial_r \partial_{r-1} \cdots \partial_1 
\Sym_{c^{-1}(v_1,\ldots,v_n)}(\x,\y) = 
\Sym_{c^{-1}(v_2,v_3,\ldots,v_{r+1}, v_1-r, v_{r+2},\ldots, v_n)}(\x,\y).
\end{equation}
\end{Lemma}


We are now ready to prove \cref{Lemma:leadingcoefficient}.

\begin{proof}[Proof of \cref{Lemma:leadingcoefficient}] We use induction on the number of parts of $\lambda$. If $\lambda=(\ell)$ has one part, then the definition
of $\mathbf{z}$-Schubert polynomial plus \eqref{eq:divided} gives
\begin{align*}
    \LC(\Sym_{(\ell)}^{n}(\z;\x;\y))&=\partial_{n-\ell-1}\cdots \partial_1(\prod\limits_{l=1}^{n-1}(x_1-y_l))
    =\partial_{n-\ell-1}\cdots\partial_2 \partial_1(\Sym_{c^{-1}(n-1,0,\ldots,0)}(\x,\y))\\
    &=\partial_{n-\ell-1}\cdots\partial_2 \bigl(\Sym_{c^{-1}(0,n-2,0,\ldots,0)}(\x,\y)\bigr)=\Sym_{c^{-1}(0,\ldots,0,\ell,0,\ldots,0)}(\x,\y)\\
    &=\Sym_{c^{-1}(g_n((\ell)))}(\x,\y).
\end{align*}

Now let $\lambda=(\lambda_1,\ldots,\lambda_k)$ with $k>1$ and assume the statement holds for partitions with at most $k-1$ parts. Setting $\lambda'=(\lambda_2,\ldots,\lambda_k)$, we have
\begin{align*}
     \LC(\Sym_{\lambda}^{n}(\z;\x;\y))&=\partial_{n-\lambda_1-\mul(\lambda)}\cdots\partial_1 \bigl( \LC(\Sym_{\lambda'}^{n-1}(\sigma^{\lambda_1-\lambda_2+1}\z;\x_{\hat{1}};\y))\prod\limits_{l=1}^{n-\mul(\lambda)}(x_1-y_l)\bigr)\\&=
     \partial_{n-\lambda_1-\mul(\lambda)}\cdots\partial_1 \bigl( \Sym_{c^{-1}(g_{n-1}(\lambda'))}(\x_{\hat{1}},\y)\prod\limits_{l=1}^{n-\mul(\lambda)}(x_1-y_l)\bigr).
\end{align*}
Let $g_{n-1}(\lambda')=(v_1,\ldots,v_{n-1})$ and let $w=c^{-1}(v_1,\ldots,v_{n-1})\in S_{n-1}$. 
From \cref{rem:g} we have $v_i=\lambda'_1$ for $n-1-\lambda'_1 \leq i \leq n-\mul(\lambda')-\lambda'_1$. 
This plus \cref{def:g} 
implies that 
\begin{align*}
   w_{n-\mul(\lambda')-\lambda'_1}&=n-\mul(\lambda')\\
    &\ \vdots\\
     w_{n-2-\lambda'_1}&=n-2\\
      w_{n-1-\lambda'_1}&=n-1.
\end{align*}

It follows from \cref{rem:codesequence} that if we write $c(w^{-1})=(\tilde c_1,\ldots, \tilde c_{n-1})$, then $\tilde c_m =0$ for all $m > n-\mul(\lambda')-1$. 
Since $n-\mul(\lambda)\geq n-\mul(\lambda')-1$, 
we have $\tilde c_m=0$ for all $m > n-\mul(\lambda)$, and \cref{prop:doubleschubert property} implies that 
\begin{align*}
    \Sym_{c^{-1}(v_1,\ldots,v_{n-1})}(\x_{\hat{1}},\y)\prod\limits_{l=1}^{n-\mul(\lambda)}(x_1-y_l)=\Sym_{c^{-1}(n-\mul(\lambda),v_1,\ldots,v_{n-1})}(\x,\y).
\end{align*}
Therefore we have that 
\begin{equation*}
     \LC(\Sym_{\lambda}^{n}(\z;\x;\y))=
     \partial_{n-\lambda_1-\mul(\lambda)}\cdots\partial_1 \bigl(\Sym_{c^{-1}(n-\mul(\lambda),v_1,\ldots,v_{n-1})}(\x,\y)\bigr).
\end{equation*}
We will apply \eqref{eq:divided}, but need to first check that 
$n-\mul(\lambda_1) \geq v_i+i$ for $1 \leq i \leq n-\lambda_1-\mul(\lambda).$  To see this,
note that since $(v_1,\ldots,v_{n-1}) = g_{n-1}(\lambda')$, each $v_i \leq \lambda_1$.
But then $v_i+i \leq \lambda_1 + n-\lambda_1-\mul(\lambda) = n-\mul(\lambda),$ as desired.
Applying \eqref{eq:divided}  gives 
\begin{equation*}
     \LC(\Sym_{\lambda}^{n}(\z;\x;\y))=\Sym_{c^{-1}(v_1,\ldots,v_{n-\lambda_1-\mul(\lambda)},\lambda_1,v_{n-\lambda_1-\mul(\lambda)+1},\ldots,v_{n-1})}(\x,\y)=\Sym_{c^{-1}(g_n(\lambda))}(\x,\y).
\end{equation*}
\end{proof}

\subsection{$\mathbf{z}$-Schubert polynomials and flagged Schur polynomials}

We relate $\mathbf{z}$-Schubert polynomials $\Sym_{\lambda}^{n}(\z;\x;\y)$ with flagged Schur polynomials (\cref{cor:flaggedSchur}). We start with \cref{lem:easy} which is immediate from the definition of $g_n(\lambda)$.

\begin{Lemma}\label{lem:easy}
Let $n$ be a positive integer and $\lambda$ be a partition of length at most $n-2$. 
\begin{itemize}
    \item The first and last components of $g_n(\lambda)$ are always zero.
\item If we order the entries of $g_n(\lambda)$ in weakly decreasing order, we obtain
$\lambda$.  
\end{itemize}
\end{Lemma}

\begin{Lemma}\label{lem:flag}
Let $\lambda$ be a partition of length $\leq (n-2)$, and write 
$\lambda = (\mu_1^{k_1}, \ldots, \mu_l^{k_l})$.  Let $w$ be the permutation whose code is 
$g_n(\lambda)$.  Then $w$ is vexillary, and if we write the flag of $w$ and $w^{-1}$
(cf. \cref{def:flag})
as $(f_1^{k_1},\ldots,f_l^{k_l})$ (where $f_1\leq \cdots \leq f_l$) and $(g_1^{b_1},\ldots,g_l^{b_l})$, respectively, we have that 
$$\{(f_1,g_l), (f_2, g_{l-1}),\ldots, (f_l, g_1)\} = \{(n-\mu_1, n-k_1), (n-\mu_2, n-k_1-k_2),\ldots, (n-\mu_l, n-k_1-\cdots-k_l)\}.$$
\end{Lemma}
\begin{proof}
Recall that the \emph{essential set} of a diagram $D(w)$ is the set of southeast corners of the diagram, that is,
$$\Ess(w) = \{(i,j) \in D(w) \ \vert \ (i+1,j), (i,j+1), (i+1,j+1) \notin D(w)\}.$$
Writing $\lambda=(\mu_1^{k_1},\ldots,\mu_l^{k_l})$ as 
in \cref{def:g}, we see that the essential set of $w$ is 
$$\Ess(w) =\{(n-\mu_1, n-k_1), (n-\mu_2, n-k_1-k_2), (n-\mu_3, n-k_1-k_2-k_3),\ldots\}.$$  This essential set lies
on a piecewise linear curve always oriented SW-NE, so by \cite[Proposition 2.2.8]{Manivel}, the permutation 
$w$ must be vexillary.

The claim about the flags follows immediately from the 
definition of the flag in terms of the code, or alternatively from 
\cite[Exercise 2.2.11]{Manivel}.
\end{proof}

\begin{Cor}\label{cor:flaggedSchur}
Write $\lambda = (\mu_1^{k_1}, \mu_2^{k_2},\ldots, \mu_l^{k_l}).$  Then 
$\LC(\Sym_{\lambda}^{n}(\z;\x;\y))$ equals the flagged double Schur polynomial
$s_{\lambda}(\underbrace{X_{n-\mu_1}-Y_{n-k_1}}_\text{$k_1$}, \underbrace{X_{n-\mu_2}-Y_{n-k_1-k_2}}_\text{$k_2$},\ldots,
\underbrace{X_{n-\mu_l}-Y_{n-k_1-\cdots -k_l}}_\text{$k_l$}).$

And when we set $y_i=0$ for all $i$, $\LC(\Sym_{\lambda}^{n}(\z;\x;\y))$ equals the flagged Schur polynomial
$s_{\lambda}(X_{n-\lambda_1}, X_{n-\lambda_2},\ldots).$
\end{Cor}
\begin{proof}
This follows from \cref{lem:flag}, \cref{Lemma:leadingcoefficient} and \cref{th:flag}. 
\end{proof}


\begin{Ex}
The permutation $13542$ has  code $(0,1,2,1,0)=g_5((2,1,1))$.  Therefore
for $\lambda=(2,1,1)$, we have that $$\LC(\Sym_{(2,1,1)}^{5}(\z;\x;\y))=\Sym_{13542}(\mathbf{x},\mathbf{y})=s_{(2,1,1)}(X_3-Y_4,X_4-Y_2,X_4-Y_2).$$
\end{Ex}

\section{Proof of \cref{thm:maintechnical}}\label{section:mainproof}

In this section we prove \cref{thm:maintechnical}, which in turn implies 
\cref{thm:main1} and \cref{thm:main0}. Our strategy will be to prove 
\cref{thm:maintechnical} first in the case of $w\in \St(n,1)$,
and then use induction on $k$ to prove it for $w\in \St(n,k)$.

We note that in this section, the divided difference operators
\emph{act on the $\mathbf{z}$-variables.} 
For brevity we will often denote the $\mathbf{z}$-Schubert polynomial $\Sym_{\lambda}^{n}(\sigma^a(\z);\x;\y)$ with shifted $\mathbf{z}$-variables by $\Sym_{\lambda}^{n}(\sigma^a(\z))$.  As in the previous section, the subscripts for $\mathbf{z}$-variables are considered modulo $n$.

\begin{Def}\label{Def: IGP Consturction}
For a partition $\lambda\in\Val(n)$, we identify it with the lattice path 
$L(\lambda;n)$ cutting out the Young diagram that takes unit steps south and east from the upper right corner $(\lambda_1, n-\lambda_1)$ to the lower left corner $(0,0)$ of the rectangle. Label the vertical steps from the top to bottom with numbers $1$ through $n-\lambda_1$. Then label the horizontal steps from the right to the left with numbers $n-\lambda_1+1$ through $n$. We define $w(\lambda;n)$ to be the permutation of length $n$ obtained by reading off the numbers along the lattice path. 
\end{Def} 
See \cref{fig00} for an example.  Clearly $w(\lambda;n)\in \St(n,1)$.  As we will see in \cref{Prop:IGP property and trivial factor} (1),  $\Psi(w(\lambda; n)) = (\lambda)$.

\begin{figure}[h]\centering
\begin{tikzpicture}[scale=0.8]

\draw[-] (0,0) --(6,0);
\draw[-] (0,1) --(6,1);
\draw[-] (0,2) --(6,2);
\draw[-] (0,3) --(6,3);
\draw[-] (0,4) --(6,4);
\draw[-] (0,5) --(6,5);
\draw[-] (0,6) --(6,6);
\draw[-] (6,0) --(6,7);
\draw[-] (0,0) --(0,7);
\draw[-] (1,0) --(1,7);
\draw[-] (2,0) --(2,7);
\draw[-] (3,0) --(3,7);
\draw[-] (4,0) --(4,7);
\draw[-] (5,0) --(5,7);
\draw[-] (0,7) --(6,7);

\draw[-,blue,thick] (6,7) --(6,5);
\draw[-,blue,thick] (6,5) --(4,5);
\draw[-,blue,thick] (4,5) --(4,3);
\draw[-,blue,thick] (4,3) --(2,3);
\draw[-,blue,thick] (2,3) --(2,1);
\draw[-,blue,thick] (2,1) --(0,1);
\draw[-,blue,thick] (0,1) --(0,0);

\filldraw[black] (6,6.5) circle (0.000001pt) node[anchor=west] {1};
\filldraw[black] (6,5.5) circle (0.000001pt) node[anchor=west] {2};
\filldraw[black] (4,4.5) circle (0.000001pt) node[anchor=west] {3};
\filldraw[black] (4,3.5) circle (0.000001pt) node[anchor=west] {4};
\filldraw[black] (2,2.5) circle (0.000001pt) node[anchor=west] {5};
\filldraw[black] (2,1.5) circle (0.000001pt) node[anchor=west] {6};
\filldraw[black] (0,0.5) circle (0.000001pt) node[anchor=west] {7};

\filldraw[red] (5.5,5) circle (0.000001pt) node[anchor=south] {8};
\filldraw[red] (4.5,5) circle (0.000001pt) node[anchor=south] {9};
\filldraw[red] (3.5,3) circle (0.000001pt) node[anchor=south] {10};
\filldraw[red] (2.5,3) circle (0.000001pt) node[anchor=south] {11};
\filldraw[red] (1.5,1) circle (0.000001pt) node[anchor=south] {12};
\filldraw[red] (0.5,1) circle (0.000001pt) node[anchor=south] {13};

\end{tikzpicture}
\caption{For
$\lambda=(6,6,4,4,2,2)\in \Val(13)$, we can read off 
    $w(\lambda;13)=(1,2,8,9,3,4,10,11,5,6,12,13,7)$ from the lattice path $L(\lambda;13)$ cutting out the Young
    diagram.} \label{fig00}
\end{figure}
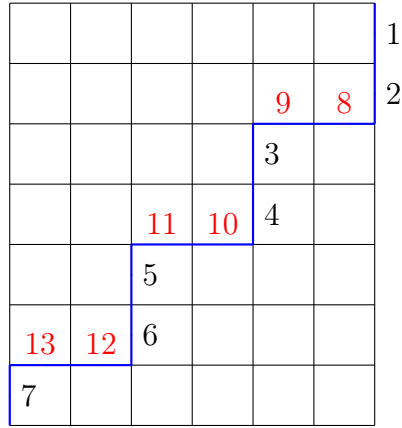

Our first goal is to analyze the trivial factor $\TF(w)$ for $w=w(\lambda;n)$ (\cref{Cor: IGP trivial factor}).  This will help us prove \cref{thm:maintechnical} in the case of $w\in \St(n,1)$.  We start by refining the quantities
introduced in \eqref{eq:xy}, \eqref{eq:xz}, and \eqref{eq:yz}.

\begin{Def}
Fix a positive integer $n$ and choose $w\in\St(n)$.  We define 
\begin{align*}
   \xyFact(w;i)&=\prod_{\substack{i+1<k \leq n \\ i \to i+1 \to k}}(x_1-y_{n+1-k})\cdots(x_i-y_{n+1-k}) \qquad \text{ for }1 \leq i \leq n-2\\
\xzFact(w;i)&= \prod_{\substack{j \neq w_i \\ \min\{w_i, w_{i-1},w_{i-2},\ldots,j\}=j}} (z_i-x_j) \quad \text{ for }1 \leq i \leq n\\
\yzFact(w;i)&= \prod_{\substack{j \neq n+1-w_i \\ \max\{w_i, w_{i+1},w_{i+2},\ldots,n+1-j\}=n+1-j}} (z_i-y_j) \quad \text{ for }1 \leq i \leq n.    
\end{align*}
    \end{Def}
    Clearly we have that 
    $$\xyFact(w)=\prod\limits_{i=1}^{n-2}\xyFact(w;i), \xzFact(w)=\prod\limits_{i=1}^{n}\xzFact(w;i), \yzFact(w)=\prod\limits_{i=1}^n \yzFact(w;i).$$

\begin{Prop}\label{Prop:IGP property and trivial factor}
Let $\lambda\in\Val(n)$ and $w=(w_1,\ldots,w_n)=w(\lambda;n)$.  
Recall that $w\in \St(n,1)$.  The following statements hold:
\begin{enumerate}
    \item We have that $\Psi(w(\lambda; n)) = (\lambda)$.  Equivalently, $c(w^{-1})=\lambda_1^{n-\lambda_1}-\lambda$,
    where we regard the vectors on the right-hand side as vectors of length $n$ by adding trailing $0$'s. 
    \item Suppose that $w_i$ lies on a vertical step of $L(\lambda;n)$. Let $A$ be the set of numbers on the horizontal steps below $w_i$ and $B$ be the set of numbers on the vertical steps that are on the {\bf same} vertical line as $w_i$ and below $w_i$. We have
    \begin{align*}
        \xzFact(w;i)&=\prod\limits_{k=1}^{w_i-1}(z_i-x_k)\\
        \yzFact(w;i)&=\prod\limits_{k\in A \cup B}(z_i-y_{n+1-k}).
    \end{align*}
    \item Suppose $w_i$ lies on a horizontal step of $L(\lambda;n)$. Let $C$ be the set of numbers on the vertical steps above $w_i$ and $D$ be the set of numbers on the horizontal steps that are on the {\bf same} horizontal line of $w_i$ and to the right of $w_i$. We have
    \begin{align*}
        \xzFact(w;i)&=\prod\limits_{k\in C\cup D}(z_i-x_k)\\
        \yzFact(w;i)&=\prod\limits_{k=w_i+1}^{n}(z_i-y_{n+1-k}).
    \end{align*}
\end{enumerate}
\end{Prop}
\begin{proof}
(1) The numbers $n-\lambda_1+1$ through $n$ appear in increasing order in $w$ so $c(w^{-1})$ vanishes after the $(n-\lambda_1)$st component. 
For $1 \leq k \leq n-\lambda_1$, let $w_i$ be the letter on the $k$th vertical step of $L(\lambda;n)$.  Then there are $\lambda_1-\lambda_k$ numbers bigger than $w_i$ in $w_1,\ldots,w_{i-1}$ (where we regard $\lambda_k=0$ if $k>\length(\lambda)$. Thus the $k$th component of $c(w^{-1})$ is $\lambda_1-\lambda_k$, and 
$c(w^{-1})$ has a unique descent in position $n-\lambda_1$.  The fact that 
$\Psi(w(\lambda;n))=\lambda$ now follows from the definition of $\Psi$.

(2) The numbers $1$ through $w_i$ appear in increasing order in $w$ so we have $\xzFact(w;i)=\prod\limits_{k=1}^{w_i-1}(z_i-x_k)$. 
To compute $\yzFact(w;i)$, we need to find all letters $\ell$ which are maximum
among $\{w_i, w_{i+1},\ldots, \ell\}$ and for each one we pick up a factor of 
$(z_i-y_{n+1-\ell})$.  Clearly these letters are precisely the ones in $A\cup B$.


(3) The proof is similar to part (2). 
\end{proof}

\begin{Ex}\label{ex: IGP}
Let $\lambda=(6,6,4,4,2,2)\in \Val(13)$ as in \cref{fig00}.  We have
\begin{align*}
    w&=w(\lambda;13)=(1,2,8,9,3,4,10,11,5,6,12,13,7)\\
c(w^{-1})&=(0,0,2,2,4,4,6,0,0,0,0,0,0)\\
\lambda_1^{n-\lambda_1}-\lambda&=(0,0,2,2,4,4,6).
    \end{align*}
    For $w_5=3$, we have
    \begin{align*}
\xzFact(w;5)&=(z_5-x_1)(z_5-x_2)\\
\yzFact(w;5)&=(z_5-y_{10})(z_5-y_1)(z_5-y_2)(z_5-y_3)(z_5-y_4).
    \end{align*}
    
\end{Ex}

\begin{Prop}\label{Cor: IGP trivial factor}Let $\lambda\in\Val(n)$ with $\mul(\lambda)=b$. Write $\lambda=((\lambda_1)^{b},\tilde{\lambda})$ for some $\tilde{\lambda}$. Let $w=w(((\lambda_1)^{b},\tilde{\lambda});n)$ and $w'=w(((\lambda_1)^{b-1},\lambda_1-1,\tilde{\lambda});n)$. Then:
\begin{enumerate}
    \item If $b>1$, we have 
	    \begin{enumerate}
		    \item[(a)] $w'=s_b w$ and $w_b<w_{b+1}$, 
		    \item[(b)] as well as \begin{align*}
       \TF(s_b w)&=M_1(x_b-y_{\lambda_1})\prod\limits_{i=1}^{\mul((\lambda_1-1,\tilde{\lambda}))-1}(z_{b+1}-y_{n+1-b-i})\\
    \TF(w)&=M_1(z_b-y_{\lambda_1})(z_{b+1}-x_b)\prod\limits_{i=1}^{b-1}(z_i-y_{n+1-b})\prod\limits_{i=b+2}^{b+\lambda_1-\tilde{\lambda}_1}(z_i-x_{n+1-\lambda_1}).
\end{align*}
for some rational expression $M_1$ which
is symmetric in  $z_b$ and $z_{b+1}$.
	    \end{enumerate}
   \item If $b=1$, we have 
	   \begin{enumerate}
		   \item[(a)] $\sigma(w')=s_1 w$
		and $w_1<w_{2}$,
	\item[(b)] as well as  \begin{align*}
       \TF(s_1 w)&=M_2(x_1-y_{\lambda_1})\prod\limits_{i=1}^{n-\lambda_1}(z_1-x_i)\prod\limits_{i=1}^{\mul((\lambda_1-1,\tilde{\lambda}))-1}(z_2-y_{n-i})\\
    \TF(w)&=M_2(z_1-y_{\lambda_1})(z_2-x_1)\prod\limits_{i=3}^{1+\lambda_1-\tilde{\lambda}_1}(z_i-x_{n+1-\lambda_1}).
\end{align*}
for some rational expression $M_2$ which is symmetric in  $z_1$ and $z_{2}$.
	   \end{enumerate}
\end{enumerate}
\end{Prop}
\begin{proof}
	Parts (a) are straightforward from \cref{Def: IGP Consturction}.

(1b) For $1\leq i\leq b-1$, $w_b=b$ is on the same vertical line with $w_i$ in $L(((\lambda_1)^{b},\tilde{\lambda});n)$ but $w'_{b+1}=b$ is not on the same vertical line with $w_i$ in $L(((\lambda_1)^{b-1},\lambda_1-1,\tilde{\lambda});n)$. By \cref{Prop:IGP property and trivial factor} (2), whenever
  $1\leq i\leq b-1$,
we have
\begin{align}\label{propeq1}
    \xzFact(w;i)\yzFact(w;i)&=\xzFact(w';i)\yzFact(w';i)(z_i-y_{n+1-b}).
\end{align}
For $b+2\leq i\leq b+\lambda_1-\tilde{\lambda}_1$, $w_{b+1}=n+1-\lambda_1$ is on the same horizontal line with $w_i$ in $L(((\lambda_1)^{b},\tilde{\lambda});n)$ but $w'_{b}=n+1-\lambda_1$ is not on the same horizontal line with $w_i$ in $L(((\lambda_1)^{b-1},\lambda_1-1,\tilde{\lambda});n)$. By \cref{Prop:IGP property and trivial factor} (3), whenever  $1\leq i\leq b-1$, we have
\begin{align}\label{propeq2}
\xzFact(w;i) \yzFact(w;i)(z_i-x_{n+1-\lambda_1})&=\xzFact(w';i)\yzFact(w';i). 
\end{align}
Note that $w'_{b+1},w'_{b+2},\ldots,w'_{b+\mul(\tilde{\lambda})+1}$ are on the same vertical line if and only if $\lambda_1-1=\tilde{\lambda}_1$. By \cref{Prop:IGP property and trivial factor} (2), we have
\begin{align}\label{propeq3}
    \xzFact(w;b)\yzFact(w;b)&=\prod\limits_{k=1}^{b-1}(z_b-x_k)\prod\limits_{k=1}^{\lambda_1}(z_b-y_k)\\
    \xzFact(w';b+1)\yzFact(w';b+1)&=\prod\limits_{k=1}^{b-1}(z_{b+1}-x_k)\prod\limits_{k=1}^{\lambda_1-1}(z_{b+1}-y_k)\prod\limits_{k=1}^{\mul((\lambda_1-1,\tilde{\lambda}))-1}(z_{b+1}-y_{n+1-b-k}).\nonumber
\end{align}
And by \cref{Prop:IGP property and trivial factor} (3), we have
\begin{align}\label{propeq4}
    \xzFact(w;b+1)\yzFact(w;b+1)&=\prod\limits_{k=1}^{b}(z_{b+1}-x_k)\prod\limits_{k=1}^{\lambda_1-1}(z_{b+1}-y_k)\\
    \xzFact(w';b)\yzFact(w';b)&=\prod\limits_{k=1}^{b-1}(z_{b}-x_k)\prod\limits_{k=1}^{\lambda_1-1}(z_{b}-y_k).\nonumber
\end{align}

For $w$, we have $b-1\rightarrow b\rightarrow w_{b+1}=n+1-\lambda_1$ while for $w'$ we have $b\rightarrow b+1\rightarrow w'_{b}=n+1-\lambda_1$. So we conclude
\begin{equation}\label{propeq5}
    \frac{\xyFact(w)}{\xyFact(w')}= \frac{\prod\limits_{k=1}^{b-1}(x_k-y_{\lambda_1})}{\prod\limits_{k=1}^{b}(x_k-y_{\lambda_1})}=\frac{1}{x_b-y_{\lambda_1}}.
\end{equation}
Combining \eqref{propeq1}, \eqref{propeq2}, \eqref{propeq3},\eqref{propeq4} and \eqref{propeq5} proves the argument.

(2b) The proof is similar to part (1b).
\end{proof}

\begin{Prop}\label{PROP:IGP TRUE}
\cref{thm:maintechnical} is true for $w\in \St(n,1)$.
\end{Prop}
\begin{proof}
We use induction on $|\lambda|$ for $\Psi(w)=(\lambda)$. The base case $|\lambda|=0$ corresponds to the identity permutation in $\St(n,0)$. Take any $w\in \St(n,1)$ such that $\Psi(w)=(\lambda)$. Let $b=\mul(\lambda)$ and write $\lambda=((\lambda_1)^b,\tilde{\lambda})$ for some $\tilde{\lambda}$. Denoting $w'=w(((\lambda_1)^{b-1},\lambda_1-1,\tilde{\lambda});n)$, by the induction hypothesis we have $$\psi_{w'}(\z)=\TF(w')\Sym_{((\lambda_1)^{b-1},\lambda_1-1,\tilde{\lambda})}^{n}(\z).$$ 

We first prove the induction step when $b>1$. By \cref{Cor: IGP trivial factor} (1) and \cref{CantiniProp}, we have
\begin{equation}\label{ssss}
    \psi_w(\z)=\pi_b(w_{b+1},w_b)\psi_{w'}(\z)=\frac{(z_b-y_{\lambda_1})(z_{b+1}-x_b)}{x_b-y_{\lambda_1}}\partial_b \left(\TF(w')\Sym_{((\lambda_1)^{b-1},\lambda_1-1,\tilde{\lambda})}^{n}(\z)\right).
\end{equation}
By \cref{Cor: IGP trivial factor} (3) we can write \begin{align*}
       \TF(w')&=M_1(x_b-y_{\lambda_1})\prod\limits_{i=1}^{\mul((\lambda_1-1,\tilde{\lambda}))-1}(z_{b+1}-y_{n+1-b-i})\\
    \TF(w)&=M_1(z_b-y_{\lambda_1})(z_{b+1}-x_b)\prod\limits_{i=1}^{b-1}(z_i-y_{n+1-b})\prod\limits_{i=b+2}^{b+\lambda_1-\tilde{\lambda}_1}(z_i-x_{n+1-\lambda_1})
\end{align*}
for some $M_1$ that is symmetric in variables $z_b$ and $z_{b+1}$. Plugging this
	into \eqref{ssss} gives
\begin{equation*}
    \psi_w(\z)=M_1(z_b-y_{\lambda_1})(z_{b+1}-x_b)\partial_b\left(\Sym_{((\lambda_1)^{b-1},\lambda_1-1,\tilde{\lambda})}^{n}(\z)\prod\limits_{i=1}^{\mul((\lambda_1-1,\tilde{\lambda}))-1}(z_{b+1}-y_{n+1-b-i})\right)
\end{equation*}
	By \cref{prop 116} (whose proof is deferred to the appendix) we have 
\begin{align*}
    \psi_w(\z)&=M_1(z_b-y_{\lambda_1})(z_{b+1}-x_b)\left(\Sym_{((\lambda_1)^{b},\tilde{\lambda})}^{n}(\z)\prod\limits_{i=1}^{b-1}(z_i-y_{n+1-b})\prod\limits_{i=b+2}^{b+\lambda_1-\tilde{\lambda}_1}(z_i-x_{n+1-\lambda_1})\right)\\&=\TF(w)\Sym_{\lambda}^{n}(\z).
\end{align*}

Now consider the case $b=1$. By \cref{Cor: IGP trivial factor} (2) and \cref{CantiniProp}, we have
\begin{equation}\label{ssss1}
	\psi_w(\z)=\pi_1(w_{2},w_1)\psi_{\sigma(w')}(\z)=\frac{(z_1-y_{\lambda_1})(z_{2}-x_1)}{x_1-y_{\lambda_1}}\partial_1 \left(\TF(\sigma(w'))\Sym_{(\lambda_1-1,\tilde{\lambda})}^{n}(\sigma(\z))\right).
\end{equation}
By \cref{Cor: IGP trivial factor} (4) we can write \begin{align*}
	\TF( \sigma(w'))&=M_2(x_1-y_{\lambda_1})\prod\limits_{i=1}^{n-\lambda_1}(z_1-x_i)\prod\limits_{i=1}^{\mul((\lambda_1-1,\tilde{\lambda}))-1}(z_2-y_{n-i})\\
    \TF(w)&=M_2(z_1-y_{\lambda_1})(z_2-x_1)\prod\limits_{i=3}^{1+\lambda_1-\tilde{\lambda}_1}(z_i-x_{n+1-\lambda_1}).
\end{align*}
for some $M_2$ that is symmetric in variables $z_1$ and $z_{2}$. Plugging 
	this into \eqref{ssss1} gives
\begin{equation*}
    \psi_w(\z)=M_2(z_1-y_{\lambda_1})(z_{2}-x_1)\partial_1\left(\Sym_{(\lambda_1-1,\tilde{\lambda})}^{n}(\sigma(\z))\prod\limits_{i=1}^{n-\lambda_1}(z_1-x_i)\prod\limits_{i=1}^{\mul((\lambda_1-1,\tilde{\lambda}))-1}(z_2-y_{n-i})\right)
\end{equation*}
	By \cref{prop 115} (also from the appendix) we have 
\begin{align*}
    \psi_w(\z)&=M_2(z_1-y_{\lambda_1})(z_{2}-x_1)\left(\Sym_{(\lambda_1,\tilde{\lambda})}^{n}(\z)\prod\limits_{i=3}^{1+\lambda_1-\tilde{\lambda}_1}(z_i-x_{n+1-\lambda_1})\right)\\&=\TF(w)\Sym_{\lambda}^{n}(\z).
\end{align*}
\end{proof}
\begin{Def}\label{def:direct sum operator}
If $\pi\in S_m$ 
and $\sigma\in S_p$, the 
\emph{direct sum} $\pi \oplus \sigma\in S_{m+p}$ is the permutation defined by 
$(\pi\oplus \sigma)(i) = 
\begin{cases}
\pi(i) & \text{ if } 1 \leq i \leq m\\
\sigma(i-m)+m &\text{ if } m+1 \leq i \leq m+n.
\end{cases}$

For example, 
$(3,2,1) \oplus (3,1,2,5,4) = 
(3,2,1, 6,4,5,8,7).$
\end{Def}
The following lemma is easy to verify.
\begin{Lemma}\label{lem:onepartition}
Given $\lambda\in\Val(n)$, let $$u=(u_1,\ldots,u_n):=
	\sigma^{\length(\lambda)+\lambda_1-n}(w(\lambda;n)) .$$
Let $\bar{w}(\lambda;n):=(u_1,\ldots, u_{n-\lambda_{\last}})$.  Then $\bar{w}(\lambda;n) \in S_{n-\lambda_{\last}}$ and $u=\bar{w}(\lambda;n)\oplus \id_{\lambda_{\last}},$ where $\id_m$ is the identity permutation on $m$ letters.
\end{Lemma}

\begin{Ex}
Let $n=5$ and $\lambda=(2,2)\in \Val(5)$, so that $\lambda_{\last}=2$.
Then $w(\lambda;5)=(1,2,4,5,3)$ and $\length(\lambda)+\lambda_1-n=-1$.
	We have $u=\sigma^{-1}(1,2,4,5,3)=(3,1,2,4,5)$, so 
$\bar{w}(\lambda;5)=(3,1,2)$.  We have $u=(3,1,2) \oplus \id_2$.
\end{Ex}

\begin{Prop}\label{Prop: SS Decomposition}
Let $w\in \St(n,k)$ for $k\geq 2$ with $\Psi(w)=(\lambda^1,\lambda^2,\ldots,\lambda^k)$ and 
	$s(w)=(a_1,\ldots,a_k)$. Then we can write $\sigma^{a_2}(w)=\bar{w}(\lambda^1;n)\oplus w'$ for some $w'$. If we let $\wdown:=\id_{n-\lambda^1_{\last}}\oplus w'$, 
then $\wdown\in \St(n,k-1)$ and $\Psi(\wdown)=(\lambda^2,\ldots,\lambda^k)$.
\end{Prop}

\begin{Ex}
Let $w=(1,2,5,4,3)\in \St(5,2)$.  Then $\Psi(w)=(\lambda^1,\lambda^2)=
((2,2),(1,1,1))$ and $s(w)=(a_1,a_2)=(0,-1).$  From the previous example,
	$\bar{w}(\lambda;5)=(3,1,2)$.  We have $\sigma^{a_2}(w)=(3,1,2,5,4) = 
(3,1,2) \oplus w'$ where $w'=(2,1)$.  And we have 
$\wdown = \id_{n-\lambda^1_{\last}} \oplus w' = (1,2,3) \oplus (2,1) = (1,2,3,5,4)\in \St(5,1)$ with $\Psi(\wdown) = (\lambda^2) = ((1,1,1)).$
\end{Ex}

\begin{proof}
Note that $a_2=\length(\lambda^1)+\lambda^1_1-n.$
We have $c(w^{-1})=((\lambda^1_1)^{n-\lambda^1_1}-\lambda^1)+((\lambda^2_1)^{n-\lambda^2_1}-\lambda^2)+\cdots+((\lambda^k_1)^{n-\lambda^k_1}-\lambda^k)$ and by the definition of $\ParSeq(n,k)$ we know that the first $(n-\lambda^1_{\last})$ parts of $((\lambda^i_1)^{n-\lambda^i_1}-\lambda^i)$ are zero for $2\leq i\leq k$. So the first $(n-\lambda^1_{\last})$ parts of $c(w^{-1})$ equal $((\lambda^1_1)^{n-\lambda^1_1}-\lambda^1)$ which is the same as $c((w(\lambda^1;n))^{-1})$ by \cref{Prop:IGP property and trivial factor} (1). So the positions of the numbers $1$ through  $(n-\lambda^1_{\last})$ are the same in $w$ and $w(\lambda^1;n)$. Since taking the first $(n-\lambda^1_{\last})$ parts of 
	$\sigma^{\length(\lambda^1)+\lambda^1_1-n}(w(\lambda^1;n))$ 
	gives $\bar{w}(\lambda;n)$, we conclude taking the first $(n-\lambda^1_{\last})$ parts of 
	$\sigma^{\length(\lambda^1)+\lambda^1_1-n}(w)$
	also gives $\bar{w}(\lambda;n)$.

Note that the $i$-th component of $c(w^{-1})$ counts 
	the number of $w_j>i$ for $1\leq j\leq w^{-1}(i)$. Thus $c(w^{-1})$ and 
	$c(\sigma^{\length(\lambda^1)+\lambda^1_1-n}(w)$
	coincide after the 
	$(n-\lambda^1_{\last})$-th component as taking the numbers smaller than $(n-\lambda^1_{\last})$ to the front does not affect the code of its inverse after $(n-\lambda^1_{\last})$-th component. Thus 
	$c((\id_{n-\lambda^1_{\last}}\oplus w')^{-1})$ coincides with $c(w^{-1})$ after the $(n-\lambda^1_{\last})$ component and the first $(n-\lambda^1_{\last})$ parts are zero as the permutation starts with $\id_{n-\lambda^1_{\last}}$. We conclude    $c((\id_{n-\lambda^1_{\last}}\oplus w')^{-1})=((\lambda^2_1)^{n-\lambda^2_1}-\lambda^2)+\cdots+((\lambda^k_1)^{n-\lambda^k_1}-\lambda^k)$ and we are done.
\end{proof}
\begin{Lemma}\label{trivial factor fraction equality}
Let $u$ and $u'$ be permutations in $S_n$ and $w$ and $w'$ be permutations in $S_{m}$. We have
\begin{equation*}
    \frac{\TF(u'\oplus w)}{\TF(u\oplus w)}=\frac{\TF(u'\oplus w')}{\TF(u\oplus w')}.
\end{equation*}
\end{Lemma}
\begin{proof}
It is enough to show the following three equations
\begin{align}
     \label{tftrivialfactor1}\frac{\xyFact(u'\oplus w)}{\xyFact(u\oplus w)}=\frac{\xyFact(u'\oplus w')}{\xyFact(u\oplus w')}\\
     \label{tftrivialfactor2} \frac{\xzFact(u'\oplus w)}{\xzFact(u\oplus w)}=\frac{\xzFact(u'\oplus w')}{\xzFact(u\oplus w')}\\
      \label{tftrivialfactor3} \frac{\yzFact(u'\oplus w)}{\yzFact(u\oplus w)}=\frac{\yzFact(u'\oplus w')}{\yzFact(u\oplus w')}.
\end{align}

For $1\leq i\leq n-1$ we have $\xyFact(u\oplus w;i)=\xyFact(u\oplus w';i)$ and $\xyFact(u'\oplus w;i)=\xyFact(u'\oplus w';i)$. For $n+1\leq i\leq n+m-1$ we have $\xyFact(u\oplus w;i)=\xyFact(u'\oplus w;i)$ and $\xyFact(u\oplus w';i)=\xyFact(u'\oplus w';i)$. And for $i=n$ we have $\xyFact(u\oplus w;i)=\xyFact(u'\oplus w;i)$ and $\xyFact(u\oplus w';i)=\xyFact(u'\oplus w';i)$. So the first equation \eqref{tftrivialfactor1} follows.

For $1\leq i\leq n$ we have $\xzFact(u\oplus w;i)=\xzFact(u\oplus w';i)$ and $\xzFact(u'\oplus w;i)=\xzFact(u'\oplus w';i)$. Now consider the case $n+1\leq i\leq n+m$. For $1\leq j\leq n$, $\xzFact(u\oplus w;i)$ has a factor $(z_i-x_j)$ for $\xzFact(u\oplus w';i)$ has a factor $(z_i-x_j)$. And the same is true for $\xzFact(u'\oplus w;i)$ and $\xzFact(u'\oplus w';i)$. For $n+1\leq j\leq n+m$ $\xzFact(u\oplus w;i)$ has a factor $(z_i-x_j)$ for $\xzFact(u'\oplus w;i)$ has a factor $(z_i-x_j)$. And the same is true for $\xzFact(u\oplus w';i)$ and $\xzFact(u'\oplus w';i)$. So the second equation \eqref{tftrivialfactor2} follows.

The proof for \eqref{tftrivialfactor3} is similar to the proof for \eqref{tftrivialfactor2}.
\end{proof}

\begin{Prop}\cite[Theorem 20]{C} Let $u\in S_n$ and $w\in S_m$. Then we can write 
\begin{equation*}
    \psi_{u\oplus w}(\z)=\psi^{1}_{u}(\z)\psi^{2}_{w}(\z)
\end{equation*}
where $\psi^{1}_{u}(\z)$ (respectively $\psi^{2}_{w}(\z)$) depends only on $u$ (respectively $w$).\footnote{The result stated in 
\cite[Theorem 20]{C} concerns the skew sum of the permutations
$u$ and $w$, not the direct sum; however, the direct sum of $u$ and $w$ is a cyclic rotation of the skew sum of $w$ and $u$, so the result we've stated follows.}
\end{Prop}
\begin{Cor}\label{psi fraction equality }
Let $u$ and $u'$ be permutations in $S_n$ and $w$ and $w'$ be permutations in $S_{m}$. We have
\begin{equation*}
    \frac{\psi_{u'\oplus w}(\z)}{\psi_{u\oplus w}(\z)}=\frac{\psi_{u'\oplus w'}(\z)}{\psi_{u\oplus w'}(\z)}.
\end{equation*}
\end{Cor}

\textit{Proof of \cref{thm:maintechnical}.} We prove \cref{thm:maintechnical} for $w\in\St(n,k)$ using induction on $k$. 
\cref{thm:maintechnical} holds for $k=1$ by \cref{PROP:IGP TRUE}. Suppose the theorem holds for all elements in $\St(n,k-1)$.  

We now consider
 $w\in \St(n,k)$.
Let $\Psi(w)=(\lambda^1,\lambda^2,\ldots,\lambda^k)$ and $s(w)=(a_1,a_2,\ldots,a_k)$. By \cref{Prop: SS Decomposition}, we have $\sigma^{a_2}(w)=\bar{w}(\lambda^1;n)\oplus w'$ for some $w'$.  Moreover, if we set $\wdown := \id_{n-\lambda^1_{\last}}\oplus w'$, we have that 
$\wdown\in \St(n,k-1)$ and $\Psi(\wdown)=(\lambda^2,\ldots,\lambda^k)$. Using \cref{def:shiftingvector} it is easy to see that  $s(\wdown)=(0,a_3-a_2,\ldots,a_k-a_2)$. By the induction hypothesis we have
\begin{equation}\label{eqinmain1}
    \psi_{\id_{n-\lambda^1_{\last}}\oplus w'}(\z)=\TF(\id_{n-\lambda^1_{\last}}\oplus w')\prod\limits_{i=2}^{k}\Sym_{\lambda^i}^n(\sigma^{a_i-a_2}(\z)).
\end{equation}
By \cref{lem:onepartition}, and using the fact that 
$a_2=\length(\lambda^1)+\lambda_1^1-n$, we have that 
$\bar{w}(\lambda^1;n)\oplus \id_{\lambda^1_{\last}}=
\sigma^{a_2}(w(\lambda^1;n))$, where $w(\lambda^1;n)\in \St(n,1)$.  Therefore
the induction hypothesis and \eqref{eq:cyclic} implies that
\begin{equation}\label{eqinmain2}
    \psi_{\bar{w}(\lambda^1;n)\oplus \id_{\lambda^1_{\last}}}(\z)=\TF(\bar{w}(\lambda^1;n)\oplus \id_{\lambda^1_{\last}})\ \Sym_{\lambda^1}^{n}(\sigma^{a_1-a_2}(\z)).
\end{equation}

By \cref{psi fraction equality } we have
\begin{equation*}
    \psi_{\bar{w}(\lambda^1;n)\oplus w'}(\z)=
    \frac{\psi_{\bar{w}(\lambda^1;n)\oplus \id_{\lambda^1_{\last}}}(\z) \ \psi_{\id_{\lambda^1_{n-\last}}\oplus w'}(\z)}{\psi_{\id_{n-\lambda^1_{\last}}\oplus \id_{\lambda^1_{\last}}}(\z)}
\end{equation*}
Plugging this into \eqref{eqinmain1} and \eqref{eqinmain2} and using \cref{trivial factor fraction equality} gives 
\begin{align*}
    \psi_{\bar{w}(\lambda^1;n)\oplus w'}(\z)&=\frac{\psi_{\bar{w}(\lambda^1;n)\oplus \id_{\lambda^1_{\last}}}(\z)\ \psi_{\id_{\lambda^1_{n-\last}}\oplus w'}(\z)}{\psi_{\id_{n-\lambda^1_{\last}}\oplus \id_{\lambda^1_{\last}}}(\z)}\\&=\frac{\TF(\bar{w}(\lambda^1;n)\oplus \id_{\lambda^1_{\last}})\TF(\id_{\lambda^1_{n-\last}}\oplus w')}{\TF(\id_{n-\lambda^1_{\last}}\oplus \id_{\lambda^1_{\last}})}\prod\limits_{i=1}^{k}\Sym_{\lambda^i}^n(\sigma^{a_i-a_2}(\z))\\&=\TF(\bar{w}(\lambda^1;n)\oplus w')\prod\limits_{i=1}^{k}\Sym_{\lambda^i}^n(\sigma^{a_i-a_2}(\z)).
\end{align*}
Now using \eqref{eq:cyclic} and cyclically shifting $\mathbf{z}$-variables completes the proof.

\section{Multiline queues and steady state probabilities}\label{sec:MLQ}

When each $y_i=0$, there is a combinatorial 
formula for the steady state probabilities of the inhomogeneous TASEP 
in terms of 
the \emph{multiline queues} of Ferrari and Martin \cite{FM}.
This result (\cref{thm:MLQ}) was conjectured in \cite{AL} and proved in \cite{AM}.
In this section we show 
in \cref{Thm: bijection}
that if $w^{-1}$ is a Grassmann permutation, then
the multiline queues of type $w$ are 
in bijection with certain collections of nonintersecting paths,
which in turn are in bijection with semistandard tableaux.

We note that in the case that $w=w_0, s_1 w_0$ or $s_2 w_0$,
multiline queues were related to nonintersecting paths in 
\cite[Section 3]{AasLinusson18}.   And in the case
that all particles have types $0, 1, 2$, nonintersecting paths
were used to give explicit determinantal
formulas for steady state probabilities in \cite[Theorem 2.6]{Man17}.
Multiline queues were also connected to tableaux via nonintersecting paths
in \cite{AGS}, though the weights on queues  there were different from 
ours.

\begin{Def} \label{def:MLQ}
	Fix positive integers $L$ and $n$. A \emph{multiline queue}  $Q$
	of content $\mathbf{m} = (m_1,\ldots, m_L)$ is an $L\times n$ array in which there are $m_1+ \cdots + m_i$ balls in row $i$ for $1\leq i \leq L$.
	We label the rows $1,\ldots, L$ from top to bottom.  

	Given such an array,
	there is a \emph{bully path labeling}
	procedure which assigns a label to each 
	ball.  We start by assigning each ball in row $1$ the label $1$.
	Then we consider the leftmost of these balls (call it $b$)
	and construct its
	\emph{$1$-bully path}: match $b$ to the first unmatched
	ball $b'$ in row $2$ 
	which is encountered when one looks directly below $b$ then travels
	right in row $2$ (wrapping around if necessary); then match $b'$ 
	to the first ball $b''$ in row $3$ which is encountered when 
	one looks below $b'$ then travels right as before.
	Continuing, we obtain the $1$-bully path for $b$; we label
	all matched balls by $1$.  
	We then construct $1$-bully paths for each other ball in row $1$
	(considering them from left to right).
	We then assign the label $2$ to 
	all unmatched balls
	in row $2$ (considered left to right), 
	and similarly construct their \emph{$2$-bully paths}.
	Continuing in this way gives the \emph{bully path labeling}
	to all balls in $Q$, see \cref{fig0}.

	\begin{figure}[h]\centering
\begin{tikzpicture}[scale=0.5]
\draw[-] (-1,1) --(-2,1);
\draw[-] (-2.2,1) --(-3.2,1);
\draw[-] (-3.4,1) --(-4.4,1);
\draw[-] (-4.6,1) --(-5.6,1);
\draw[-] (-5.8,1) --(-6.8,1);
\draw[-] (-7,1) --(-8,1);
\draw[-] (-8.2,1) --(-9.2,1);
\draw[-] (-9.4,1) --(-10.4,1);

\begin{scope}[shift={(0,1.5)}]
\draw[-] (-1,1) --(-2,1);
\draw[-] (-2.2,1) --(-3.2,1);
\draw[-] (-3.4,1) --(-4.4,1);
\draw[-] (-4.6,1) --(-5.6,1);
\draw[-] (-5.8,1) --(-6.8,1);
\draw[-] (-7,1) --(-8,1);
\draw[-] (-8.2,1) --(-9.2,1);
\draw[-] (-9.4,1) --(-10.4,1);

\draw (-3.9,1.7) circle (0.5);
\filldraw[black] (-3.9,1.2) circle (0.000001pt) node[anchor=south] {$2$};
\draw (-5.1,1.7) circle (0.5);
\filldraw[black] (-5.1,1.2) circle (0.000001pt) node[anchor=south] {$2$};
\draw (-6.3,1.7) circle (0.5);
\filldraw[black] (-6.3,1.2) circle (0.000001pt) node[anchor=south] {$1$};
\draw (-8.7,1.7) circle (0.5);
\filldraw[black] (-8.7,1.2) circle (0.000001pt) node[anchor=south] {$2$};

\end{scope}

\begin{scope}[shift={(0,3)}]
\draw[-] (-1,1) --(-2,1);
\draw[-] (-2.2,1) --(-3.2,1);
\draw[-] (-3.4,1) --(-4.4,1);
\draw[-] (-4.6,1) --(-5.6,1);
\draw[-] (-5.8,1) --(-6.8,1);
\draw[-] (-7,1) --(-8,1);
\draw[-] (-8.2,1) --(-9.2,1);
\draw[-] (-9.4,1) --(-10.4,1);

\draw (-7.5,1.7) circle (0.5);
\filldraw[black] (-7.5,1.2) circle (0.000001pt) node[anchor=south] {$1$};

\end{scope}
\draw[red] (-7.5,4.2)-- (-7.5,3);
\draw[red] (-7.5,3)-- (-6.75,3);

\draw[red] (-6.3,2.7)-- (-6.3,1.5);
\draw[red] (-6.3,1.5)-- (-4.35,1.5);

\draw[blue] (-8.7,2.7)-- (-8.7,2.2);

\draw[blue] (-5.1,2.7)-- (-5.1,1.8);
\draw[blue] (-5.1,1.8)-- (-3.15,1.8);

\draw[blue] (-3.6,2.8)-- (-3.6,2);
\draw[blue] (-3.6,2)-- (-1.95,2);

\draw (-1.5,1.7) circle (0.5);
\filldraw[black] (-1.5,1.2) circle (0.000001pt) node[anchor=south] {$2$};
\draw (-2.7,1.7) circle (0.5);
\filldraw[black] (-2.7,1.2) circle (0.000001pt) node[anchor=south] {$2$};
\draw (-3.9,1.7) circle (0.5);
\filldraw[black] (-3.9,1.2) circle (0.000001pt) node[anchor=south] {$1$};
\draw (-8.7,1.7) circle (0.5);
\filldraw[black] (-8.7,1.2) circle (0.000001pt) node[anchor=south] {$2$};
\draw (-9.9,1.7) circle (0.5);
\filldraw[black] (-9.9,1.2) circle (0.000001pt) node[anchor=south] {$3$};
\end{tikzpicture}
\caption{A multiline queue of type $(2,2,1,4,4,4,2,3)$.} \label{fig0}
\end{figure}
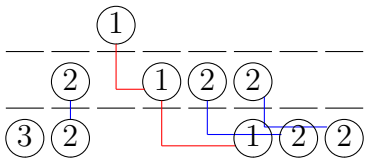

After completing the bully path projection for $Q$, let $w=(w_1,\ldots,w_n)$ be the labeling of the balls read from right to the left in row $L$ (where a vacancy is denoted by $L+1$). 
This will be a composition (not necessarily a permutation).  We say that $Q$ is a multiline queue of \emph{type $w$} and let $MLQ(w)$ denote the set of multiline queues of type $w$. 
We also define the \emph{type} of row $r$ in $Q$ to be 
the labeling of the balls in row $r$ read from right to left 
(where a vacancy is denoted by $r+1$).

 A vacancy in $Q$ is called \emph{$i$-covered}
 if it is traversed by an $i$-bully path, but not  by an 
 $i'$-bully path for $i'<i$. A \textit{trivial bully path} is 
 a bully path that goes straight down from its starting point.
\end{Def}


\begin{Def}\label{def: weight of a multiline queue}
Given an $L\times n$ multiline queue $Q$, let 
$v_r$ be the number of vacancies in row $r$ and 
$z_{r,i}$ be the number of $i$-covered vacancies in row $r$. 
Let $V_i=\sum\limits_{j=i+1}^{L}v_j$.
	The \emph{weight} of the multiline queue $Q$ is defined by
\begin{equation*}
    \wt(Q)=\prod\limits_{i=1}^{L-1}x^{V_i}_i\prod\limits_{1\leq i<r\leq L} \left(\frac{x_r}{x_i}\right)^{z_{r,i}}.
\end{equation*}
\end{Def}
\begin{Ex}
The multiline queue $Q$ in \cref{fig0} is of type $(2,2,1,4,4,4,2,3)$. Its first  and second rows have types $(2,2,2,2,2,1,2,2)$ and $(3,3,2,2,1,3,2,3)$, respectively. There is one $1$-covered vacancy  in the second row and two $1$-covered
	vacancies
 in the third row. There is no $2$-covered vacancy. The weight of $Q$ is 
\begin{equation*}
    \wt(Q)=x^{4+3}_1 x^{3}_2 (\frac{x_2}{x_1}) (\frac{x_3}{x_1})^2=x^{4}_1x^{4}_2x^{2}_3.
\end{equation*}
\end{Ex}

The next theorem was conjectured in \cite{AL} and proved in \cite{AM}.
It holds for the inhomogeneous TASEP on a ring where 
$y_i=0$ for all $i$, and the weights of particles can be 
any positive numbers (with repeats allowed).
\begin{Th}\label{thm:MLQ} \cite{AM}
	Let $w=(w_1,\ldots,w_n)$ be a composition, and 
	consider the inhomogeneous TASEP on a ring (with $y_i=0$ for all $i$)
	whose states are all compositions obtained
by permuting the parts of $w$.
of $w$.  
	Then the (unnormalized) 
	steady state probability $\psi_w$ can be expressed as a 
	weight-generating functions for multiline queues, that is, 
\begin{equation*}
\psi_w=\sum\limits_{Q\in MLQ(w)}\wt(Q).    
\end{equation*}
\end{Th}

\cref{Thm: bijection} below says that there is 
a bijection between multiline queues associated to Grassmann permutations
and flagged semistandard tableau  
(see \cref{def:flaggedtableaux}).  
After giving some preparatory lemmas,
we state the bijection in 
\cref{def:bijection} and illustrate it in 
\cref{ex:bij}.
Note that \cref{Thm: bijection}  gives a new
proof of 
\cref{thm:main0} when $k=1$.

\begin{Th}\label{Thm: bijection}
Given $\lambda\in\Val(n)$, let $d=(n-\lambda_1,n-\lambda_2,\ldots,n-\lambda_{\length(\lambda)})$. Then there is a bijection $f: MLQ(w(\lambda;n))
	\rightarrow SSYT(\lambda,d)$ 
	such that the number of $i$-covered vacancies in row $r$ of $Q\in MLQ(w(\lambda;n))$ equals the number of $r$'s in row $i$ of $f(Q)$.
\end{Th}

\begin{Lemma}\label{Lemma: MLQ basic}
	Given an $L \times n$ multiline queue $Q$, let 
	$w^{(r)}=\{w^{(r)}_i\}$ be 
		the type of row $r$ in $Q$, where the subscript $i$ in 
	$w^{(r)}_i$ refers to the $i$th column of $Q$, read right to left. Suppose that 
	there exist $a<b\leq L' \leq L$
	and $i, j$ in $\{1,2,\ldots,n\}$ such that:
\begin{align*}
    w^{(L)}_i&=a\\
    w^{(r)}_j&=b, \hspace{2mm} \text{for $ L'\leq r\leq L$}\\
    w^{(r)}_k&\neq a,  \hspace{2mm} \text{for $ L'-1\leq r\leq L$ and $k\in \{i+1,\ldots,j-1\}$}
\end{align*}
where $\{i+1,\ldots,j-1\}$ denotes $\{i+1,\ldots,n,1,\ldots,j-1\}$ if $j<i$. 
Then $w^{(r)}_i=a$ for $ L'-1\leq r\leq L$.  See 
\cref{fig:lemmaQ}.
\end{Lemma}
\begin{proof}
By assumption the statement holds for $r=L$.  We prove it for $r=L-1, L-2,\ldots, L'-1$ in that order. Assume we have $w^{(r)}_i=a$ for some $r$ between $L'$ and $L$. If $w^{(r-1)}_i\neq a$ then we have $w^{(r-1)}_k\neq a$ for $k$ in $\{i,\ldots,j-1\}$. So the $a$-bully path going to the ball $w^{(r)}_i=a$ traverses the position $w^{(r)}_j=b$. Since $b\leq r$, the position $w^{(r)}_j$ is not a vacancy. It is a contradiction since the $a$-bully path traverses the ball labeled by $b>a$. Thus we have $w^{(r-1)}_i=a$.   
\end{proof}

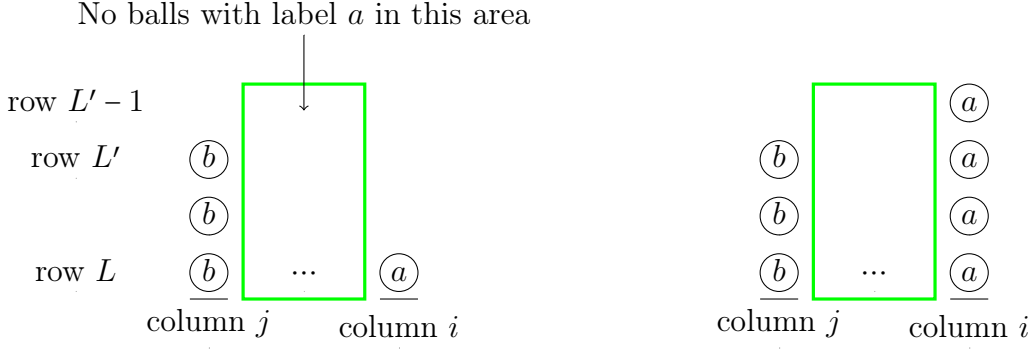
\begin{figure}[h]\centering
\begin{tikzpicture}[scale=0.5]

\draw[-] (-1,1) --(-2,1);
\draw (-1.5,1.7) circle (0.5);
\filldraw[black] (-1.5,1.2) circle (0.000001pt) node[anchor=south] {$b$};

\filldraw[black] (-5,1.2) circle (0.000001pt) node[anchor=south] {row $L$};
\filldraw[black] (-5,4.2) circle (0.000001pt) node[anchor=south] {row $L'$};
\filldraw[black] (-5,5.7) circle (0.000001pt) node[anchor=south] {row $L'-1$};

\filldraw[black] (1,1.2) circle (0.000001pt) node[anchor=south] {$\cdots$};

\draw[green, very thick] (-0.6,1) rectangle (2.6,6.7);

\begin{scope}[shift={(5,0)}]
\draw[-] (-1,1) --(-2,1);
\draw (-1.5,1.7) circle (0.5);
\filldraw[black] (-1.5,1.2) circle (0.000001pt) node[anchor=south] {$a$};
\end{scope}

\begin{scope}[shift={(0,1.5)}]
\draw (-1.5,1.7) circle (0.5);
\filldraw[black] (-1.5,1.2) circle (0.000001pt) node[anchor=south] {$b$};
\end{scope}

\begin{scope}[shift={(0,3)}]
\draw (-1.5,1.7) circle (0.5);
\filldraw[black] (-1.5,1.2) circle (0.000001pt) node[anchor=south] {$b$};
\end{scope}

\begin{scope}[shift={(5,-1.5)}]

\filldraw[black] (-1.5,1.2) circle (0.000001pt) node[anchor=south] {column $i$};
\end{scope}

\begin{scope}[shift={(0,-1.5)}]

\filldraw[black] (-1.5,1.2) circle (0.000001pt) node[anchor=south] {column $j$};
\end{scope}

\filldraw[black] (1,8) circle (0.000001pt) node[anchor=south] {No balls with label $a$ in this area};

\draw[->] (1,8)--(1,6);

\begin{scope}[shift={(15,0)}]
\draw[-] (-1,1) --(-2,1);
\draw (-1.5,1.7) circle (0.5);
\filldraw[black] (-1.5,1.2) circle (0.000001pt) node[anchor=south] {$b$};

\filldraw[black] (1,1.2) circle (0.000001pt) node[anchor=south] {$\cdots$};

\draw[green, very thick] (-0.6,1) rectangle (2.6,6.7);

\begin{scope}[shift={(5,0)}]
\draw[-] (-1,1) --(-2,1);
\draw (-1.5,1.7) circle (0.5);
\filldraw[black] (-1.5,1.2) circle (0.000001pt) node[anchor=south] {$a$};
\end{scope}

\begin{scope}[shift={(5,1.5)}]
\draw (-1.5,1.7) circle (0.5);
\filldraw[black] (-1.5,1.2) circle (0.000001pt) node[anchor=south] {$a$};
\end{scope}

\begin{scope}[shift={(5,3)}]
\draw (-1.5,1.7) circle (0.5);
\filldraw[black] (-1.5,1.2) circle (0.000001pt) node[anchor=south] {$a$};
\end{scope}

\begin{scope}[shift={(5,4.5)}]
\draw (-1.5,1.7) circle (0.5);
\filldraw[black] (-1.5,1.2) circle (0.000001pt) node[anchor=south] {$a$};
\end{scope}

\begin{scope}[shift={(0,1.5)}]
\draw (-1.5,1.7) circle (0.5);
\filldraw[black] (-1.5,1.2) circle (0.000001pt) node[anchor=south] {$b$};
\end{scope}

\begin{scope}[shift={(0,3)}]
\draw (-1.5,1.7) circle (0.5);
\filldraw[black] (-1.5,1.2) circle (0.000001pt) node[anchor=south] {$b$};
\end{scope}

\begin{scope}[shift={(5,-1.5)}]

\filldraw[black] (-1.5,1.2) circle (0.000001pt) node[anchor=south] {column $i$};
\end{scope}

\begin{scope}[shift={(0,-1.5)}]

\filldraw[black] (-1.5,1.2) circle (0.000001pt) node[anchor=south] {column $j$};
\end{scope}

\end{scope}

\end{tikzpicture}
\caption{The diagrams at the left and right of this figure
illustrate the hypothesis and conclusion
of \cref{Lemma: MLQ basic}.}
\label{fig:lemmaQ}
\end{figure}

Recall the definition of $w(\lambda;n)$ from \cref{Def: IGP Consturction}.
\cref{Cor: MLQ TRIVIALPATH}
 shows that the bottom part of 
many of the bully paths 
in $Q\in MLQ(w(\lambda;n))$ will be \emph{trivial} (a column of balls with the same label, as in 
\cref{fig:Applying Coroallry 8.8}),
which in turn will help us show that
the corresponding tableau we associate to $Q$ has the correct number
of rows.

\begin{Cor}\label{Cor: MLQ TRIVIALPATH}
For $\lambda\in\Val(n)$ and $Q\in MLQ(w(\lambda;n))$, let 
$w^{(r)}=\{w^{(r)}_i\}$ be 
the type of row $r$ in $Q$, where the subscript $i$ in 
$w^{(r)}_i$ refers to the $i$th column, read right to left.
Choose any $1 \leq i \leq n$.  

If $1\leq w^{(n-1)}_{i}\leq \length(\lambda)$, then  
	  $w^{(r)}_{i}=w^{(n-1)}_{i}$ for $ n-\lambda_{w^{(n-1)}_{i}}\leq r\leq n-1$.   
	  In other words, the height of the trivial part
	  of the bully path in column $i$ will be $\lambda_{w_i^{(n-1)}}$.

	If $\length(\lambda)+1\leq w^{(n-1)}_{i}\leq n$, then 
  $w^{(r)}_{i}=w^{(n-1)}_{i}$ for $w^{(n-1)}_{i}\leq r\leq n-1$. 
	 In other words, the height of the trivial part
	  of the bully path in column $i$ will be $n-w_i^{(n-1)}$.
\end{Cor}
\begin{proof}

Note that if $w^{(n-1)}_i=n,n-1$ or $\lambda_{w^{(n-1)}_i}=1$ then the claim is vacuous. There are several cases to consider.

\begin{enumerate}
	\item[(1)] Suppose $1\leq i\leq \lambda_1+\length(\lambda)-1$. We use induction on $i$ starting from the largest number to the smallest. The base case $i= \lambda_1+\length(\lambda)-1$ is trivial since we have either $w^{(n-1)}_i=n-1$ or $\lambda_{w^{(n-1)}_{i}}=1$. Suppose the 
	desired statement holds  for $j$ such that $i< j\leq \lambda_1+\length(\lambda)-1$. 
		\begin{enumerate}
			\item[(a)]
				If $1\leq w^{(n-1)}_i\leq \length(\lambda)$ then we have either $1\leq w^{(n-1)}_{i+1}\leq \length(\lambda)$ with $w^{(n-1)}_{i+1}= w^{(n-1)}_i+1$ and $\lambda_{ w^{(n-1)}_i}=\lambda_{ w^{(n-1)}_{i+1}}$, or $\length(\lambda)+1\leq w^{(n-1)}_{i+1}\leq n-1$ with $w^{(n-1)}_{i+1}=n-\lambda_{w^{(n-1)}_{i}}+1$. Either way, the 
	induction hypothesis implies that $w^{(r)}_{i+1}=w^{(n-1)}_{i+1}$ for $n-\lambda_{w^{(n-1)}_i}+1\leq r\leq n-1$, and applying \cref{Lemma: MLQ basic} gives the claim for $i$. 
\item[(b)] If $\length(\lambda)+1\leq w^{(n-1)}_i\leq n-1$ then consider $k$ such that $w^{(n-1)}_k=w^{(n-1)}_i+1$. If $k=i+1$ then the claim follows from \cref{Lemma: MLQ basic} as before. If $k>i+1$ we have $1\leq w^{(n-1)}_l\leq \length(\lambda)$ and $n-\lambda_{w^{(n-1)}_l}+1=w^{(n-1)}_k$ for $i+1\leq l\leq k-1$. By the induction hypothesis, we have $w^{(r)}_l=w^{(n-1)}_l$ for $w^{(n-1)}_{i}\leq r\leq n-1$ and $i+1\leq l\leq k-1$. In particular, $w^{(r)}_l\neq w^{(n-1)}_i$ for $w^{(n-1)}_{i}\leq r\leq n-1$ and $i+1\leq l\leq k-1$. Now applying \cref{Lemma: MLQ basic} gives the claim for $i$.
		\end{enumerate}
	\item[(2)] If 
 $i=\lambda_1+\length(\lambda)$, the desired statement is trivial.  
	\item[(3)] If  $\lambda_1+\length(\lambda)+1\leq i\leq n$, then the 
	proof is similar to that of Case 1.
\end{enumerate}
\end{proof}

\begin{Ex}
Consider $\lambda=(3,3,2,1)\in \Val(9)$ and $w(\lambda;9)=(1,2,7,3,8,4,9,5,6)$. For  $Q\in MLQ(w(\lambda;9))$, \cref{Cor: MLQ TRIVIALPATH} implies that the bottom portions of 
many of the $i$-bully paths are trivial (end with a sequence of 
	vertical steps), as shown 
	in \cref{fig:Applying Coroallry 8.8}.
\end{Ex}

\begin{figure}[h]\centering
\begin{tikzpicture}[scale=0.5]

\draw[-] (-1,1) --(-2,1);
\draw[-] (-2.2,1) --(-3.2,1);
\draw[-] (-3.4,1) --(-4.4,1);
\draw[-] (-4.6,1) --(-5.6,1);
\draw[-] (-5.8,1) --(-6.8,1);
\draw[-] (-7,1) --(-8,1);
\draw[-] (-8.2,1) --(-9.2,1);
\draw[-] (-9.4,1) --(-10.4,1);
\draw[-] (-10.6,1) --(-11.6,1);

\draw (-1.5,1.7) circle (0.5);
\filldraw[black] (-1.5,1.2) circle (0.000001pt) node[anchor=south] {$1$};

\draw (-2.7,1.7) circle (0.5);
\filldraw[black] (-2.7,1.2) circle (0.000001pt) node[anchor=south] {$2$};

\draw (-3.9,1.7) circle (0.5);
\filldraw[black] (-3.9,1.2) circle (0.000001pt) node[anchor=south] {$7$};

\draw (-5.1,1.7) circle (0.5);
\filldraw[black] (-5.1,1.2) circle (0.000001pt) node[anchor=south] {$3$};

\draw (-6.3,1.7) circle (0.5);
\filldraw[black] (-6.3,1.2) circle (0.000001pt) node[anchor=south] {$8$};

\draw (-7.5,1.7) circle (0.5);
\filldraw[black] (-7.5,1.2) circle (0.000001pt) node[anchor=south] {$4$};

\draw (-9.9,1.7) circle (0.5);
\filldraw[black] (-9.9,1.2) circle (0.000001pt) node[anchor=south] {$5$};

\draw (-11.1,1.7) circle (0.5);
\filldraw[black] (-11.1,1.2) circle (0.000001pt) node[anchor=south] {$6$};

\begin{scope}[shift={(0,1.5)}]
\draw (-1.5,1.7) circle (0.5);
\filldraw[black] (-1.5,1.2) circle (0.000001pt) node[anchor=south] {$1$};

\draw (-2.7,1.7) circle (0.5);
\filldraw[black] (-2.7,1.2) circle (0.000001pt) node[anchor=south] {$2$};

\draw (-3.9,1.7) circle (0.5);
\filldraw[black] (-3.9,1.2) circle (0.000001pt) node[anchor=south] {$7$};

\draw (-5.1,1.7) circle (0.5);
\filldraw[black] (-5.1,1.2) circle (0.000001pt) node[anchor=south] {$3$};

\draw (-9.9,1.7) circle (0.5);
\filldraw[black] (-9.9,1.2) circle (0.000001pt) node[anchor=south] {$5$};

\draw (-11.1,1.7) circle (0.5);
\filldraw[black] (-11.1,1.2) circle (0.000001pt) node[anchor=south] {$6$};
\end{scope}

\begin{scope}[shift={(0,3)}]
\draw (-1.5,1.7) circle (0.5);
\filldraw[black] (-1.5,1.2) circle (0.000001pt) node[anchor=south] {$1$};

\draw (-2.7,1.7) circle (0.5);
\filldraw[black] (-2.7,1.2) circle (0.000001pt) node[anchor=south] {$2$};

\draw (-9.9,1.7) circle (0.5);
\filldraw[black] (-9.9,1.2) circle (0.000001pt) node[anchor=south] {$5$};

\draw (-11.1,1.7) circle (0.5);
\filldraw[black] (-11.1,1.2) circle (0.000001pt) node[anchor=south] {$6$};
\end{scope}

\begin{scope}[shift={(0,4.5)}]
\draw (-9.9,1.7) circle (0.5);
\filldraw[black] (-9.9,1.2) circle (0.000001pt) node[anchor=south] {$5$};

\end{scope}

\draw[-][red] (-11.1,2.7) --(-11.1,2.2);
\draw[-][red] (-11.1,4.2) --(-11.1,3.7);

\draw[-][red] (-9.9,2.7) --(-9.9,2.2);
\draw[-][red] (-9.9,4.2) --(-9.9,3.7);
\draw[-][red] (-9.9,5.7) --(-9.9,5.2);

\draw[-][blue] (-5.1,2.7) --(-5.1,2.2);

\draw[-][blue] (-2.7,4.2) --(-2.7,3.7);
\draw[-][blue] (-2.7,2.7) --(-2.7,2.2);

\draw[-][red] (-3.9,2.7) --(-3.9,2.2);

\draw[-][blue] (-1.5,2.7+1.5) --(-1.5,2.2+1.5);
\draw[-][blue] (-1.5,2.7) --(-1.5,2.2);

\end{tikzpicture}
\caption{Part of a multiline queue in $MLQ(w(3,3,2,1);9)=MLQ(1,2,7,3,8,4,9,5,6)$.} 
\label{fig:Applying Coroallry 8.8}
\end{figure}
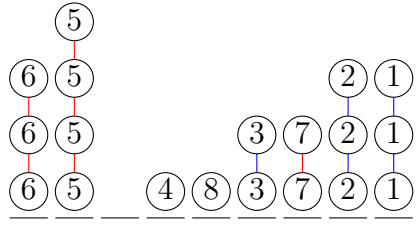

The following lemma is closely related to \cite[Lemma 5.12]{AGS}.  
\begin{Lemma}\label{Lemma: non intersecting non wrapping}
Fix $\lambda\in\Val(n)$ and $Q\in MLQ(w(\lambda;n))$, and consider row $r$
of $Q$.
If $n-\lambda_1 \leq r\leq n-1$ then the numbers 1 through $n-\lambda_1$ 
	occur in increasing order as we read the row type $w^{(r)}$
	(from right to left as usual). And if $r<n-\lambda_1$, the numbers 1 through $r$ occur in increasing order in $w^{(r)}$. 
It follows that the $i$-bully paths for $1\leq i\leq n-\lambda_1$ are 
non-intersecting and they do not wrap around.
\end{Lemma}
\begin{proof}
We use  induction on $r$ starting from the largest number to the smallest. The base case $r=n-1$ is trivial. 

	First consider the case $n-\lambda_1\leq r<n-1$ and assume the statement holds for $(r+1)$. Then by the induction hypothesis, $w^{(r+1)}_{a_i}=i$ for $\leq i \leq n-\lambda_1$ and $a_1<\cdots<a_{n-\lambda_1}$. For $1\leq i\leq n-\lambda_1-1$, it is immediate that $w^{(r)}_{b_i}=i$ for some $a_i\leq b_i <a_{i+1}$, otherwise the $i$-bully path must 
	cross over the ball labeled $i+1$ in row $r+1$ and column $a_{i+1}$, which is a contradiction.
	It remains to show  $w^{(r)}_{b_{n-\lambda_1}}=n-\lambda_1$ for some $a_{n-\lambda_1}\leq b_{n-\lambda_1}$. By \cref{Cor: MLQ TRIVIALPATH} we have $w^{(r)}_{k}=k$ for $1\leq k\leq \mul(\lambda)$ and $w^{(r+1)}_{\mul(\lambda+1)}=n-\lambda_1+1$. Thus $b_{n-\lambda_1}\geq \mul(\lambda)+1$ and if $a_{n-\lambda_1}> b_{n-\lambda_1}$ then the $(n-\lambda_1)$-bully path must cross over the 
	ball labeled $n-\lambda_1+1$ in row $r+1$.

For the case $r<n-\lambda_1$, we use
induction on $r$, treating row $n-\lambda_1$ as our base case,
and decreasing $r$ by $1$ at each step.
	Assume the statement holds for row $(r+1)$.
	 Then by the induction hypothesis 
	we have positions $a_1<\cdots<a_{r+1}$
such that 
	$w^{(r+1)}_{a_i}=i$ for $1 \leq i \leq r+1$.
	For $1\leq i\leq r$, it is immediate that $w^{(r)}_{b_i}=i$ for some $a_i\leq b_i <a_{i+1}$, otherwise the 
	$i$-bully path must cross over the ball labeled $i+1$
	in row $r+1$. 
\end{proof}

\begin{Lemma}\label{Lemma:identify with non-intersecting lattice paths}
For $\lambda\in\Val(n)$, we can identify a multiline queue $Q\in MLQ(w(\lambda;n))$ with a set of non-intersecting lattice paths $\{P_1,\ldots,P_{\length(\lambda)}\}$ where $P_i$ is a lattice path from $(\lambda_1+i,i)$ to $(\lambda_1-\lambda_i+i,n-\lambda_i)$.
\end{Lemma}
\begin{proof}
Let $w^{(r)}$ be the type of row $r$ in $Q$ (read right to left) and we denote $(i,j)$ to be the position for $w^{(j)}_i$ in $Q$. For each $1\leq i\leq \length(\lambda)$, we extend $i$-bully path so that it starts from $(n,i)$, far left position in $Q$. With this setting, we first claim that $\length(\lambda)$-bully path passes the point $(\lambda_1+\length(\lambda),\length(\lambda))$. If $n=\lambda_1+\length(\lambda)$, there is nothing to prove. If $n>\lambda_1+\length(\lambda)$ then we have $w^{(n-1)}_{\lambda_1+\length(\lambda)+1}=\length(\lambda)+1$ so by \cref{Cor: MLQ TRIVIALPATH} we know $w^{(\length(\lambda)+1)}_{\lambda_1+\length(\lambda)+1}=\length(\lambda)+1$. Since $\length(\lambda)$-bully path cannot traverse the position $(\lambda_1+\length(\lambda)+1,\length(\lambda)+1)$, it must pass $(\lambda_1+\length(\lambda),\length(\lambda))$. 

By \cref{Lemma: non intersecting non wrapping}, $i$-bully paths for $1\leq i\leq \length(\lambda)$ are non-intersecting and $\length(\lambda)$-bully path passes the point $(\lambda_1+\length(\lambda),\length(\lambda))$. So $(\length(\lambda)-1)$-bully path must pass the point $(\lambda_1+\length(\lambda)-1,\length(\lambda))-1$. Repeating this we deduce each $i$-bully path must pass the position $(\lambda_1+i,i)$ for $1\leq i\leq \length(\lambda)$. From the definition of $w(\lambda;n)$ we have $w^{(n-1)}_{\lambda_1+i-\lambda_i}=i$ and by \cref{Cor: MLQ TRIVIALPATH} $i$-bully path only moves vertically from the position $(\lambda_1+i-\lambda_i,n-\lambda_i)$. So it is enough to assign a lattice path from $(\lambda_1+i,i)$ to $(\lambda_1+i-\lambda_i,n-\lambda_i)$ to decide $i$-bully path. 

Note that $i$-bully paths for $i>\length(\lambda)$ are trivial bully paths by \cref{Cor: MLQ TRIVIALPATH}. So we can identify $Q\in MLQ(w(\lambda;n))$ with a set of non-intersecting lattice paths \{$P_1,\ldots,P_{\length(\lambda)}\}$ where $P_i$ is a path from $(\lambda_1+i,i)$ to $(\lambda_1+i-\lambda_i,n-\lambda_i)$. 
\end{proof}

Now we can define the bijection whose existence was asserted in 
 \cref{Thm: bijection}. 
\begin{Def}\label{def:bijection}
Given $\lambda\in\Val(n)$, set $d=(n-\lambda_1,n-\lambda_2,\ldots,n-\lambda_{\length(\lambda)})$, and choose some
 $Q\in MLQ(w(\lambda;n))$.  We associate to $Q$ 
a set of non-intersecting lattice paths $\{P_1,\ldots,P_{\length(\lambda)}\}$ as in \cref{Lemma:identify with non-intersecting lattice paths}. 
	We then let $f(Q)$ denote the semistandard tableaux with 
	$\length(\lambda)$ rows, where the entries in the $i$th row
	record the row numbers of the horizontal steps of path $P_i$.
	Clearly the $i$th row has $\lambda_i$ entries which 
	are bounded by $n-\lambda_i$, so 
 $f(Q) \in 
SSYT(\lambda,d)$.
\end{Def}

\begin{figure}[h]\centering
\begin{tikzpicture}[scale=0.5]

\draw[-] (-1,1) --(-2,1);
\draw[-] (-2.2,1) --(-3.2,1);
\draw[-] (-3.4,1) --(-4.4,1);
\draw[-] (-4.6,1) --(-5.6,1);
\draw[-] (-5.8,1) --(-6.8,1);
\draw[-] (-7,1) --(-8,1);
\draw[-] (-8.2,1) --(-9.2,1);
\draw[-] (-9.4,1) --(-10.4,1);
\draw[-] (-10.6,1) --(-11.6,1);

\draw (-1.5,1.7) circle (0.5);
\filldraw[black] (-1.5,1.2) circle (0.000001pt) node[anchor=south] {$1$};

\draw (-2.7,1.7) circle (0.5);
\filldraw[black] (-2.7,1.2) circle (0.000001pt) node[anchor=south] {$2$};

\draw (-3.9,1.7) circle (0.5);
\filldraw[black] (-3.9,1.2) circle (0.000001pt) node[anchor=south] {$7$};

\draw (-5.1,1.7) circle (0.5);
\filldraw[black] (-5.1,1.2) circle (0.000001pt) node[anchor=south] {$3$};

\draw (-6.3,1.7) circle (0.5);
\filldraw[black] (-6.3,1.2) circle (0.000001pt) node[anchor=south] {$8$};

\draw (-7.5,1.7) circle (0.5);
\filldraw[black] (-7.5,1.2) circle (0.000001pt) node[anchor=south] {$4$};

\draw (-9.9,1.7) circle (0.5);
\filldraw[black] (-9.9,1.2) circle (0.000001pt) node[anchor=south] {$5$};

\draw (-11.1,1.7) circle (0.5);
\filldraw[black] (-11.1,1.2) circle (0.000001pt) node[anchor=south] {$6$};

\begin{scope}[shift={(0,1.5)}]
\draw (-1.5,1.7) circle (0.5);
\filldraw[black] (-1.5,1.2) circle (0.000001pt) node[anchor=south] {$1$};

\draw (-2.7,1.7) circle (0.5);
\filldraw[black] (-2.7,1.2) circle (0.000001pt) node[anchor=south] {$2$};

\draw (-3.9,1.7) circle (0.5);
\filldraw[black] (-3.9,1.2) circle (0.000001pt) node[anchor=south] {$7$};

\draw (-5.1,1.7) circle (0.5);
\filldraw[black] (-5.1,1.2) circle (0.000001pt) node[anchor=south] {$3$};

\draw (-7.5,1.7) circle (0.5);
\filldraw[black] (-7.5,1.2) circle (0.000001pt) node[anchor=south] {$4$};

\draw (-9.9,1.7) circle (0.5);
\filldraw[black] (-9.9,1.2) circle (0.000001pt) node[anchor=south] {$5$};

\draw (-11.1,1.7) circle (0.5);
\filldraw[black] (-11.1,1.2) circle (0.000001pt) node[anchor=south] {$6$};
\end{scope}

\begin{scope}[shift={(0,3)}]
\draw (-1.5,1.7) circle (0.5);
\filldraw[black] (-1.5,1.2) circle (0.000001pt) node[anchor=south] {$1$};

\draw (-2.7,1.7) circle (0.5);
\filldraw[black] (-2.7,1.2) circle (0.000001pt) node[anchor=south] {$2$};

\draw (-5.1,1.7) circle (0.5);
\filldraw[black] (-5.1,1.2) circle (0.000001pt) node[anchor=south] {$3$};

\draw (-7.5,1.7) circle (0.5);
\filldraw[black] (-7.5,1.2) circle (0.000001pt) node[anchor=south] {$4$};

\draw (-9.9,1.7) circle (0.5);
\filldraw[black] (-9.9,1.2) circle (0.000001pt) node[anchor=south] {$5$};

\draw (-11.1,1.7) circle (0.5);
\filldraw[black] (-11.1,1.2) circle (0.000001pt) node[anchor=south] {$6$};
\end{scope}

\begin{scope}[shift={(0,4.5)}]
\draw (-1.5,1.7) circle (0.5);
\filldraw[black] (-1.5,1.2) circle (0.000001pt) node[anchor=south] {$1$};
\draw (-3.9,1.7) circle (0.5);
\filldraw[black] (-3.9,1.2) circle (0.000001pt) node[anchor=south] {$2$};
\draw (-5.1,1.7) circle (0.5);
\filldraw[black] (-5.1,1.2) circle (0.000001pt) node[anchor=south] {$3$};
\draw (-8.7,1.7) circle (0.5);
\filldraw[black] (-8.7,1.2) circle (0.000001pt) node[anchor=south] {$4$};
\draw (-9.9,1.7) circle (0.5);
\filldraw[black] (-9.9,1.2) circle (0.000001pt) node[anchor=south] {$5$};

\end{scope}

\begin{scope}[shift={(0,6)}]
\draw (-1.5,1.7) circle (0.5);
\filldraw[black] (-1.5,1.2) circle (0.000001pt) node[anchor=south] {$1$};
\draw (-3.9,1.7) circle (0.5);
\filldraw[black] (-3.9,1.2) circle (0.000001pt) node[anchor=south] {$2$};
\draw (-5.1,1.7) circle (0.5);
\filldraw[black] (-5.1,1.2) circle (0.000001pt) node[anchor=south] {$3$};

\draw (-8.7,1.7) circle (0.5);
\filldraw[black] (-8.7,1.2) circle (0.000001pt) node[anchor=south] {$4$};

\end{scope}

\begin{scope}[shift={(0,7.5)}]
\draw (-2.7,1.7) circle (0.5);
\filldraw[black] (-2.7,1.2) circle (0.000001pt) node[anchor=south] {$1$};
\draw (-3.9,1.7) circle (0.5);
\filldraw[black] (-3.9,1.2) circle (0.000001pt) node[anchor=south] {$2$};

\draw (-6.3,1.7) circle (0.5);
\filldraw[black] (-6.3,1.2) circle (0.000001pt) node[anchor=south] {$3$};

\end{scope}

\begin{scope}[shift={(0,9)}]
\draw (-2.7,1.7) circle (0.5);
\filldraw[black] (-2.7,1.2) circle (0.000001pt) node[anchor=south] {$1$};
\draw (-5.1,1.7) circle (0.5);
\filldraw[black] (-5.1,1.2) circle (0.000001pt) node[anchor=south] {$2$};
\end{scope}

\begin{scope}[shift={(0,10.5)}]
\draw (-2.7,1.7) circle (0.5);
\filldraw[black] (-2.7,1.2) circle (0.000001pt) node[anchor=south] {$1$};
\end{scope}

\draw[-][red] (-11.1,2.7) --(-11.1,2.2);
\draw[-][red] (-11.1,4.2) --(-11.1,3.7);

\draw[-][red] (-9.9,2.7) --(-9.9,2.2);
\draw[-][red] (-9.9,4.2) --(-9.9,3.7);
\draw[-][red] (-9.9,5.7) --(-9.9,5.2);

\draw[-][blue] (-8.7,7.2) --(-8.7,6.7);
\draw[-][blue] (-8.7,5.7) --(-8.7,4.7);
\draw[-][blue] (-8.7,4.7) --(-8.0,4.7);
\draw[-][blue] (-7.5,2.7) --(-7.5,2.2);
\draw[-][blue] (-7.5,4.2) --(-7.5,3.7);

\draw[-][blue] (-7.5,9.2) --(-6.8,9.2);
\draw[-][blue] (-6.3,8.7) --(-6.3,7.7);
\draw[-][blue] (-6.3,7.7) --(-5.6,7.7);
\draw[-][blue] (-5.1,7.2) --(-5.1,6.7);
\draw[-][blue] (-5.1,5.7) --(-5.1,5.2);
\draw[-][blue] (-5.1,4.2) --(-5.1,3.7);
\draw[-][blue] (-5.1,2.7) --(-5.1,2.2);

\draw[-][blue] (-6.3,10.7) --(-5.6,10.7);
\draw[-][blue] (-5.1,10.2) --(-5.1,9.2);
\draw[-][blue] (-5.1,9.2) --(-4.4,9.2);
\draw[-][blue] (-3.9,8.7) --(-3.9,8.2);
\draw[-][blue] (-3.9,7.2) --(-3.9,6.7);
\draw[-][blue] (-3.9,5.7) --(-3.9,4.7);
\draw[-][blue] (-3.9,4.7) --(-3.2,4.7);
\draw[-][blue] (-2.7,4.2) --(-2.7,3.7);
\draw[-][blue] (-2.7,2.7) --(-2.7,2.2);
\draw[-][red] (-3.9,2.7) --(-3.9,2.2);

\draw[-][blue] (-5.1,12.2) --(-3.2,12.2);
\draw[-][blue] (-2.7,11.7) --(-2.7,11.2);
\draw[-][blue] (-2.7,10.2) --(-2.7,9.7);
\draw[-][blue] (-2.7,8.7) --(-2.7,7.7);
\draw[-][blue] (-2.7,7.7) --(-2,7.7);
\draw[-][blue] (-1.5,2.7+4.5) --(-1.5,2.2+4.5);
\draw[-][blue] (-1.5,2.7+3) --(-1.5,2.2+3);
\draw[-][blue] (-1.5,2.7+1.5) --(-1.5,2.2+1.5);
\draw[-][blue] (-1.5,2.7) --(-1.5,2.2);

\draw[green, very thick] (-9.3,7.1) rectangle (-8.1,8.3);
\begin{scope}[shift={(1.2,1.5)}]
\draw[green, very thick] (-9.3,7.1) rectangle (-8.1,8.3);
\end{scope}
\begin{scope}[shift={(2.4,3)}]
\draw[green, very thick] (-9.3,7.1) rectangle (-8.1,8.3);
\end{scope}
\begin{scope}[shift={(3.6,4.5)}]
\draw[green, very thick] (-9.3,7.1) rectangle (-8.1,8.3);
\end{scope}

\begin{scope}[shift={(1.2,-6)}]
\draw[green, very thick] (-9.3,7.1) rectangle (-8.1,8.3);
\end{scope}
\begin{scope}[shift={(3.6,-4.5)}]
\draw[green, very thick] (-9.3,7.1) rectangle (-8.1,8.3);
\end{scope}
\begin{scope}[shift={(6,-3)}]
\draw[green, very thick] (-9.3,7.1) rectangle (-8.1,8.3);
\end{scope}
\begin{scope}[shift={(7.2,-3)}]
\draw[green, very thick] (-9.3,7.1) rectangle (-8.1,8.3);
\end{scope}

\begin{scope}[shift={(7,0)}]
\draw[-] (0,2) --(2,2);
\draw[-] (2,2) --(2,10);
\draw[-] (0,4) --(4,4);
\draw[-] (4,4) --(4,10);
\draw[-] (0,6) --(6,6);
\draw[-] (0,8) --(6,8);
\draw[-] (6,6) --(6,10);
\draw[-] (6,10) --(0,10);
\draw[-] (0,10) --(0,2);

\filldraw[black] (1,2.5) circle (0.000001pt) node[anchor=south] {$6$};
\filldraw[black] (1,4.5) circle (0.000001pt) node[anchor=south] {$3$};
\filldraw[black] (1,6.5) circle (0.000001pt) node[anchor=south] {$2$};
\filldraw[black] (1,8.5) circle (0.000001pt) node[anchor=south] {$1$};

\filldraw[black] (3,4.5) circle (0.000001pt) node[anchor=south] {$4$};
\filldraw[black] (3,6.5) circle (0.000001pt) node[anchor=south] {$3$};
\filldraw[black] (3,8.5) circle (0.000001pt) node[anchor=south] {$1$};

\filldraw[black] (5,6.5) circle (0.000001pt) node[anchor=south] {$6$};
\filldraw[black] (5,8.5) circle (0.000001pt) node[anchor=south] {$4$};

\filldraw[black] (6,9) circle (0.000001pt) node[anchor=west] {$\leq 6$};
\filldraw[black] (6,7) circle (0.000001pt) node[anchor=west] {$\leq 6$};
\filldraw[black] (4,5) circle (0.000001pt) node[anchor=west] {$\leq 7$};
\filldraw[black] (2,3) circle (0.000001pt) node[anchor=west] {$\leq 8$};

\end{scope}

\end{tikzpicture}
\caption{The left figure shows the multiline queue $Q$ in $MLQ(w(\lambda;9))$ for $\lambda=(3,3,2,1)\in\Val(9)$. And the right figure shows the corresponding semistandard young tableau.} \label{fig:bij}
\end{figure}
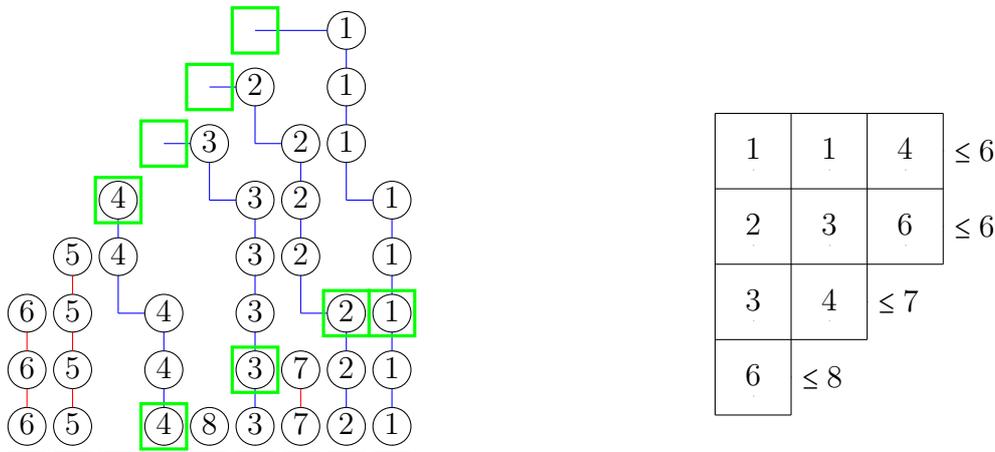

\begin{Ex}\label{ex:bij}
The left of \cref{fig:bij} shows a multiline queue $Q$ in $MLQ(w(\lambda;9))$ for $\lambda=(3,3,2,1)\in\Val(9)$. As in \cref{Lemma:identify with non-intersecting lattice paths} we associate a set of non-intersecting lattice paths $\{P_1,P_2,P_3,P_4\}$ from positions $(\lambda_1+i,i)$ to positions $(\lambda_1+i-\lambda_i,n-\lambda_i)$ (indicated with green boxes). 

The figure on the right shows the corresponding semistandard Young tableau 
	obtained by the map above. Note that the 1-bully path has horizontal steps in rows 1, 1 and 4, so we fill the first row of $\lambda$ with 1, 1 and 4. 
\end{Ex}

\begin{proof}[Proof of \cref{Thm: bijection}]
By \cref{Lemma:identify with non-intersecting lattice paths}, 
we can identify 
the bully paths of 
each multiline queue $Q\in MLQ(w(\lambda;n))$ with a collection of non-intersecting lattice paths $\{P_1,\ldots,P_{\length(\lambda)}\}$, where $P_i$ is a lattice path from $(\lambda_1+i,i)$ to $(\lambda_1-\lambda_i+i,n-\lambda_i)$; these lattice paths in turn get mapped 
	via \cref{def:bijection}
	to a semistandard tableaux
	in $SSYT(\lambda,d)$, 
	where $d=(n-\lambda_1,n-\lambda_2,\ldots,n-\lambda_{\length(\lambda)})$.   Clearly the non-intersecting condition on lattice paths
	corresponds to the strictly increasing condition on 
	columns of the corresponding semistandard tableaux.  The 
	map is also clearly invertible, so this shows that 
	the map is a bijection.
\end{proof}

We note that the fact that the map $f$ from \cref{Thm: bijection} is a bijection can also be deduced from 
the determinantal formula from \cite[Theorem 3.10]{AGS}; we thank the referee for pointing this 
out to us.

\section{A multiline queue formula for $\mathbf{z}$-deformed  probabilites}\label{sec:zMLQ}
Recall from  \cref{thm:MLQ} that when each $y_i=0$, there
is a combinatorial formula for the steady state probabilities $\psi_w$ in terms of {multiline queues}.  It is an  open problem to extend 
this result to the case that the $y_i$'s are general. 
In this section we continue to work in the case $y_i=0$,
but we give a generalization of \cref{thm:MLQ} that works
for the $\mathbf{z}$-deformed steady state probabilities.  As in the previous
 section, we allow our states $w$ to be \emph{compositions},
not just permutations.

\begin{Def}\label{def:Qir}
For a multiline queue $Q$, let $Q_{i,r}$ denote the position 
	in the $i$th column from the right and the  $r$th row from above.
	We assign a value $v(Q_{i,r})$ as follows:
\begin{align*}
	v(Q_{i,r})=\begin{cases*}0 &\text{if $Q_{i,r}$ has a ball}\\
		x_j &\text{if $Q_{i,r}$ is $j$-covered vacancy}\\
		x_{r} & \text{otherwise}.
    \end{cases*}
\end{align*}
	We define the \emph{$\mathbf{z}$-weight} $\zwt(Q)$ of $Q$ by
\begin{equation*}
    \zwt(Q)=\wt(Q)\prod\limits_{Q_{i,r}\in Q}(z_i-v(Q_{i,r})), 
\end{equation*}
where $\wt(Q)$  is given by \cref{def: weight of a multiline queue}.
And we let $$F_w(\z)=\sum\limits_{Q\in MLQ(w)}\zwt(Q)$$ be the $\mathbf{z}$-weight 
generating function for 
multiline queues of type $w$.
\end{Def}
\begin{Ex}
The multiline queue $Q$ in \cref{fig0} has  $\mathbf{z}$-weight
\begin{align*}
    \zwt(Q)=&\wt(Q)(z_1-x_1)(z_1-x_2)z_1(z_2-x_1)(z_2-x_2)z_2(z_3-x_1)z_3^{2}(z_4-x_1)^2 z_4\\&(z_5-x_1)^2 z_5(z_6-x_1)(z_6-x_3) z_6 (z_7-x_1)z_7^{2}(z_8-x_1)(z_8-x_2)z_8.
\end{align*}
\end{Ex}

\begin{Th}\label{theorem: zweight}
Let $w$ be a composition, and consider 
the inhomogeneous TASEP (with $y_i=0$ for all $i$) on the set of states obtained by permuting
the entries of $w$.  Then 
	the (unnormalized) $\mathbf{z}$-deformed steady 
state probability of state $w$ is given by  \begin{equation*}
    \psi_w(\z)=\sum\limits_{Q\in MLQ(w)}\zwt(Q).
\end{equation*}
\end{Th}

Clearly the leading coefficient of $\zwt(Q)$ in the 
$\mathbf{z}$-variables gives $\wt(Q)$. So 
 \cref{theorem: zweight}
directly implies 
\cref{thm:MLQ}. 

To prove Theorem \ref{theorem: zweight}, it is enough to show that $F_w(\z)$ satisfies the following exchange equations (see \cite[Equation (23)]{C}), where $s_i$ acts on the $\mathbf{z}$-variables. 
\begin{align}
	F_{s_i w}(\z)&=\frac{z_i(z_{i+1}-x_{w_{i+1}})}{x_{w_{i+1}}}\frac{(1-s_i)F_w(\z)}{z_i-z_{i+1}} \hspace{5mm}&\text{if $w_i>w_{i+1}$}
	\label{eq:exchange1}\\
	(1-s_i)F_w(\z)&=0 \hspace{5mm}&\text{if $w_i=w_{i+1}$}.\label{eq:exchange2}
\end{align}

\begin{Rem}
Equations \eqref{eq:exchange1} and \eqref{eq:exchange2} are 
obtained by expanding Equation (23) in \cite{C}, and 
\eqref{eq:exchange1} is a specialization of 
\cref{CantiniProp} when $y_i=0$.
We need \eqref{eq:exchange2} to deal with  the case that 
there is more than one particle of a given type.
(In \cite{C}, Cantini studied the case where there is at 
most one particle for each type so he included
	 \eqref{eq:exchange1} as \cite[Equation (27)]{C} but 
	 did not mention \eqref{eq:exchange2}.)
\end{Rem}

\begin{Rem}\label{rem:symmetry}
One can easily verify that if \eqref{eq:exchange1} holds, then 
	$F_w(\z)+F_{s_i w}(\z)$ is symmetric in $z_i$ and $z_{i+1}$,
	i.e. $(1-s_i)(F_w(\z)+F_{s_i w}(\z)) = 0$.
\end{Rem}

 Given a composition $w$, let $max(w)$  denote its largest part.
 Note that $F_w(\z)$ is a weighted sum over multiline queues with $(max(w)-1)$ rows. We will prove \cref{theorem: zweight}
by induction on the number of rows; our strategy is similar to the one used in 
\cite[Section 3]{CMW}.
 For the inductive step, we will view a multiline queue $Q$ with $L$ rows 
 as a multiline queue $Q'$ with $L-1$ rows (the restriction of $Q$ to rows $2$ through $L$)
 glued on top of a (generalized) multiline queue $Q_0$ with $2$ rows
 (the restriction of $Q$ to rows $1$ and $2$).

\begin{Def}
For compositions $u$ and $w$, we let $MLQ(u,w)$ denote the set of multiline queues $Q$ of type $w$ 
whose second row from the bottom has type $u$. We also define 
\begin{equation*}
    F_{w}^{u}=\frac{\zwt(Q)}{\zwt(Q')},
\end{equation*}
where $Q \in MLQ(u,w)$ and $Q'$ is the multiline queue obtained by
	deleting the bottom row of $Q$. It is easy to check that $F_{w}^{u}$ does not depend on the choice of $Q$ in $MLQ(u,w)$. If $MLQ(u,w)$ is empty then we set  $F_{w}^{u}=0$. 
\end{Def}
The following lemma follows directly from the definitions.
\begin{Lemma}\label{lemma fuw formula}
	Consider compositions $u=(u_1,\ldots,u_n)$ and $w=(w_1,\ldots,w_n)$ 
 such that $MLQ(u,w)$ is not empty. Pick $Q\in MLQ(u,w)$ and assume that $Q$ has $L$ rows and the number of vacancies in the last row is $t$. Then we have 
\begin{equation*}
    F_{w}^{u}=(x_1 x_2 \cdots x_L)^t  \prod\limits_{\text{$i: Q_{i,L}$ occupied}}z_i \prod\limits_{\text{$i: Q_{i,L}$  vacant}} \frac{z_i-v(Q_{i,L})}{v(Q_{i,L})}.
\end{equation*}
\end{Lemma}
\begin{Ex}
The multiline queue $Q$ in \cref{fig0} is in $MLQ(u,w)$ for $u=(3,2,2,1,3,2,3)$ and $w=(2,2,1,4,4,4,2,3)$. By \cref{lemma fuw formula}, we have
\begin{equation*}
     F_{w}^{u}=(x_1 x_2 x_3)^3   z_1 z_2 z_3 \frac{(z_4-x_1)}{x_1} \frac{(z_5-x_1)}{x_1} \frac{(z_6-x_3)}{x_3} z_7 z_8.
\end{equation*}
\end{Ex}

\begin{Rem}
By construction, for fixed $i$, 
we can write $F_w(\z)$ as 
\begin{equation}
    F_w(\z)=\sum\limits_{u}F_{w}^{u}F_u(\z)=\sum\limits_{u,u_i>u_{i+1}}(F_{w}^{u}F_u(\z)+F_{w}^{s_i u}F_{s_i u}(\z))+\sum\limits_{u,u_i=u_{i+1}}F_{w}^{u}F_u(\z).
	\label{eq:F expand}
\end{equation}
\end{Rem}

Our next goal is to analyze the quantities $F_w^{u}$, $F_w^{s_i u}$, $F_{s_i w}^{u}$, $F_{s_i w}^{s_i u}$, case by case.

\begin{Lemma}\label{lemma: mlq violation}
Let $u=(u_1,\ldots, u_n)$ and $w=(w_1,\ldots,w_n)$ be compositions.
	Then we have the following (where indices $i$ of columns
	are considered modulo $n$).
\begin{enumerate}
    \item If there exists $i$ such that $u_i<w_i<max(w)$ then 
	    $MLQ(u,w)=\emptyset$.
    \item If there exists $i$ such that $u_{i+1}<w_i<w_{i+1}=max(w)$ 
	   then $MLQ(u,w)=\emptyset$.
    \item  If there exists $i$ such that $w_i<w_{i+1}<max(w)$ 
	   and $u_{i}\neq w_{i}$,  $MLQ(u,w)=\emptyset$.
\end{enumerate}
\end{Lemma}
\begin{proof}
The first statement above says that 
in a multiline queue there is no ball which is directly above another
ball with a larger label.
This statement and the other two follow directly from the bully path algorithm.
%
%
\end{proof}

The following lemma also follows directly 
from the bully path algorithm.

\begin{Lemma}\label{lemma mlq si validity}
Let $u=(u_1,\ldots, u_n)$ and $w=(w_1,\ldots,w_n)$ be compositions.
	Then we have the following (where indices $i$ of columns
	are considered modulo $n$).
\begin{enumerate}
    \item If $u_i\geq u_{i+1}$, $w_i\geq w_{i+1}$ and $u_{i+1}=w_{i+1}$ then 
	   $MLQ(u,w) \neq \emptyset$ if and only if $MLQ(s_i u,s_i w) \neq \emptyset$. When they are both nonempty, the 
	matchings of balls between the bottom two rows are preserved.
    \item If $w_{i+1}=max(w)$ or $w_{i+1}<min(u_i,u_{i+1})$ 
	   then $MLQ(u,w) \neq \emptyset$ if and only if $MLQ(s_i u,w) \neq \emptyset$. When they are both nonempty, the 
		matchings of balls between the bottom two rows are preserved.
    \item If $w_{i}=max(w)$ and $w_{i+1}<min(u_i,u_{i+1})$ 
	   then $MLQ(u,w) \neq \emptyset$ if and only if $MLQ(u,s_i w) \neq \emptyset$. When they are both nonempty, the matchings of balls between the bottom two rows are preserved.
\end{enumerate}
\end{Lemma}

\begin{Lemma}\label{Lemma: fuw analysis1}
Let $u=(u_1,\ldots, u_n)$ and $w=(w_1,\ldots,w_n)$ be compositions
	such that $u_i>u_{i+1}$ and $w_i>w_{i+1}$.  Then we
	have the following.
\begin{enumerate}
	\item 
		If $w_i<max(w)$ and $u_{i+1}<w_{i+1}$ then 
       $ F_{w}^{u}=F_{w}^{s_i u}=F_{s_i w}^{u}=F_{s_i w}^{s_i u}=0.$
    \item If $w_i=max(w)$ and $u_{i+1}<w_{i+1}$ then 
$	    F_{w}^{u}=F_{s_i w}^{u}=F_{s_i w}^{s_i u}=0, $
    \begin{equation*}
	    \text{and  }
         F_{w}^{s_i u}=F' \cdot (z_i-x_{u_{i+1}})z_{i+1}
    \end{equation*}
    for some $F'$ that does not depend on $z_i$ and $z_{i+1}$.
    \item If $w_i<max(w)$ and $u_{i+1}=w_{i+1}$ then we have $F_{s_i w}^{u}=F_{w}^{s_i u}=0$ and we can write
    \begin{equation*}
        F_{w}^{u}=F_{s_i w}^{s_i u}=F' z_i z_{i+1}
    \end{equation*}
    for some $F'$ that does not depend on $z_i$ and $z_{i+1}$.
    \item If $w_i<max(w)$ and $u_{i+1}>w_{i+1}$ then we have $F_{s_i w}^{u}=F_{s_i w}^{s_i u}=0$ and we can write
    \begin{equation*}
        F_{w}^{u}=F_{ w}^{s_i u}=F' z_i z_{i+1}
    \end{equation*}
    for some $F'$ that does not depend on $z_i$ and $z_{i+1}$.
    \item If $w_i=max(w)$ and $u_{i+1}=w_{i+1}$ then we can write 
\begin{align*}
	F_{w}^{u}&=\frac{(z_i-x_c)z_{i+1} F'}{x_c}, 
	&F_{s_i w}^{s_i u}=\frac{(z_{i+1}-x_c)z_{i} F'}{x_c}\\
	F_{s_i w}^{u}&=\frac{(z_{i+1}-x_{w_{i+1}})z_{i}F'}{x_{w_{i+1}}},   
	&F_{w}^{s_i u}=F'' \cdot (z_i-x_{w_{i+1}})z_{i+1}.
\end{align*}
 for some integer $c$ and quantities $F', F''$ that do not depend on $z_i$ or $z_{i+1}$.
 
 \item If $w_i=max(w)$ and $u_{i+1}>w_{i+1}$ then we can write
\begin{equation*}
	F_{w}^{u}=F_{w}^{s_i u}=\frac{(z_i-x_c)z_{i+1} F'}{x_c} \text{ and } F_{s_i w}^{u}=F_{s_i w}^{s_i u}=\frac{(z_{i+1}-x_{w_{i+1}})z_{i}F'}{x_{w_{i+1}}} 
\end{equation*}
 for some integer $c$ and some $F'$ that does not depend on $z_i$ or $z_{i+1}$.
\end{enumerate}
\end{Lemma} 
\begin{proof}
Part (1) follows from \cref{lemma: mlq violation} (1). We denote $max(w)=L$.

(2) 
     We have $F_{w}^{u}=F_{s_i w}^{u}=F_{s_i w}^{s_i u}=0$ by  \cref{lemma: mlq violation} (1) and (2). Then pick $Q^{1}\in MLQ(s_i u, w)$ (if it is nonempty). Then $Q^{1}_{i,L}$ is a vacancy and there is a ball with  label $u_{i+1}$ right above it. So a $u_{i+1}$-bully path traverses $Q^{1}_{i,L}$. Since $Q^{1}_{i+1,L}$ is a ball with  label $w_{i+1}$, no bully path with a label smaller than $u_{i+1}$ traverses $Q^{1}_{i,L}$. Therefore $v(Q^{1}_{i,L})=x_{u_{i+1}}$. The conclusion for $F^{s_i u}_w$ follows from \cref{lemma fuw formula}.
    
(3) 
	By \cref{lemma: mlq violation} (1) 
	we have $F^{u}_{s_i w}=F^{s_i u}_{w}=0$. And by 
	\cref{lemma mlq si validity}, 
	$MLQ(u,w)$ is nonempty if and only if $MLQ(s_i u,s_i w)$ nonempty. When they are both nonempty (if they are both empty then we can simply take $F'=0$) take $Q^{1}\in MLQ(u,w)$ and $Q^{2}\in MLQ(s_i u,s_i w)$. Then positions $Q^{1}_{i,L}, Q^{1}_{i+1,L}, Q^{2}_{i,L}, Q^{2}_{i+1,L}$ are all occupied by balls, so by \cref{lemma fuw formula}, we can write $F_{w}^{u}=F_{s_i w}^{s_i u}=F' z_i z_{i+1}$ for some $F'$ that does not depend on $z_i$ and $z_{i+1}$.

(4) The proof is similar to part (3)
	but uses \cref{lemma: mlq violation} (3) and 
	\cref{lemma mlq si validity} (2).

(5) The conclusion for $F^{s_i u}_w$ follows similarly as in part (2).

The other three cases are straightforward using 
 \cref{lemma mlq si validity} (1), (2), 
and 	\cref{lemma fuw formula}.

(6) The proof of this part is straightforward and 
	similar to the previous proofs.

\end{proof}
The next two lemmas are very similar 
to \cref{Lemma: fuw analysis1},
but with slightly different hypotheses.  The proofs are similar so we omit them.

\begin{Lemma} \label{Lemma: fuw analysis2}
Let $u=(u_1,\ldots, u_n)$ and $w=(w_1,\ldots,w_n)$ be compositions 
	such that $u_i=u_{i+1}$ and $w_i>w_{i+1}$.  Then 
	we have the following.
\begin{enumerate}
     \item If $w_i<max(w)$ and $u_{i+1}\leq w_{i+1}$ then 
    $    F_{w}^{u}=F_{s_i w}^{u}=0.$
    \item If $w_i=max(w)$ and $u_{i+1}<w_{i+1}$ then 
    $    F_{w}^{u}=F_{s_i w}^{u}=0.$
    \item If $w_i<max(w)$ and $u_{i+1}>w_{i+1}$ then  $F_{s_i w}^{u}=0$ and we can write 
    \begin{equation*}
        F_{w}^{u}=F' z_i z_{i+1}
    \end{equation*}
    for some $F'$ that does not depend on $z_i$ and $z_{i+1}$
    \item If $w_i=max(w)$ and $u_{i+1}\geq w_{i+1}$ then we can write
        \begin{equation*}
		F_{w}^{u}=\frac{(z_{i}-x_{c})z_{i+1}F'}{x_{c}} \text{ and } F_{s_i w}^{u}=\frac{(z_{i+1}-x_{w_{i+1}})z_{i}F'}{x_{w_{i+1}}}
 \end{equation*}
     for some integer $c$ and quantity $F'$ that does not depend on $z_i$ or $z_{i+1}$.
\end{enumerate}
\end{Lemma}

\begin{Lemma}\label{Lemma: fuw analysis3}
Let $u=(u_1,\ldots, u_n)$ and $w=(w_1,\ldots,w_n)$ be compositions 
	such that $u_i>u_{i+1}$ and $w_i=w_{i+1}$.  Then we have the following.
\begin{enumerate}
   \item If $w_i<max(w)$ then  $F^{u}_{w}=F^{s_i u}_{w}$ and both are symmetric in  $z_i$ and $z_{i+1}$.
    \item If $w_i=max(w)$ then either we have
\begin{equation*}
	F_{w}^{u}=F_w^{s_i u} = \frac{(z_{i}-x_c)(z_{i+1}-x_c)F'}{x^{2}_{c}} 
 \end{equation*}
for some integer $c$ and quantity $F'$ that does not depend on $z_i$ or $z_{i+1}$, OR
     \begin{equation*}
	     F_{w}^{u}=\frac{(z_{i}-x_{u_{i+1}})(z_{i+1}-x_{u_{i+1}})F'}{x^{2}_{u_{i+1}}} \text{ and } F_{w}^{s_i u}=\frac{(z_{i}-x_{u_{i+1}})(z_{i+1}-x_{c})F'}{x_{u_{i+1}}x_c}
 \end{equation*}
	for some integer $c$ and quantity 
 $F'$ that does not depend on $z_i$ or $z_{i+1}$.
\end{enumerate}
\end{Lemma}

The following lemma is obvious.
\begin{Lemma}\label{Lemma: fuw analysis4}
Let $u=(u_1,\ldots, u_n)$ and $w=(w_1,\ldots,w_n)$ be compositions 
such that $u_i=u_{i+1}$ and $w_i=w_{i+1}$.  Then 
$F_{w}^{u}$ is symmetric in variables $z_i$ and $z_{i+1}$.
\end{Lemma}

\cref{prop:almost} will be used together with \eqref{eq:F expand}
to prove
\cref{theorem: zweight}.

\begin{Prop}\label{prop:almost}
Let $u=(u_1,\ldots, u_n)$ and $w=(w_1,\ldots,w_n)$ be compositions, 
	and assume that 
	\eqref{eq:exchange1} 
	or \eqref{eq:exchange2} holds
	for $u$ (as appropriate). Then we have the following.

		If $w_i>w_{i+1}$ and $u_i>u_{i+1}$, we have
		\begin{equation} 
		(1-s_i)(F_{w}^{u}F_u(\z)+F_{w}^{s_i u}F_{s_i u}(\z)) = 
			\frac{x_{w_{i+1}} (z_i-z_{i+1})}{z_i (z_{i+1}-x_{w_{i+1}})} 
			(F_{s_i w}^{u}F_u(\z)+F_{s_i w}^{s_i u}F_{s_i u}(\z))
		\label{eq:case1}
		\end{equation}

	     If $w_i>w_{i+1}$ and $u_i=u_{i+1}$, we have
\begin{equation}
	(1-s_i)F_{w}^{u}F_u(\z) = 
	\frac{x_{w_{i+1}} (z_i-z_{i+1})}{z_i (z_{i+1}-x_{w_{i+1}})} 
	(F_{s_i w}^{u}F_u(\z)).
\label{eq:case2}\end{equation}

		If $w_i=w_{i+1}$ and $u_i>u_{i+1}$, we have
\begin{equation}
     (1-s_i) (F_{w}^{u}F_u(\z)+F_{w}^{s_i u}F_{s_i u}(\z))=0. \label{eq:case3}\end{equation}

	     If $w_i=w_{i+1}$ and $u_i=u_{i+1}$, we have
\begin{equation}
    (1-s_i)F_{w}^{u}F_u(\z)=0. \label{eq:case4}\end{equation}  
\end{Prop}
\begin{proof} We may assume that $max(w)=max(u)+1$ otherwise there is nothing to prove as $F_{w}^{u}=F_{w}^{s_i u}=F_{s_i w}^{u}=F_{s_i w}^{s_i u}=0$. We will prove 
		\eqref{eq:case1}
	 using \cref{Lemma: fuw analysis1}. The proofs of the other three equations  are similar using \cref{Lemma: fuw analysis2}, \cref{Lemma: fuw analysis3} and \cref{Lemma: fuw analysis4} respectively so we omit them.
	We divide the proof of \eqref{eq:case1} into six possible cases.

-Case 1: Suppose $w_i<max(w)$ and $u_{i+1}<w_{i+1}$. Then by \cref{Lemma: fuw analysis1} (1), we have $F_{w}^{u}=F_{s_i w}^{u}=F_{w}^{s_i u}=F_{s_i w}^{s_i u}=0$ so both sides of \eqref{eq:case1} are zero.

-Case 2: Suppose $w_i=max(w)$ and $u_{i+1}<w_{i+1}$. 
By \cref{Lemma: fuw analysis1} (2), 
		the right-hand side of \eqref{eq:case1} is zero.
Using 
 \cref{Lemma: fuw analysis1} (2)
	and \eqref{eq:exchange1}, the left-hand side of \eqref{eq:case1} is 
\begin{align}\label{eq: FWSSIU SYM}
    (1-s_i)(F_{w}^{s_i u}F_{s_i u}(\z))&=(1-s_i)\big(F' (z_i-x_{u_{i+1}})z_{i+1}
	\frac{z_i(z_{i+1}-x_{u_{i+1}})}{x_{u_{i+1}}}\frac{(1-s_i)F_u(\z)}{z_i-z_{i+1}}\big)\\&=(1-s_i)\big(F' z_i z_{i+1}\frac{(z_{i}-x_{u_{i+1}})(z_{i+1}-x_{u_{i+1}})}{x_{u_{i+1}}}\frac{(1-s_i)F_u(\z)}{z_i-z_{i+1}}\big),\nonumber 
\end{align}
	which is  zero 
	since we are applying 
	$(1-s_i)$ to 
	a quantity 
	symmetric in $z_i, z_{i+1}$.

-Case 3: Suppose $w_i<max(w)$ and $u_{i+1}=w_{i+1}$.  By \cref{Lemma: fuw analysis1} (3), \eqref{eq:case1} becomes
\begin{align*}
	F' z_i z_{i+1}F_{s_i u}(\z)&=\frac{z_i(z_{i+1}-x_{w_{i+1}})}{x_{w_{i+1}}}\frac{(1-s_i)(F' z_i z_{i+1}F_{u}(\z) )}{z_i-z_{i+1}}, \text{ or equivalently}\\   
     F_{s_i u}(\z)&=\frac{z_i(z_{i+1}-x_{w_{i+1}})}{x_{w_{i+1}}}\frac{(1-s_i)F_{u}(\z) }{z_i-z_{i+1}},
\end{align*}
which is \eqref{eq:exchange1} for $u$ and hence true by hypothesis.

-Case 4: Suppose $w_i<max(w)$ and $u_{i+1}>w_{i+1}$.  
The left-hand side of \eqref{eq:case1} is 
\begin{equation}\label{eq: fu fsiu sym}
	(1-s_i) (F_{w}^{u} F_u(\z)+F_{w}^{s_i u} F_{s_i u}(\z))=F' z_{i} z_{i+1}
	(1-s_i)(F_u(\z)+F_{s_i u}(\z)).
\end{equation}  
	By \cref{rem:symmetry}, 
	$(1-s_i)(F_u(\z)+F_{s_i u}(\z)) = 0$, so 
	the left-hand side is zero.
	And by \cref{Lemma: fuw analysis1} (4), the right-hand side of \eqref{eq:case1} is also zero. 

-Case 5: Suppose $w_i=max(w)$ and $u_{i+1}=w_{i+1}$. 
Then 
as shown in \eqref{eq: FWSSIU SYM}, 
we have $(1-s_i)(F_{w}^{s_i u}F_{s_i u}(\z))=0$. 
Therefore, multiplying the left-hand side of 
\eqref{eq:case1} by a constant, and using 
 \cref{Lemma: fuw analysis1} (5), and 
	\eqref{eq:exchange1},
we get 
\begin{align*}
	&\frac{z_i(z_{i+1}-x_{w_{i+1}})}{x_{w_{i+1}}}\frac{(1-s_i)(F_{w}^{u}F_u(\z)+F_{w}^{s_i u}F_{s_i u}(\z))}{z_i-z_{i+1}}
	=\frac{z_i(z_{i+1}-x_{w_{i+1}})}{x_{w_{i+1}}}\frac{(1-s_i)(F_{w}^{u}F_{ u}(\z))}{z_i-z_{i+1}}\\
	&=\frac{z_i(z_{i+1}-x_{w_{i+1}})}{x_{w_{i+1}}}\frac{(\frac{(z_i-x_c)z_{i+1}F'}{x_c}F_u(\z)-\frac{(z_{i+1}-x_c)z_{i}F'}{x_c}s_i F_u(\z))}{z_i-z_{i+1}} \\
	&=\frac{z_i(z_{i+1}-x_{w_{i+1}})}{x_{w_{i+1}}}\frac{(\frac{(z_{i+1}-x_c)z_{i}F'}{x_c}F_u(\z)-\frac{(z_{i+1}-x_c)z_{i}F'}{x_c}s_i F_u(\z))}{z_i-z_{i+1}}+\frac{(z_{i+1}-x_{w_{i+1}})z_{i}F'}{x_{w_{i+1}}}F_u(\z)\\
	&=\frac{(z_{i+1}-x_c)z_{i}F'}{x_c} \frac{z_i(z_{i+1}-x_{w_{i+1}})}{x_{w_{i+1}}}\frac{(1-s_i)F_u(\z)}{z_i-z_{i+1}}+\frac{(z_{i+1}-x_{w_{i+1}})z_{i}F'}{x_{w_{i+1}}}F_u(\z)\\
	&= \frac{(z_{i+1}-x_c)z_{i}F'}{x_c} F_{s_i u}(\z)+\frac{(z_{i+1}-x_{w_{i+1}})z_{i}F'}{x_{w_{i+1}}}F_u(\z)= F_{s_i w}^{u}F_u(\z)+F_{s_i w}^{s_i u}F_{s_i u}(\z).     
\end{align*}

-Case 6: Suppose $w_i=max(w)$ and $u_{i+1}>w_{i+1}$. Then by \cref{Lemma: fuw analysis1} (6), we have
\begin{align*}
 &\frac{z_i(z_{i+1}-x_{w_{i+1}})}{x_{w_{i+1}}}\frac{(1-s_i)(F_{w}^{u}F_u(\z)+F_{w}^{s_i u}F_{s_i u}(\z))}{z_i-z_{i+1}}\\&=\frac{z_i(z_{i+1}-x_{w_{i+1}})}{x_{w_{i+1}}}\frac{(1-s_i)(\frac{(z_i-x_c)z_{i+1}F'}{x_c}(F_u(\z)+F_{s_i u}(\z)))}{z_i-z_{i+1}}\\&=\frac{z_i(z_{i+1}-x_{w_{i+1}})}{x_{w_{i+1}}}   F'(F_u(\z)+F_{s_i u}(\z))=F_{s_i w}^{u}F_u(\z)+F_{s_i w}^{s_i u}F_{s_i u}(\z),
\end{align*}
where we used the fact that $(F_u(\z)+F_{s_i u}(\z))$ is symmetric in $z_i$ and $z_{i+1}$ (\cref{rem:symmetry}).
\end{proof}

Now we are ready to prove \cref{theorem: zweight}.

\begin{proof}[Proof of \cref{theorem: zweight}]
	For a composition $w=(w_1,\ldots,w_n)$, it is enough to show that $w$ satisfies \eqref{eq:exchange1} or \eqref{eq:exchange2}. We use induction on $max(w)$.

Consider the base case $max(w)=2$. There is only one multiline queue (consisting of one row) of type $w$, so we have
\begin{equation*}
    F_w(\z)=\prod\limits_{i, w_i=1}(z_i)\prod\limits_{i, w_i=2}(z_i-x_1).
\end{equation*}
If $w_i=w_{i+1}$,  $\eqref{eq:exchange2}$ is immediate. If $w_i>w_{i+1}$,   we can write $F_w(\z)=F' \cdot (z_i-x_1)z_{i+1}$ and $F_{s_i w}(\z)=F' z_i(z_{i+1}-x_1)$ for some $F'$ that does not depend on $z_i$ and $z_{i+1}$. Therefore the right-hand side of \eqref{eq:exchange1} becomes
\begin{align*}
    \frac{z_i (z_{i+1}-x_1)}{x_1}\frac{(1-s_i)(F' (z_i-x_1)z_{i+1})}{z_i-z_{i+1}}= \frac{z_i (z_{i+1}-x_1)}{x_1}F' x_1=F' z_i(z_{i+1}-x_1)
\end{align*}
which is $F_{s_i w}(\z)$.

Now consider the case $max(w)>2$ and assume that \eqref{eq:exchange1} and \eqref{eq:exchange2} are true for $u$ such that $max(u)=max(w)-1$. If $w_i>w_{i+1}$ then \eqref{eq:exchange1} follows from \eqref{eq:F expand}, \eqref{eq:case1} and \eqref{eq:case2}. If $w_i=w_{i+1}$ then \eqref{eq:exchange2} follows from \eqref{eq:F expand}, \eqref{eq:case3} and \eqref{eq:case4}.
\end{proof}

\section{The proof of \cref{MFC}, the monomial factor conjecture}\label{sec:monomial}

In this section we will prove \cref{MFC}, which gives  a formula for the 
largest monomial in $\x$ that divides $\psi_w$ in the setting where $y_i=0$.
We will first give an algebraic argument using the isobaric divided difference operators
to show that 
$\eta(w) = \prod_{i=1}^{n-2} x_i^{\alpha_i(w)+\cdots+\alpha_{n-2}(w)}$ always
divides $\psi_w$, see 
\cref{prop:divides}.
  We will then use 
our combinatorial formula for $\mathbf{z}$-deformed steady state probabilities
in terms of multiline queues 
(\cref{theorem: zweight}) to show that no greater monomial in $\x$
 divides $\psi_w$.

\begin{Def}
For a (multivariable) polynomial $p$ and a variable $x$, we write $x^d\mid\mid p$ if $x^d\mid p$ and $x^{d+1}\nmid p$, that is, 
$d$ is the highest power of $x$ that divides $p$. 
\end{Def}

Recall from \cref{def:alpha}
that for $w=(w_1,\ldots,w_n)\in \St(n)$, 
$\alpha_i(w)$ is the number of integers greater than $i+1$ 
among $\{i+1=w_r, w_{r+1}, \ldots , w_{s-1}, w_s=i\}$, where the subscripts
are taken modulo $n$.  
To prove \cref{MFC}, we will show that 
\begin{equation*}
    x_i^{\alpha_i(w)+\cdots+\alpha_{n-2}(w)} \mid\mid \psi_w, \hspace{5mm} \text{for $1\leq i\leq n-2$}.
\end{equation*}

\begin{Def}\label{permutation indexing}
Given an integer vector $(b_1,\ldots,b_{n-2})$ such that $0\leq b_i\leq i$, we define $S(b_1,\ldots,b_{n-2})\in\St(n)$ as follows.  
Let $w^{0}=(1,2,\ldots,n)\in\St(n)$.  Then for $1\leq i \leq n-2$, we recursively 
	construct from $w^{i-1}$ a state $w^i$ whose first $n-i-1$ letters
	are $1,2,\ldots, n-i-1$, by taking the last $b_i$ letters of $w^{i-1}$
	and inserting them after $1,2,\ldots, n-i-1$.
	That is, $$w^i = (1,2,\ldots,n-i-1, w^{i-1}_{n-b_i+1}, \ldots, w^{i-1}_{n-1},
	w^{i-1}_n, n-i, w_{n-i+1}^{i-1},\ldots, w_{n-b_i}^{i-1}).$$
We set $S(b_1,\ldots,b_{n-2})=w^{n-2}$.
\end{Def}
\begin{Ex}\label{ex:1}
Using \cref{permutation indexing}, 
we obtain $S(1,1,2,3) = w^4 \in \St(6)$ as follows. 
$$	\begin{array}{lll}
		w^0=(1,2,3,4,5,6) & 
	 w^1=(1,2,3,4,6,5)
	 & w^2=(1,2,3,5,4,6) \\
	 w^3 =(1,2, 4,6,3,5)
	& w^4=(1,6,3,5,2,4) & 
\end{array} $$
\end{Ex}
\begin{Rem}\label{rem:alpha}
It follows from \cref{permutation indexing} that $\alpha_{n-1-i}(w^{i})=\alpha_{n-1-i}(w^{i-1})-b_i$ and $\alpha_{j}(w^{i-1})=\alpha_{j}(w^{i})$ if $j\neq n-1-i$. 
\end{Rem}
\begin{Lemma}\label{lemma: construction of ST(N)}
Let $w=S(b_1,\ldots,b_{n-2})\in \St(n)$. Then we have
$$\alpha_i(w)=n-1-i-b_{n-1-i}.$$
	And for any $w\in\St(n)$, there exists  $(b_1,\ldots,b_{n-2})\in \Z^{n-2}$ such that $0\leq b_i\leq i$ and $w\sim S(b_1,\ldots,b_{n-2})$.
\end{Lemma}
\begin{proof}
Since	$\alpha_i(w^0)=n-1-i$, 
	 \cref{rem:alpha} implies that
	 $\alpha_i(w)=n-1-i-b_{n-1-i}.$

There are a total of $(n-1)!$ possible states $S(b_1,\ldots,b_{n-2})$, which are all cyclically different since the $\alpha_i(w)$'s are cyclically invariant. 
So they cover every element in $\St(n)$ up to a cyclic equivalence.
\end{proof}

\begin{Ex}
	Continuing \cref{ex:1}, 
	if $w:=S(1,1,2,3) = (1,6,3,5,2,4)$ then
$\alpha_1(w)=1$, $\alpha_2(w)=1$, $\alpha_3(w)=1$ and $\alpha_4(w)=0$, as  claimed in
 \cref{lemma: construction of ST(N)}.
\end{Ex}

\begin{Lemma}\label{lem:divisibility}
	If $\psi_{w'}(\z)=\pi_k(a,b) \psi_w(\z)$,
	and if a monomial $M$ in $x_1,\ldots,x_{n-1}$  divides $\psi_w(\z)$, then $\frac{M}{x_b}$ divides $\psi_{w'}(\z)$. 
\end{Lemma}
\begin{proof}
This follows from the fact that 
	$\pi_k(a,b)f(\z)=\frac{z_k(z_{k+1}-x_b)}{x_b}\frac{f(\z)-s_k f(\z)}{z_k-z_{k+1}}$, and each $\psi_w(\z)$ is a polynomial (\cref{CantiniProp}).
\end{proof}

\begin{Lemma}\label{lemma: pioperator monomial factor}
Consider $w, w'\in St(n)$ such that 
	$w=(1,\ldots, i, w_{i+1},\ldots, w_n)$
	and 
	$w'=(2,\ldots, i, w_n, w_{i+1},\ldots, w_{n-1}, 1)$
	for  $1 \leq i \leq n-1$.
	Let $w''=(1,\ldots i, w_n, w_{i+1},\ldots, w_{n-1})$.
	If a monomial $M$ in $x_1,\ldots, x_{n-1}$ divides 
	$\psi_w(\z)$, 
	 $\frac{M}{x_1\cdots x_i}$ divides 
	both $\psi_{w'}(\z)$ and 
	$\psi_{w''}(\z)$.
\end{Lemma}
\begin{proof}
Using \cref{lem:divisibility} and the fact that 
\begin{equation*}
	\psi_{w'}(\z)=\pi_{i-1}(w_{n},i)\cdots \pi_2(w_n,3) \pi_{1}(w_{n},2)\pi_{n}(w_{n},1)\psi_w(\z),
\end{equation*}
we conclude that $\frac{M}{x_1\cdots x_i}$ divides $\psi_{w'}(\z)$. 
	The fact that $\frac{M}{x_1\cdots x_i}$ divides $\psi_{w''}(\z)$
 follows from \cref{CantiniProp}
	since $w'$ and $w''$ are cyclically equivalent.
\end{proof}

\begin{Prop} \label{prop:divides}
For $w\in\St(n)$, we have
\begin{equation*}
    x_i^{\alpha_i(w)+\cdots+\alpha_{n-2}(w)} \mid \psi_w, \hspace{5mm} \text{for $1\leq i\leq n-2$}.
\end{equation*}
\end{Prop}
\begin{proof}
By \cref{CantiniProp}, 
it is enough to show this for any cyclic shift of $w$,
so we may assume $w=S(b_1,\ldots,b_{n-2})$ for some vector $(b_1,\ldots,b_{n-2})$. 
Using the notation in \cref{permutation indexing}, we will use induction on $i$ to show
that
\begin{equation}\label{whatever}
 x_1^{\alpha_1(w^{i})} (x_1x_2)^{\alpha_2(w^{i})} \cdots (x_1\cdots x_{n-2})^{\alpha_{n-2}(w^{i})} \mid \psi_{w^{i}}(\z)
\end{equation}
 
	Since $w^{0}$ is the identity permutation, we know from the expression
	for $\psi_{(1,2,\ldots,n)}(\z)$ in 
	\cref{CantiniProp} 
	 that $x_1^{n-2} (x_1x_2)^{n-3} \cdots (x_1\cdots x_{n-2})^{1}$ divides $\psi_{w^{0}}(\z)$. Now assume \eqref{whatever} holds for $i-1$. By repeatedly
	 applying \cref{lemma: pioperator monomial factor}, we have that 
	 $$\frac{x_1^{\alpha_1(w^{i-1})} (x_1x_2)^{\alpha_2(w^{i-1})} \cdots (x_1\cdots x_{n-2})^{\alpha_{n-2}(w^{i-1})}}{(x_1\cdots x_{n-1-i})^{b_i}}$$ 
	 divides $\psi_{{w}^{i}}(\z)$. Now 
	 \cref{rem:alpha} implies that
	 \begin{equation*}
    \frac{x_1^{\alpha_1(w^{i-1})} (x_1x_2)^{\alpha_2(w^{i-1})} \cdots (x_1\cdots x_{n-2})^{\alpha_{n-2}(w^{i-1})}}{(x_1\cdots x_{n-1-i})^{b_i}}= x_1^{\alpha_1(w^{i})} (x_1x_2)^{\alpha_2(w^{i})} \cdots (x_1\cdots x_{n-2})^{\alpha_{n-2}(w^{i})}.\end{equation*}
    Thus we conclude \eqref{whatever} is true for $i$. And \eqref{whatever} for $i=n-2$ gives the claim.
\end{proof}
Having proved \cref{prop:divides}, our next goal is to show that 
no larger monomial in $\x$ divides $\psi_w$.

\begin{Lemma}\label{lem: psipr}
If $x_i^{d}\mid\mid \psi_w$ then $x_i^{d} \mid\mid \psi_{w}(\z)$.
\end{Lemma}
\begin{proof}
If  $x_i^{d} \mid \psi_{w}(\z)$ then $x_i^{d}$ also divides $\psi_w=\LC(\psi_{w}(\z))$. For the converse, assume $x_i^{d} \mid \psi_{w}$. Then by \cref{thm:MLQ}, $x_i^{d} \mid \wt(Q)$ for every multiline queue $Q\in MLQ(w)$. Since $\wt(Q) \mid \zwt(Q)$ we have $x_i^{d} \mid \zwt(Q)$ for every $Q\in MLQ(w)$, which implies $x_i^{d} \mid \psi_{w}(\z)$ by \cref{theorem: zweight}.
\end{proof}
\begin{Lemma}\label{lem: MLQ MF}
Given a multiline queue $Q\in MLQ(w_1,\ldots,w_{n},1)$, 
	there exists  $\hat{Q}\in MLQ(w_1+1,\ldots,w_{n}+1,2,1)$ such that  $x_1^d \mid \mid \wt(Q)$ if and only if  $x_1^d \mid \mid \wt(\hat{Q})$. 
\end{Lemma}
\begin{proof}
Recall the notion of the type of a row of a multiline queue
from \cref{def:MLQ}.
We will construct $\hat{Q}$ so that if the $r$th row of $Q$ has type $(u_1,\ldots,u_n,1)$ (up to a cyclic shift), the $(r+1)$st row of $\hat{Q}$ has type $(u_1+1,\ldots,u_n+1,2,1)$ (up to a cyclic shift). 

Let $\hat{Q}^{0}$  be the two-row multiline queue with $n+2$ columns, 
whose
first and second rows have types $(2,\ldots,2,1,2,2)$ and $(3,\ldots,3,2,1)$. For $1 \leq i \leq n-1$, we inductively construct
$\hat{Q}^i$ as follows:
\begin{itemize}
	\item Cyclically shift the columns
		of $Q$ so that the leftmost entry in the $i$th row of the resulting
		multiline queue $Q'$
		is a ball labeled $1$.  Let $(v_1,\ldots,v_k, 1, v_{k+1},\ldots,v_n)$ be the type of the $(i+1)$st row of $Q'$.
		Add a new $(i+2)$nd row of type 
	 $(v_1+1,\ldots,v_k+1,2,1,v_{k+1}+1,\ldots,v_n+1)$ to $\hat{Q}^{i-1}$,
		 then cyclically shift the columns of the multiline queue
		 so that the leftmost entry in the  $(i+2)$nd row is 
	 a ball labeled $1$. Denote  the resulting multiline queue by $\hat{Q}^{i}$. 
\end{itemize}
We set $\hat{Q}=\hat{Q}^{n-1}$.
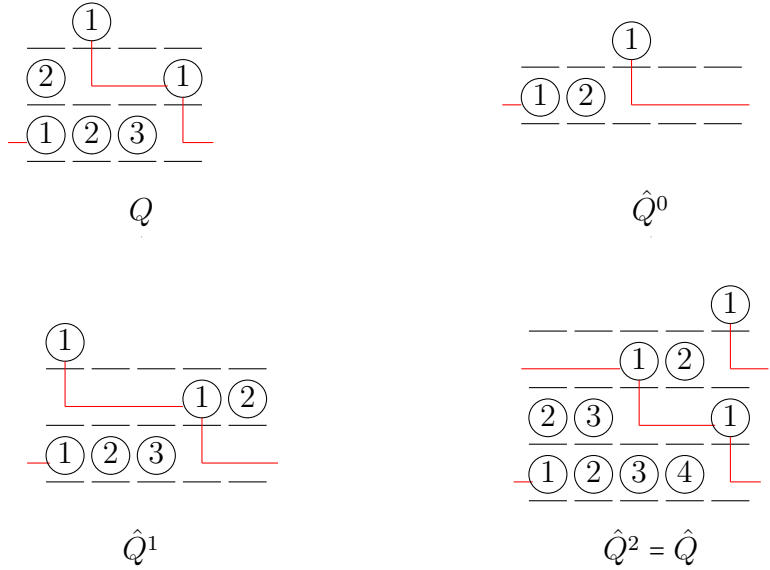
\begin{figure}[h]\centering
\begin{tikzpicture}[scale=0.5]
\draw[-] (-2.2,1) --(-3.2,1);
\draw[-] (-3.4,1) --(-4.4,1);
\draw[-] (-4.6,1) --(-5.6,1);
\draw[-] (-5.8,1) --(-6.8,1);

\begin{scope}[shift={(0,1.5)}]
\draw[-] (-2.2,1) --(-3.2,1);
\draw[-] (-3.4,1) --(-4.4,1);
\draw[-] (-4.6,1) --(-5.6,1);
\draw[-] (-5.8,1) --(-6.8,1);
\end{scope}

\begin{scope}[shift={(0,3)}]
\draw[-] (-2.2,1) --(-3.2,1);
\draw[-] (-3.4,1) --(-4.4,1);
\draw[-] (-4.6,1) --(-5.6,1);
\draw[-] (-5.8,1) --(-6.8,1);
\end{scope}

\begin{scope}[shift={(0,1.5)}]
\draw (-6.3,1.7) circle (0.5);
\filldraw[black] (-6.3,1.2) circle (0.000001pt) node[anchor=south] {$2$};
\draw (-2.7,1.7) circle (0.5);
\filldraw[black] (-2.7,1.2) circle (0.000001pt) node[anchor=south] {$1$};
\draw[red] (-2.7,1.2)--(-2.7,0);
\draw[red] (-2.7,0)--(-1.9,0);
\draw[red] (-7.3,0)--(-6.8,0);
\end{scope}

\begin{scope}[shift={(0,3)}]
\draw (-5.1,1.7) circle (0.5);
\filldraw[black] (-5.1,1.2) circle (0.000001pt) node[anchor=south] {$1$};
\draw[red] (-5.1,1.2)--(-5.1,0);
\draw[red] (-5.1,0)--(-3.1,0);

\end{scope}

\draw (-6.3,1.7) circle (0.5);
\filldraw[black] (-6.3,1.2) circle (0.000001pt) node[anchor=south] {$1$};

\draw (-3.9,1.7) circle (0.5);
\filldraw[black] (-3.9,1.2) circle (0.000001pt) node[anchor=south] {$3$};
\draw (-5.1,1.7) circle (0.5);
\filldraw[black] (-5.1,1.2) circle (0.000001pt) node[anchor=south] {$2$};

\begin{scope}[shift={(13,-2)}]

\begin{scope}[shift={(0,3)}]
\draw[-] (-1,1) --(-2,1);
\draw[-] (-2.2,1) --(-3.2,1);
\draw[-] (-3.4,1) --(-4.4,1);
\draw[-] (-4.6,1) --(-5.6,1);
\draw[-] (-5.8,1) --(-6.8,1);
\end{scope}

\begin{scope}[shift={(0,4.5)}]
\draw[-] (-1,1) --(-2,1);
\draw[-] (-2.2,1) --(-3.2,1);
\draw[-] (-3.4,1) --(-4.4,1);
\draw[-] (-4.6,1) --(-5.6,1);
\draw[-] (-5.8,1) --(-6.8,1);
\end{scope}

\begin{scope}[shift={(0,3)}]
\draw (-5.1,1.7) circle (0.5);
\filldraw[black] (-5.1,1.2) circle (0.000001pt) node[anchor=south] {$2$};
\draw (-6.3,1.7) circle (0.5);
\filldraw[black] (-6.3,1.2) circle (0.000001pt) node[anchor=south] {$1$};
\end{scope}

\begin{scope}[shift={(0,4.5)}]
\draw (-3.9,1.7) circle (0.5);
\filldraw[black] (-3.9,1.2) circle (0.000001pt) node[anchor=south] {$1$};
\draw[red] (-3.9,1.2)--(-3.9,0);
\draw[red] (-3.9,0)--(-0.8,0);
\draw[red] (-7.3,0)--(-6.8,0);
\end{scope}

\end{scope}

\begin{scope}[shift={(0.5,-10)}]

\begin{scope}[shift={(0,1.5)}]
\draw[-] (-1,1) --(-2,1);
\draw[-] (-2.2,1) --(-3.2,1);
\draw[-] (-3.4,1) --(-4.4,1);
\draw[-] (-4.6,1) --(-5.6,1);
\draw[-] (-5.8,1) --(-6.8,1);
\end{scope}

\begin{scope}[shift={(0,3)}]
\draw[-] (-1,1) --(-2,1);
\draw[-] (-2.2,1) --(-3.2,1);
\draw[-] (-3.4,1) --(-4.4,1);
\draw[-] (-4.6,1) --(-5.6,1);
\draw[-] (-5.8,1) --(-6.8,1);
\end{scope}

\begin{scope}[shift={(0,4.5)}]
\draw[-] (-1,1) --(-2,1);
\draw[-] (-2.2,1) --(-3.2,1);
\draw[-] (-3.4,1) --(-4.4,1);
\draw[-] (-4.6,1) --(-5.6,1);
\draw[-] (-5.8,1) --(-6.8,1);
\end{scope}

\begin{scope}[shift={(0,3)}]
\draw (-1.5,1.7) circle (0.5);
\filldraw[black] (-1.5,1.2) circle (0.000001pt) node[anchor=south] {$2$};
\draw (-2.7,1.7) circle (0.5);
\filldraw[black] (-2.7,1.2) circle (0.000001pt) node[anchor=south] {$1$};
\draw[red] (-2.7,1.2)--(-2.7,0);
\draw[red] (-2.7,0)--(-0.7,0);
\draw[red] (-7.3,0)--(-6.7,0);
\end{scope}

\begin{scope}[shift={(0,4.5)}]
\draw (-6.3,1.7) circle (0.5);
\filldraw[black] (-6.3,1.2) circle (0.000001pt) node[anchor=south] {$1$};
\draw[red] (-6.3,1.2)--(-6.3,0);
\draw[red] (-6.3,0)--(-3.2,0);
\end{scope}

\begin{scope}[shift={(0,1.5)}]
\draw (-6.3,1.7) circle (0.5);
\filldraw[black] (-6.3,1.2) circle (0.000001pt) node[anchor=south] {$1$};
\draw (-5.1,1.7) circle (0.5);
\filldraw[black] (-5.1,1.2) circle (0.000001pt) node[anchor=south] {$2$};
\draw (-3.9,1.7) circle (0.5);
\filldraw[black] (-3.9,1.2) circle (0.000001pt) node[anchor=south] {$3$};

\end{scope}

\end{scope}

\begin{scope}[shift={(12,-9)}]
\draw[-] (0.2,1) --(-0.8,1);
\draw[-] (-1,1) --(-2,1);
\draw[-] (-2.2,1) --(-3.2,1);
\draw[-] (-3.4,1) --(-4.4,1);
\draw[-] (-4.6,1) --(-5.6,1);

\begin{scope}[shift={(0,1.5)}]
\draw[-] (0.2,1) --(-0.8,1);
\draw[-] (-1,1) --(-2,1);
\draw[-] (-2.2,1) --(-3.2,1);
\draw[-] (-3.4,1) --(-4.4,1);
\draw[-] (-4.6,1) --(-5.6,1);
\end{scope}

\begin{scope}[shift={(0,3)}]
\draw[-] (0.2,1) --(-0.8,1);
\draw[-] (-1,1) --(-2,1);
\draw[-] (-2.2,1) --(-3.2,1);
\draw[-] (-3.4,1) --(-4.4,1);
\draw[-] (-4.6,1) --(-5.6,1);
\end{scope}

\begin{scope}[shift={(0,4.5)}]
\draw[-] (0.2,1) --(-0.8,1);
\draw[-] (-1,1) --(-2,1);
\draw[-] (-2.2,1) --(-3.2,1);
\draw[-] (-3.4,1) --(-4.4,1);
\draw[-] (-4.6,1) --(-5.6,1);
\end{scope}

\begin{scope}[shift={(0,3)}]
\draw (-1.5,1.7) circle (0.5);
\filldraw[black] (-1.5,1.2) circle (0.000001pt) node[anchor=south] {$2$};
\draw (-2.7,1.7) circle (0.5);
\filldraw[black] (-2.7,1.2) circle (0.000001pt) node[anchor=south] {$1$};
\draw[red] (-2.7,1.2)--(-2.7,0);
\draw[red] (-2.7,0)--(-0.7,0);

\end{scope}

\begin{scope}[shift={(0,4.5)}]
\draw (-0.3,1.7) circle (0.5);
\filldraw[black] (-0.3,1.2) circle (0.000001pt) node[anchor=south] {$1$};
\draw[red] (-0.3,1.2)--(-0.3,0);
\draw[red] (-0.3,0)--(0.7,0);
\draw[red] (-5.8,0)--(-3.2,0);
\end{scope}

\begin{scope}[shift={(0,1.5)}]
\draw (-0.3,1.7) circle (0.5);
\filldraw[black] (-0.3,1.2) circle (0.000001pt) node[anchor=south] {$1$};
\draw (-5.1,1.7) circle (0.5);
\filldraw[black] (-5.1,1.2) circle (0.000001pt) node[anchor=south] {$2$};
\draw (-3.9,1.7) circle (0.5);
\filldraw[black] (-3.9,1.2) circle (0.000001pt) node[anchor=south] {$3$};
\draw[red] (-0.3,1.2)--(-0.3,0);
\draw[red] (-0.3,0)--(0.5,0);
\draw[red] (-6,0)--(-5.5,0);
\end{scope}

\draw (-1.5,1.7) circle (0.5);
\filldraw[black] (-1.5,1.2) circle (0.000001pt) node[anchor=south] {$4$};
\draw (-2.7,1.7) circle (0.5);
\filldraw[black] (-2.7,1.2) circle (0.000001pt) node[anchor=south] {$3$};
\draw (-3.9,1.7) circle (0.5);
\filldraw[black] (-3.9,1.2) circle (0.000001pt) node[anchor=south] {$2$};
\draw (-5.1,1.7) circle (0.5);
\filldraw[black] (-5.1,1.2) circle (0.000001pt) node[anchor=south] {$1$};
\end{scope}

\filldraw[black] (-3.8,-1) circle (0.000001pt) node[anchor=south] {$Q$};
\filldraw[black] (9.6,-1) circle (0.000001pt) node[anchor=south] {$\hat{Q}^{0}$};
\filldraw[black] (-3.8,-10) circle (0.000001pt) node[anchor=south] {$\hat{Q}^{1}$};

\filldraw[black] (9.6,-10) circle (0.000001pt) node[anchor=south] {$\hat{Q}^{2}=\hat{Q}$};

\end{tikzpicture}
\caption{An example of the construction in \cref{lem: MLQ MF} with $n=3$.} \label{fig7}
\end{figure}

By \cref{def: weight of a multiline queue}, the exponent of 
$x_1$ in $\wt(Q)$ is the number of vacancies in rows $2,\ldots,n$ minus
the number of $1$-covered vacancies  in rows $2,\ldots,n$.
By construction, the number of vacancies and $1$-covered vacancies in
the $r$th row of $Q$ agrees with the corresponding number of vacancies
in the $(r+1)$st row of $\hat{Q}$ for $r \geq 2$.
All  the vacancies in row $2$ of $\hat{Q}$ are $1$-covered, so 
the exponents of $x_1$ in $\wt(Q)$ and $\wt(\hat{Q})$ are the same.  
\end{proof}
\begin{Ex}
 \cref{fig7} shows a multiline queue $Q$ of type $(4,3,2,1)$ as well
	as $\hat{Q}^0$, $\hat{Q}^1$, and $\hat{Q}^2=\hat{Q}$, as constructed in  the proof of \cref{lem: MLQ MF}. We have
\begin{align*}
    &\wt(Q)=(x_1)^{1+2}x_2^{1}(\frac{x_2}{x_1})^2(\frac{x_3}{x_1})=x_2^{3}x_3\\
    &\wt(\hat{Q})=(x_1)^{1+2+3}x_2^{1+2}x_3^{1}(\frac{x_2}{x_1})^3(\frac{x_3}{x_1})^2(\frac{x_4}{x_1})=x_2^{6}x_3^{3}x_4.
\end{align*}
Note that the exponents of $x_1$ in $\wt(Q)$ and $\wt(\hat{Q})$ are both zero.
\end{Ex}

\begin{Prop} \label{prop: x1}
For the state $w\in \St(n)$, we have
\begin{equation*}
    x_1^{\alpha_1(w)+\cdots+\alpha_{n-2}(w)} \mid\mid \psi_{w}.
\end{equation*}
\end{Prop}
\begin{proof}
By \cref{prop:divides}, we know that 
    $x_1^{\alpha_1(w)+\cdots+\alpha_{n-2}(w)} \mid \psi_{w}.$
So we need to show that no greater power of $x_1$ divides $\psi_w$.
We use induction, and suppose this is true for states in $\St(n)$.
For the sake of contradiction,
suppose there exists a state $w\in\St(n+1)$ such that  $x_1^{\alpha_1(w)+\cdots+\alpha_{n-1}(w)+1} \mid \psi_{w}$. By \cref{lem: psipr}, $x_1^{\alpha_1(w)+\cdots+\alpha_{n-1}(w)+1} \mid \psi_{w}(\z)$. Cyclically shifting the state if necessary, we can assume the last component of $w$ is a $1$; 
	so write $w=(w_1,\ldots,w_{n-k-1},2,w_{n-k},\ldots,w_{n-1},1).$\footnote{It is possible that $w=(w_1,\ldots,w_{n-1},2,1)$, in which case the following
	 argument simplifies: there are no $\pi$ operators in the following 
	 line, and we take $w'=w$.} 
	 Then we have
\begin{equation*}
     \psi_{(w_1,\ldots,w_{n-k-1},2,1,w_{n-k},\ldots,w_{n-1})}(\z)=\pi_{n-k+1}(w_{n-k},1)\cdots\pi_{n}(w_{n-1},1)\psi_w(\z),
\end{equation*}
which implies 
 that 
	$x_1^{\alpha_1(w)+\cdots+\alpha_{n-1}(w)+1-k} \mid \psi_{(w_1,\ldots,w_{n-k-1},2,1,w_{n-k},\ldots,w_{n-1})}(\z)$ 
	(by \cref{lem:divisibility}).

Applying a cyclic shift, we let 
	$w'=(w_{n-k},\ldots,w_{n-1},w_1,\ldots,w_{n-k-1},2,1);$
we also have $x_1^{\alpha_1(w)+\cdots+\alpha_{n-1}(w)+1-k} \mid \psi_{w'}$. 
	Since $\alpha_1(w')=\alpha_1(w)-k$ and $\alpha_2(w')=\alpha_2(w),\ldots,\alpha_{n-1}(w')=\alpha_{n-1}(w)$, we have 
	\begin{equation}\label{eq:s}x_1^{a_1(w')+\cdots+a_{n-1}(w')+1} \mid \psi_{w'} \text{ for }
		w'=(w_{n-k},\ldots,w_{n-1},w_1,\ldots,w_{n-k-1},2,1) \in \St(n+1).
	\end{equation} 
		
		Consider the state $w''=(w_{n-k}-1,\ldots,w_{n-1}-1,w_1-1,\ldots,w_{n-k-1}-1,1) \in \St(n)$. 
By the induction hypothesis, 
	$x_1^{\alpha_1(w'')+\cdots +\alpha_{n-2}(w'')}\mid\mid \psi_{w''},$ so
	there exists a multiline queue $Q$ of type $w''$ such that 
	$$x_1^{\alpha_1(w'')+\cdots +\alpha_{n-2}(w'')}\mid\mid \wt(Q).$$
By  \cref{lem: MLQ MF}, there is a multiline queue $Q'$ of type $w'$ such that $$x_1^{\alpha_1(w'')+\cdots+\alpha_{n-2}(w'')}\mid\mid \wt(Q').$$
Since $\alpha_1(w')=0$ and $\alpha_2(w')=\alpha_1(w''),\ldots,\alpha_{n-1}(w')=\alpha_{n-2}(w'')$, we have $$x_1^{\alpha_1(w')+\cdots+\alpha_{n-1}(w')}\mid\mid \wt(Q').$$ This contradicts  \eqref{eq:s}.
\end{proof}

\begin{Lemma}\label{lemma: trivial bully path insertion}
For $w\in\St(n)$ and  $1\leq i\leq n-3$, let  $w'\in \St(n-i)$ be obtained by erasing $1,\ldots,i$ in $w$ and decreasing the other integers by $i$. Given $Q'\in MLQ(w')$, there exists $Q\in MLQ(w)$ such that  $x_1^d \mid \mid \wt(Q')$ if
	and only if $x_{i+1}^d \mid \mid \wt(Q)$.
\end{Lemma}
\begin{proof}
We increase the label of each ball in $Q'$ by $i$, and insert trivial bully paths
	of type $1,\ldots,i$ in the appropriate positions to get a multiline queue of type $w$.  
\end{proof}
\begin{Ex}
 Let $w=(1,4,2,7,6,5,3)\in\St(7)$ and $i=3$.  Then $w'=(1,4,3,2)$.
	 \cref{fig: inserting trivial bully path} shows $Q'\in MLQ(w')$ on the left
	 and $Q$ on the right.
	 The weights are
 \begin{align*}
     \wt(Q')&=x_1^{1+2}x_2^{1}(\frac{x_2}{x_1})^2(\frac{x_3}{x_1})=x_2^{3}x_3\\
     \wt(Q)&=x_1^{1+2+3+4+5}x_2^{1+2+3+4}x_3^{1+2+3}x_4^{1+2}x_5^{1}(\frac{x_5}{x_4})^2(\frac{x_6}{x_4})=x_1^{15}x_2^{10}x_3^{6}x_5^{3}x_6.
 \end{align*}
The exponent of $x_1$ in $\wt(Q')$ and the exponent of $x_4$ in $\wt(Q)$ are both zero.  
\end{Ex}
\vspace{-.6cm}
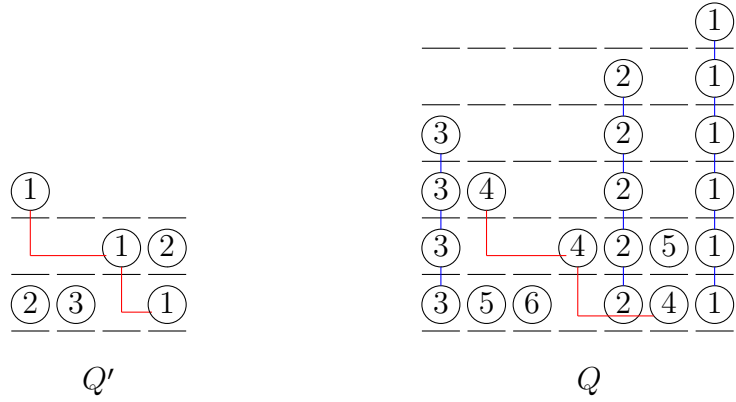
\begin{figure}[h]\centering
\begin{tikzpicture}[scale=0.5]
\draw[-] (-1,1) --(-2,1);
\draw[-] (-2.2,1) --(-3.2,1);
\draw[-] (-3.4,1) --(-4.4,1);
\draw[-] (-4.6,1) --(-5.6,1);

\begin{scope}[shift={(0,1.5)}]
\draw[-] (-1,1) --(-2,1);
\draw[-] (-2.2,1) --(-3.2,1);
\draw[-] (-3.4,1) --(-4.4,1);
\draw[-] (-4.6,1) --(-5.6,1);
\end{scope}

\begin{scope}[shift={(0,3)}]
\draw[-] (-1,1) --(-2,1);
\draw[-] (-2.2,1) --(-3.2,1);
\draw[-] (-3.4,1) --(-4.4,1);
\draw[-] (-4.6,1) --(-5.6,1);
\end{scope}

\begin{scope}[shift={(0,1.5)}]
\draw (-1.5,1.7) circle (0.5);
\filldraw[black] (-1.5,1.2) circle (0.000001pt) node[anchor=south] {$2$};
\draw (-2.7,1.7) circle (0.5);
\filldraw[black] (-2.7,1.2) circle (0.000001pt) node[anchor=south] {$1$};
\draw[red] (-2.7,1.2)--(-2.7,0);
\draw[red] (-2.7,0)--(-1.9,0);
\end{scope}

\begin{scope}[shift={(0,3)}]
\draw (-5.1,1.7) circle (0.5);
\filldraw[black] (-5.1,1.2) circle (0.000001pt) node[anchor=south] {$1$};
\draw[red] (-5.1,1.2)--(-5.1,0);
\draw[red] (-5.1,0)--(-3.1,0);
\end{scope}

\draw (-1.5,1.7) circle (0.5);
\filldraw[black] (-1.5,1.2) circle (0.000001pt) node[anchor=south] {$1$};

\draw (-3.9,1.7) circle (0.5);
\filldraw[black] (-3.9,1.2) circle (0.000001pt) node[anchor=south] {$3$};
\draw (-5.1,1.7) circle (0.5);
\filldraw[black] (-5.1,1.2) circle (0.000001pt) node[anchor=south] {$2$};

\filldraw[black] (-3.3,-1) circle (0.000001pt) node[anchor=south] {$Q'$};
\filldraw[black] (9.6,-1) circle (0.000001pt) node[anchor=south] {$Q$};

\begin{scope}[shift={(12,0)}]
\draw[-] (1.4,1) --(0.4,1);
\draw[-] (0.2,1) --(-0.8,1);
\draw[-] (-1,1) --(-2,1);
\draw[-] (-2.2,1) --(-3.2,1);
\draw[-] (-3.4,1) --(-4.4,1);
\draw[-] (-4.6,1) --(-5.6,1);
\draw[-] (-5.8,1) --(-6.8,1);

\begin{scope}[shift={(0,1.5)}]
\draw[-] (1.4,1) --(0.4,1);
\draw[-] (0.2,1) --(-0.8,1);
\draw[-] (-1,1) --(-2,1);
\draw[-] (-2.2,1) --(-3.2,1);
\draw[-] (-3.4,1) --(-4.4,1);
\draw[-] (-4.6,1) --(-5.6,1);
\draw[-] (-5.8,1) --(-6.8,1);
\end{scope}

\begin{scope}[shift={(0,3)}]
\draw[-] (1.4,1) --(0.4,1);
\draw[-] (0.2,1) --(-0.8,1);
\draw[-] (-1,1) --(-2,1);
\draw[-] (-2.2,1) --(-3.2,1);
\draw[-] (-3.4,1) --(-4.4,1);
\draw[-] (-4.6,1) --(-5.6,1);
\draw[-] (-5.8,1) --(-6.8,1);
\end{scope}

\begin{scope}[shift={(0,4.5)}]
\draw[-] (1.4,1) --(0.4,1);
\draw[-] (0.2,1) --(-0.8,1);
\draw[-] (-1,1) --(-2,1);
\draw[-] (-2.2,1) --(-3.2,1);
\draw[-] (-3.4,1) --(-4.4,1);
\draw[-] (-4.6,1) --(-5.6,1);
\draw[-] (-5.8,1) --(-6.8,1);
\end{scope}

\begin{scope}[shift={(0,6)}]
\draw[-] (1.4,1) --(0.4,1);
\draw[-] (0.2,1) --(-0.8,1);
\draw[-] (-1,1) --(-2,1);
\draw[-] (-2.2,1) --(-3.2,1);
\draw[-] (-3.4,1) --(-4.4,1);
\draw[-] (-4.6,1) --(-5.6,1);
\draw[-] (-5.8,1) --(-6.8,1);
\end{scope}

\begin{scope}[shift={(0,7.5)}]
\draw[-] (1.4,1) --(0.4,1);
\draw[-] (0.2,1) --(-0.8,1);
\draw[-] (-1,1) --(-2,1);
\draw[-] (-2.2,1) --(-3.2,1);
\draw[-] (-3.4,1) --(-4.4,1);
\draw[-] (-4.6,1) --(-5.6,1);
\draw[-] (-5.8,1) --(-6.8,1);
\end{scope}

\draw (0.9,1.7) circle (0.5);
\filldraw[black] (0.9,1.2) circle (0.000001pt) node[anchor=south] {$1$};
\draw (-1.5,1.7) circle (0.5);
\filldraw[black] (-1.5,1.2) circle (0.000001pt) node[anchor=south] {$2$};
\draw (-0.3,1.7) circle (0.5);
\filldraw[black] (-0.3,1.2) circle (0.000001pt) node[anchor=south] {$4$};
\draw (-3.9,1.7) circle (0.5);
\filldraw[black] (-3.9,1.2) circle (0.000001pt) node[anchor=south] {$6$};
\draw (-5.1,1.7) circle (0.5);
\filldraw[black] (-5.1,1.2) circle (0.000001pt) node[anchor=south] {$5$};
\draw (-6.3,1.7) circle (0.5);
\filldraw[black] (-6.3,1.2) circle (0.000001pt) node[anchor=south] {$3$};

\begin{scope}[shift={(0,1.5)}]
\draw (0.9,1.7) circle (0.5);
\filldraw[black] (0.9,1.2) circle (0.000001pt) node[anchor=south] {$1$};
\draw (-1.5,1.7) circle (0.5);
\filldraw[black] (-1.5,1.2) circle (0.000001pt) node[anchor=south] {$2$};
\draw (-6.3,1.7) circle (0.5);
\filldraw[black] (-6.3,1.2) circle (0.000001pt) node[anchor=south] {$3$};
\draw (-0.3,1.7) circle (0.5);
\filldraw[black] (-0.3,1.2) circle (0.000001pt) node[anchor=south] {$5$};
\draw (-2.7,1.7) circle (0.5);
\filldraw[black] (-2.7,1.2) circle (0.000001pt) node[anchor=south] {$4$};

\draw[blue] (-6.3,1.2)--(-6.3,0.7);
\draw[blue] (-1.5,1.2)--(-1.5,0.7);
\draw[blue] (0.9,1.2)--(0.9,0.7);
\draw[red] (-2.7,1.2)--(-2.7,-0.1);
\draw[red] (-2.7,-0.1)--(-0.7,-0.1);
\end{scope}

\begin{scope}[shift={(0,3)}]
\draw (0.9,1.7) circle (0.5);
\filldraw[black] (0.9,1.2) circle (0.000001pt) node[anchor=south] {$1$};
\draw (-1.5,1.7) circle (0.5);
\filldraw[black] (-1.5,1.2) circle (0.000001pt) node[anchor=south] {$2$};
\draw (-6.3,1.7) circle (0.5);
\filldraw[black] (-6.3,1.2) circle (0.000001pt) node[anchor=south] {$3$};
\draw (-5.1,1.7) circle (0.5);
\filldraw[black] (-5.1,1.2) circle (0.000001pt) node[anchor=south] {$4$};

\draw[blue] (-6.3,1.2)--(-6.3,0.7);
\draw[blue] (-1.5,1.2)--(-1.5,0.7);
\draw[blue] (0.9,1.2)--(0.9,0.7);
\draw[red] (-5.1,1.2)--(-5.1,0);
\draw[red] (-5.1,0)--(-3,0);
\end{scope}

\begin{scope}[shift={(0,4.5)}]
\draw (0.9,1.7) circle (0.5);
\filldraw[black] (0.9,1.2) circle (0.000001pt) node[anchor=south] {$1$};
\draw (-6.3,1.7) circle (0.5);
\filldraw[black] (-6.3,1.2) circle (0.000001pt) node[anchor=south] {$3$};
\draw (-1.5,1.7) circle (0.5);
\filldraw[black] (-1.5,1.2) circle (0.000001pt) node[anchor=south] {$2$};

\draw[blue] (-6.3,1.2)--(-6.3,0.7);
\draw[blue] (-1.5,1.2)--(-1.5,0.7);
\draw[blue] (0.9,1.2)--(0.9,0.7);
\end{scope}

\begin{scope}[shift={(0,6)}]
\draw (0.9,1.7) circle (0.5);
\filldraw[black] (0.9,1.2) circle (0.000001pt) node[anchor=south] {$1$};
\draw (-1.5,1.7) circle (0.5);
\filldraw[black] (-1.5,1.2) circle (0.000001pt) node[anchor=south] {$2$};

\draw[blue] (-1.5,1.2)--(-1.5,0.7);
\draw[blue] (0.9,1.2)--(0.9,0.7);
\end{scope}

\begin{scope}[shift={(0,7.5)}]
\draw (0.9,1.7) circle (0.5);
\filldraw[black] (0.9,1.2) circle (0.000001pt) node[anchor=south] {$1$};

\draw[blue] (0.9,1.2)--(0.9,0.7);
\end{scope}
\end{scope}
\end{tikzpicture}
\caption{An example of \cref{lemma: trivial bully path insertion}, with 
$w=(1,4,2,7,6,5,3)$, $i=3$, and $w'=(1,4,3,2)$.} \label{fig: inserting trivial bully path}
\end{figure}
\vspace{-.3cm}
\textit{Proof of \cref{MFC}.} For $w\in \St(n)$ and a number $1\leq i\leq n-3$, consider the state $w'\in\St(n-i)$ obtained by erasing $1,\ldots,i$ in $w$ and decreasing other integers by $i$. By \cref{prop: x1}, 
    $x_1^{\alpha_1(w')+\cdots+\alpha_{n-i-2}(w')} \mid\mid \psi_{w'}$, hence by 
\cref{thm:MLQ}, 
there exists a multiline queue $Q'$ of type $w'$ such that $$x_1^{\alpha_1(w')+\cdots+\alpha_{n-i-2}(w')}\mid\mid \wt(Q').$$
By \cref{lemma: trivial bully path insertion}, there exists $Q\in MLQ(w)$ such that 
\begin{equation}\label{eq:divides1}
x_{i+1}^{\alpha_{1}(w')+\cdots+\alpha_{n-i-2}(w')}\mid\mid \wt(Q).
\end{equation}
By the construction of $w'$ we have 
	\begin{equation*}
		\alpha_{i+1}(w)=\alpha_1(w'), \hspace{.4cm}
		\alpha_{i+2}(w)=\alpha_2(w'), \hspace{.4cm}
		\ldots \hspace{.4cm}
    \alpha_{n-2}(w)=\alpha_{n-i-2}(w'), 
\end{equation*}
	so \eqref{eq:divides1} implies
\begin{equation*} 
	x_{i+1}^{\alpha_{i+1}(w)+\cdots+\alpha_{n-2}(w)} \mid\mid \wt(Q).
\end{equation*} 
By 
\cref{prop:divides} we have that
$$ x_{i+1}^{\alpha_{i+1}(w)+\cdots+\alpha_{n-2}(w)} \mid\mid \psi_w$$
for $1\leq i\leq n-3$. Combining this with \cref{prop: x1}, we conclude that 
 the largest monomial that can be factored out of  $\psi_w$
 is $\prod\limits_{i=1}^{n-2}x_i^{\alpha_i(w)+\cdots+\alpha_{n-2}(w)}$, as
 desired.  Recall from \cref{lemma: construction of ST(N)} that if  two states $w,w'\in\St(n)$ have the same $\alpha_i(w)=\alpha_i(w')$ for all $1\leq i\leq n-2$ then $w \sim w'$.   This completes the proof.
\qed

\section{Future questions}\label{sec:future}

There are various natural questions that arise from this work.
\begin{Prob}
Find a combinatorial formula for steady state probabilities of the 
inhomogeneous TASEP 
(e.g. using multiline queues)  in the case where $y_i \neq 0$.
\end{Prob}

\begin{Prob}
Find a purely combinatorial proof of \cref{MFC}.
\end{Prob}

\begin{Prob}\label{prob:geom}
Prove that for any permutation $w$, 
the steady state probability $\psi_w$ in the inhomogeneous
TASEP can be written
as a positive sum of double
Schubert polynomials $\mathfrak{S}_{w}(x_1,\ldots,x_n; y_1,\ldots,y_n)$
multiplied by some  linear factors $(x_i-y_j)$.
\end{Prob}
\cref{prob:geom} is open even 
when $y_i=0$ for all $i$.
One  way to approach this problem would be to find a 
{geometric interpretation}
of each steady state probability.
Recall that Schubert polynomials represent cohomology 
classes of Schubert varieties
in the complete flag variety. 
Can one 
associate a variety to each
state whose cohomology class is represented by the corresponding
(unnormalized) steady state probability?

\appendix
\section{Technical results for the proof of \cref{thm:maintechnical}}\label{sec:technical}
In this section, we collect some technical results (\cref{prop 115} and \cref{prop 116}) which are used
for the proof of \cref{thm:maintechnical}.

\begin{Lemma}\cite[Lemma 22]{C}\label{CantiniLemma}
Let $K(x_1;x_2,\ldots,x_{m+1})$ be a rational expression in $x_1,\ldots, x_{m+1}$ which is symmetric in the variables $x_2,\ldots,x_{m+1}$. We have
\begin{equation*}
    \partial_m\cdots\partial_1 K=\sum_{i=1}^{m+1}\frac{K(x_i;x_1,\ldots,\hat{x_{i}},\ldots,x_{m+1})}{\prod\limits_{\substack{j=1\\j\neq i}}^{m+1}(x_i-x_j)}.
\end{equation*}
\end{Lemma}

Now we rewrite \eqref{eq: def zschubert}. In what follows,
we let $[m]$ denote $\{1,2,\ldots,m\}$.
\begin{Lemma}
	\label{eq:zschubert def sym}
For $\lambda\in\Val(n)$, let $\mul(\lambda)\geq k$ for some $k$. Let $\tilde{\lambda}$ denote the partition obtained by deleting the first $k$ parts of $\lambda$. Then 
    $\Sym_{\lambda}^{n}(\z;\x;\y)$ is equal to 
\begin{equation}
	\sum\limits_{\substack{I\subseteq[n-\lambda_1-\mul(\lambda)+k] \\|I|=k}}
	\frac{ \displaystyle{ \Sym_{\tilde{\lambda}}^{n-k}(\sigma^{\lambda_1-\lambda_{k+1}+k}(\z);\x_{\hat{I}};\y)\prod\limits_{\substack{1\leq l\leq n-\mul(\lambda)\\i\in I}}(x_i-y_l)\prod\limits_{\substack{1\leq i\leq \lambda_1-\lambda_{k+1}+k\\j \in [n-\lambda_1-\mul(\lambda)+k]\setminus I }}(z_i-x_j)}  }
	{ \displaystyle{   \prod\limits_{\substack{i\in I\\ j \in [n-\lambda_1-\mul(\lambda)+k]\setminus I}}(x_i-x_j)}  }
	\nonumber
\end{equation} 
\end{Lemma}

\begin{proof}
Note that the expression in \cref{eq:zschubert def sym} implies that $\Sym_{\lambda}^{n}(\z;\x;\y)$ is symmetric in variables $x_1,\ldots,x_{n-\lambda_1-\mul(\lambda)+k}$. Since we can take $k=\mul(\lambda)$, \cref{eq:zschubert def sym} implies that $\Sym_{\lambda}^{n}(\z;\x;\y)$ is symmetric in variables $x_1,\ldots,x_{n-\lambda_1}$. Now we use induction on $n$. Suppose the statement holds
	for partitions in $\Val(n-1)$. 

Let $\lambda'$ to be the partition obtained by deleting the first part of $\lambda$. By the induction hypothesis, $\Sym_{\lambda'}^{n-1}(\z;\x;\y)$ is symmetric in variables $x_1,\ldots,x_{(n-1)-\lambda'_1}.$ Note that we have $(n-1-\lambda'_1)\geq n-\lambda_1-\mul(\lambda)$. Thus we rewrite \eqref{eq: def zschubert} using \cref{CantiniLemma} as follows
\begin{equation}\label{eq: zschuert 1 expansion}
	\Sym_{\lambda}^{n}(\z;\x;\y)=\sum\limits_{i=1}^{n-\lambda_1-\mul(\lambda)+1} \frac{\Sym_{\lambda'}^{n-1}(\sigma^{\lambda_1-\lambda_{2}+1}(\z);\x_{\hat{i}};\y)\prod\limits_{l=1}^{n-\mul(\lambda)}(x_i-y_l)\prod\limits_{l=1}^{(\lambda_1-\lambda_2+1)}\prod\limits_{\substack{j=1\\j\neq i}}^{n-\lambda_1-\mul(\lambda)+1}(z_l-x_j)}{\prod\limits_{\substack{j=1\\ j\neq i}}^{n-\lambda_1-\mul(\lambda)+1}(x_i-x_j)},
\end{equation}
which implies the $k=1$ case of \cref{eq:zschubert def sym}.

Assume $k>1$. In this case we have $\mul(\lambda)>1$, which implies $\mul(\lambda')=\mul(\lambda)-1$ and $\lambda'_1=\lambda_2=\lambda_1$. Applying \cref{eq:zschubert def sym} to $\Sym_{\lambda'}^{n-1}(\z;\x;\y)$ and $(k-1)$ gives
\begin{align}\label{eq:inductionzschubert def sym}
    &\Sym_{\lambda'}^{n-1}(\z;\x;\y)=\\&\sum\limits_{\substack{I\subseteq[n-1-\lambda_1-\mul(\lambda)+k] \\|I|=k-1}}\frac{\Sym_{\tilde{\lambda}}^{n-k}(\sigma^{\lambda_1-\lambda_{k+1}+k-1}(\z);\x_{\hat{I}};\y)\prod\limits_{\substack{1\leq l\leq n-\mul(\lambda)\\i\in I}}(x_i-y_l)\prod\limits_{\substack{1\leq i\leq \lambda_1-\lambda_{k+1}+k-1\\j \in [n-1-\lambda_1-\mul(\lambda)+k]\setminus I }}(z_i-x_j)}{\prod\limits_{\substack{i\in I\\ j \in [n-1-\lambda_1-\mul(\lambda)+k]\setminus I}}(x_i-x_j)},\nonumber
\end{align}
and rewriting \eqref{eq: zschuert 1 expansion} using the fact $\lambda_1=\lambda_2$ gives
\begin{align}\label{eq: zschubert in the proof}
        \Sym_{\lambda}^{n}(\z;\x;\y)=\sum\limits_{j=1}^{n-\lambda_1-\mul(\lambda)+1}\frac{\Sym_{\lambda'}^{n-1}(\sigma(\z);\x_{\hat{j}};\y)\prod\limits_{l=1}^{n-\mul(\lambda)}(x_j-y_l)\prod\limits_{\substack{m=1\\m\neq j}}^{n-\lambda_1-\mul(\lambda)+1}(z_1-x_m)}{\prod\limits_{\substack{m=1 \\m\neq j}}^{n-\lambda_1-\mul(\lambda)+1}(x_j-x_m)}.
    \end{align}
	Plugging \eqref{eq:inductionzschubert def sym} into \eqref{eq: zschubert in the proof} gives
\begin{align}\label{eq:plugin}
     \Sym_{\lambda}^{n}(\z;\x;\y)=\sum\limits_{\substack{I\subseteq[n-1-\lambda_1-\mul(\lambda)+k] \\|I|=k-1}}\sum\limits_{j=1}^{n-\lambda_1-\mul(\lambda)+1}f_I(\sigma^{\lambda_1-\lambda_{k+1}+k}(\z);\x_{\hat{j}};\y)M_{I}(\z;\x_{\hat{j}};\y)\prod\limits_{l=1}^{n-\mul(\lambda)}(x_j-y_l)
\end{align}
where \begin{align*}
	f_{I}(\z;\x;\y)&=\Sym_{\tilde{\lambda}}^{n-k}(\z;\x_{\hat{I}};\y)\\
	M_{I}(\z;\x;\y)&=\frac{\prod\limits_{\substack{1\leq l\leq n-\mul(\lambda)\\i\in I}}(x_i-y_l)\prod\limits_{\substack{2\leq i\leq \lambda_1-\lambda_{k+1}+k\\m \in [n-1-\lambda_1-\mul(\lambda)+k]\setminus I }}(z_i-x_m)\prod\limits_{m=1}^{n-\lambda_1-\mul(\lambda)}(z_1-x_m)}{\prod\limits_{\substack{i\in I\\ m \in [n-1-\lambda_1-\mul(\lambda)+k]\setminus I}}(x_i-x_m)\prod\limits_{m=1}^{n-\lambda_1-\mul(\lambda)}(x_j-x_m)}.
\end{align*}
For a fixed $I_0=\{i_1<\cdots<i_{k}\}\subseteq[n-\lambda_1-\mul(\lambda)+k]$, to have
\begin{equation*}
    f_I(\sigma^{\lambda_1-\lambda_{k+1}+k}(\z);\x_{\hat{j}};\y)=\Sym_{\tilde{\lambda}}^{n-k}(\sigma^{\lambda_1-\lambda_{k+1}+k}(\z);\x_{\hat{I_0}};\y)
\end{equation*} we need to take $I=I_h=\{i_1<\cdots<i_{h-1}<i_{h+1}-1<\cdots<i_{k}-1\}$ and $j=i_h\leq n-\lambda_1-\mul(\lambda)+1$. So taking the coefficient of $\Sym_{\tilde{\lambda}}^{n-k}(\sigma^{\lambda_1-\lambda_{k+1}+k}(\z);\x_{\hat{I_0}};\y)$ in \eqref{eq:plugin} gives
\begin{equation*}
    \sum_{\substack{h\geq1\\i_h\leq n-\lambda_1-\mul(\lambda)+1}} M_{I_h}(\z;\x_{\hat{i_h}};\y)\prod\limits_{l=1}^{n-\mul(\lambda)}(x_{i_h}-y_l).
\end{equation*}
Note that we have
\begin{align*}
     &M_{I_h}(\z;\x_{\hat{i_h}};\y)\prod\limits_{l=1}^{n-\mul(\lambda)}(x_{i_h}-y_l)\\&=\frac{\prod\limits_{\substack{1\leq l\leq n-\mul(\lambda)\\i\in I_0}}(x_i-y_l)\prod\limits_{\substack{2\leq i\leq \lambda_1-\lambda_{k+1}+k\\m \in [n-\lambda_1-\mul(\lambda)+k]\setminus I_0 }}(z_i-x_m)\prod\limits_{\substack{m=1\\m\neq i_h}}^{n-\lambda_1-\mul(\lambda)+1}(z_1-x_m)}{\prod\limits_{\substack{i\in I_0, i\neq i_h\\ m \in [n-\lambda_1-\mul(\lambda)+k]\setminus I_0}}(x_i-x_m)\prod\limits_{\substack{m=1  \\m\neq i_h}}^{n-\lambda_1-\mul(\lambda)+1}(x_{i_h}-x_m)}.
\end{align*}
We claim
\begin{equation}\label{eq:mid}
    \frac{\prod\limits_{\substack{1\leq l\leq n-\mul(\lambda)\\i\in I_0}}(x_i-y_l)\prod\limits_{\substack{1\leq i\leq \lambda_1-\lambda_{k+1}+k\\m \in [n-\lambda_1-\mul(\lambda)+k]\setminus I_0 }}(z_i-x_m)}{\prod\limits_{\substack{i\in I_0\\ m \in [n-\lambda_1-\mul(\lambda)+k] \setminus I_0}}(x_i-x_m)}=\sum_{\substack{h\geq1\\i_h\leq n-\lambda_1-\mul(\lambda)+1}} M_{I_h}(\z;\x_{\hat{i_h}};\y)\prod\limits_{l=1}^{n-\mul(\lambda)}(x_{i_h}-y_l).
\end{equation}
We view both sides as polynomials in $z_1$ of degree at most $(n-\lambda_1-\mul(\lambda))$. So it is enough to show that they coincide when we plug in $z_1=x_1,\ldots,x_{n-\lambda_1-\mul(\lambda)+1}$. If we plug in $z=x_m$ for $m\notin I_0$ then both sides are zero. If we plug in $z_1=x_{i_{t}}$ for some $i_t\in I_0$ then only $h=t$ on the right hand side does not vanish and we have
\begin{equation*}
   \big(M_{I_t}(\z;\x_{\hat{i_t}};\y)\prod\limits_{l=1}^{n-\mul(\lambda)}(x_{i_t}-y_l)\big)\vert_{z_1=x_{i_{t}}}=\frac{\prod\limits_{\substack{1\leq l\leq n-\mul(\lambda)\\i\in I_0}}(x_i-y_l)\prod\limits_{\substack{2\leq i\leq \lambda_1-\lambda_{k+1}+k\\m \in [n-\lambda_1-\mul(\lambda)+k]\setminus I_0 }}(z_i-x_m)}{\prod\limits_{\substack{i\in I_0, i\neq i_t\\ m \in [n-\lambda_1-\mul(\lambda)+k]\setminus I_0}}(x_i-x_m)}.
\end{equation*}
This is the same as the left hand side of \eqref{eq:mid} evaluated at $z_1=x_{i_t}$.

In conclusion \eqref{eq:plugin} becomes
\begin{align*}
    &\Sym_{\lambda}^{n}(\z;\x;\y)=\\&\sum\limits_{\substack{I_0\subseteq[n-\lambda_1-\mul(\lambda)+k] \\|I_0|=k}}\frac{\Sym_{\tilde{\lambda}}^{n-k}(\sigma^{\lambda_1-\lambda_{k+1}+k}(\z);\x_{\hat{I_0}};\y)\prod\limits_{\substack{1\leq l\leq n-\mul(\lambda)\\i\in I_0}}(x_i-y_l)\prod\limits_{\substack{1\leq i\leq \lambda_1-\lambda_{k+1}+k\\m \in [n-\lambda_1-\mul(\lambda)+k]\setminus I_0 }}(z_i-x_m)}{\prod\limits_{\substack{i\in I_0\\ m \in [n-\lambda_1-\mul(\lambda)+k]\setminus I_0}}(x_i-x_m)},
\end{align*}
which completes the proof. 
\end{proof}

\begin{Prop}\label{prop113} Choose $\lambda\in\Val(n)$ and $a$ such that
	$1\leq a\leq n-\lambda_1$.
	If we specialize  $z_1:=x_a$ in the 
 $\mathbf{z}$-Schubert polynomial $\Sym_{\lambda}^{n}(\z;\x;\y)$, we obtain
\begin{align}\label{eq:zschubert specialization}
     &\Sym_{\lambda}^{n}(\z;\x;\y))\vert_{z_1=x_a}=\\&\Sym_{\lambda'}^{n-1}(\sigma^{\lambda_1-\lambda_2+1}(\z);\x_{\hat{a}};\y))\prod\limits_{l=1}^{n-\mul(\lambda)}(x_a-y_l)\prod\limits_{i=2}^{(\lambda_1-\lambda_2+1)}\prod\limits_{\substack{m=1\\m\neq a}}^{n-\lambda_1-\mul(\lambda)+1}(z_i-x_m),\nonumber
\end{align}
where $\lambda'$ is the partition obtained by deleting the first part of $\lambda$. If $\length(\lambda)=1$ then we regard $\lambda_2=0$.
\end{Prop}

\begin{proof}
Taking $k=\mul(\lambda)$ in \eqref{eq:zschubert def sym}, we see that $\Sym_{\lambda}^{n}(\z;\x;\y)$ is symmetric in variables $x_1,\ldots,x_{n-\lambda_1}$. So it is enough to show \eqref{eq:zschubert specialization} for $a=1$,
	which follows from  \eqref{eq: zschuert 1 expansion}.
\end{proof}

Note that in \cref{prop 115} and \ref{prop 116}, the operator $\partial_i^{\mathbf{z}}$ \emph{acts on $\mathbf{z}$-variables}.

\begin{Prop}\label{prop 115}
Let $\lambda\in\Val(n)$ such that $\lambda_1>\lambda_2$, and write $\lambda=(\lambda_1,\lambda')$ for some $\lambda'$. Then
\begin{align}\label{eq: prop11-5}
    &\Sym_{(\lambda_1,\lambda')}^{n}(\z;\x;\y)\prod\limits_{i=3}^{\lambda_1-\lambda_2+1}(z_{i}-x_{n+1-\lambda_1})\\&=\partial^{\mathbf{z}}_1\big(\Sym_{(\lambda_1-1,\lambda')}^{n}(\sigma(\z);\x;\y)\prod\limits_{i=1}^{n-\lambda_1}(z_1-x_i)\prod\limits_{i=1}^{\mul((\lambda_1-1,\lambda'))-1}(z_2-y_{n-i})\big)\nonumber.
\end{align}

\end{Prop}
\begin{proof} We will refer to the left and right-hand sides of 
	\eqref{eq: prop11-5} as LHS and RHS. To prove they are equal, 
	we will view each side as a polynomial in $z_1$ and analyze its degree to use interpolation. Note that $\Sym_{(\lambda_1-1,\lambda')}^{n}(\sigma(\z);\x;\y)$ does not depend on $z_1$ so exchanging variables $z_1$ and $z_2$ is the same as plugging in $z_1$ in  place of $z_2$. So we get
\begin{align*}
	RHS=\frac{1}{z_1-z_2}& \big[\Sym_{(\lambda_1-1,\lambda')}^{n}(\sigma(\z);\x;\y)\prod\limits_{i=1}^{n-\lambda_1}(z_1-x_i)\prod\limits_{i=1}^{\mul((\lambda_1-1,\lambda'))-1}(z_2-y_{n-i})\\
    &-\Sym_{(\lambda_1-1,\lambda')}^{n}(\sigma(\z);\x;\y)\vert_{z_2=z_1}\prod\limits_{i=1}^{n-\lambda_1}(z_2-x_i)\prod\limits_{i=1}^{\mul((\lambda_1-1,\lambda'))-1}(z_1-y_{n-i})\big].
\end{align*}
By \eqref{eq: def zschubert}, $\Sym_{(\lambda_1-1,\lambda)}^{n}(\z;\x;\y)$ has degree at most $(n-(\lambda_1-1)-\mul((\lambda-1,\lambda')))$ in $z_1$. So the numerator of $RHS$ has degree at most $(n-\lambda_1)$ in $z_1$, which implies 
	that $RHS$ has degree at most $(n-\lambda_1-1)$ in $z_1$. 
	Meanwhile, LHS 
	has degree at most $(n-\lambda_1-\mul(\lambda))=(n-\lambda_1-1)$ in $z_1$ by \eqref{eq: def zschubert}. So it is enough to show that LHS and RHS
	 coincide  when we plug in $z_1=x_h$ for any $1\leq h\leq n-\lambda_1$. 
	 We have
\begin{align*}
    RHS\vert_{z_1=x_h}&=\frac{-\Sym_{(\lambda_1-1,\lambda')}^{n}(\sigma(\z);\x;\y)\vert_{z_2=x_h}\prod\limits_{i=1}^{n-\lambda_1}(z_2-x_i)\prod\limits_{i=1}^{\mul((\lambda_1-1,\lambda'))-1}(x_h-y_{n-i})}{x_h-z_2}\\&=\Sym_{(\lambda_1-1,\lambda')}^{n}(\sigma(\z);\x;\y)\vert_{z_2=x_h}\prod\limits_{\substack{i=1\\i\neq h}}^{n-\lambda_1}(z_2-x_i)\prod\limits_{i=1}^{\mul((\lambda_1-1,\lambda'))-1}(x_h-y_{n-i}).
\end{align*}

And by \cref{prop113} we have
\begin{align*}
     &\Sym_{(\lambda_1-1,\lambda')}^{n}(\z;\x;\y))\vert_{z_1=x_h}\\&=\Sym_{\lambda'}^{n-1}(\sigma^{\lambda_1-\lambda_2}(\z);\x_{\hat{h}};\y))\prod\limits_{l=1}^{n-\mul((\lambda_1-1,\lambda'))}(x_h-y_l)\prod\limits_{i=2}^{(\lambda_1-\lambda_2)}\prod\limits_{\substack{m=1\\m\neq h}}^{n-\lambda_1-\mul((\lambda_1-1,\lambda'))+2}(z_i-x_m)\\&=\Sym_{\lambda'}^{n-1}(\sigma^{\lambda_1-\lambda_2}(\z);\x_{\hat{h}};\y))\prod\limits_{l=1}^{n-\mul((\lambda_1-1,\lambda'))}(x_h-y_l)\prod\limits_{i=2}^{(\lambda_1-\lambda_2)}\prod\limits_{\substack{m=1\\m\neq h}}^{n-\lambda_1+1}(z_i-x_m),
\end{align*}
where the last equality uses the fact that when $\mul((\lambda_1-1,\lambda'))>1$ we have $\lambda_1-\lambda_2=1$ so the product over $i=2$ to $\lambda_1-\lambda_2$ is vacuous. Shifting variables by $\z\rightarrow \sigma(\z)$ we deduce
\begin{align*}
   \Sym_{(\lambda_1-1,\lambda)}^{n}(\sigma(\z);\x;\y)\vert_{z_2=x_h}=\Sym_{\lambda'}^{n-1}(\sigma^{\lambda_1-\lambda_2+1}(\z);\x_{\hat{h}};\y)\prod\limits_{l=1}^{n-\mul((\lambda_1-1,\lambda'))}(x_h-y_l)\prod\limits_{i=3}^{(\lambda_1-\lambda_2+1)}\prod\limits_{\substack{m=1\\m\neq h}}^{n-\lambda_1-1}(z_i-x_m).
\end{align*}
Plugging this into $RHS\vert_{z_1=x_h}$ gives
\begin{align*}
    RHS\vert_{z_1=x_h}&=\Sym_{\lambda'}^{n-1}(\sigma^{\lambda_1-\lambda_2+1}(\z);\x_{\hat{h}};\y)\prod\limits_{i=1}^{n-1}(x_h-y_{i})\prod\limits_{\substack{i=1\\i\neq h}}^{n-\lambda_1}(z_2-x_i)\prod\limits_{i=3}^{(\lambda_1-\lambda_2+1)}\prod\limits_{\substack{m=1\\m\neq h}}^{n-\lambda_1+1}(z_i-x_m)\\&=\Sym_{\lambda'}^{n-1}(\sigma^{\lambda_1-\lambda_2+1}(\z);\x_{\hat{h}};\y)\prod\limits_{i=1}^{n-1}(x_h-y_{i})\prod\limits_{i=2}^{(\lambda_1-\lambda_2+1)}\prod\limits_{\substack{m=1\\m\neq h}}^{n-\lambda_1}(z_i-x_m)\prod\limits_{i=3}^{\lambda_1-\lambda_2+1}(z_i-x_{n+1-\lambda_1})\\&=\Sym_{(\lambda_1,\lambda')}^{n}(\z;\x;\y)\vert_{z_1=x_h}\prod\limits_{i=3}^{\lambda_1-\lambda_2+1}(z_i-x_{n+1-\lambda_1})=LHS\vert_{z_1=x_h}.\end{align*}

\end{proof}

\begin{Prop}\label{prop 116} Choose $\lambda\in\Val(n)$ with $\mul(\lambda)=b>1$. Write $\lambda=((\lambda_1)^b,\tilde{\lambda})$ for some $\tilde{\lambda}$. We have
\begin{align}
	\Sym_{((\lambda_1)^{b},\tilde{\lambda})}^{n}&(\z;\x;\y) 
	\prod\limits_{i=1}^{b-1}(z_i-y_{n+1-b})\prod\limits_{i=b+2}^{b+\lambda_1-\tilde{\lambda}_1}(z_i-x_{n+1-\lambda_1})\\&=\partial^{\mathbf{z}}_b\left(\Sym_{((\lambda_1)^{b-1},\lambda_1-1,\tilde{\lambda})}^{n}(\z;\x;\y)\prod\limits_{i=1}^{\mul((\lambda_1-1,\tilde{\lambda}))-1}(z_{b+1}-y_{n+1-b-i})\right)\nonumber
\end{align}

\end{Prop}

\begin{proof}
We view both sides as polynomials in $z_1$ and analyze their degrees to use interpolation. By \eqref{eq: def zschubert}, $\Sym_{((\lambda_1)^{b},\tilde{\lambda})}^{n}(\z;\x;\y)$ is of degree at most $(n-\lambda_1-b)$ in $z_1$ so the right hand side is of  degree at most $(n-\lambda_1-b+1)$ in $z_1$. Likewise, $\Sym_{((\lambda_1)^{b-1},\lambda_1-1,\tilde{\lambda})}^{n}(\z;\x;\y)$ is of degree at most $(n-\lambda_1-b+1)$ in $z_1$ so the right hand side is of  degree at most $(n-\lambda_1-b+1)$ in $z_1$. Since $b>1$, it is enough to show that they coincide when we plug in $z_1=x_a$ for $1 \leq a\leq n-\lambda_1$.

We use induction on $b$. First assume $b=2$. Setting $z_1:=x_a$ on both sides, \cref{prop113} gives
\begin{align*}
    LHS\vert_{z_1=x_a}&=\Sym_{(\lambda_1,\tilde{\lambda})}^{n-1}(\z;\x_{\hat{a}};\y)\prod\limits_{l=1}^{n-1}(x_a-y_l)\prod\limits_{i=4}^{2+\lambda_1-\tilde{\lambda}_1}(z_i-x_{n+1-\lambda_1})\\
    RHS\vert_{z_1=x_a}&=\partial^{\mathbf{z}}_2\left(\Sym_{(\lambda_1-1,\tilde{\lambda})}^{n-1}(\sigma(\z);\x_{\hat{a}};\y)\prod\limits_{l=1}^{n-1}(x_a-y_l)\prod\limits_{\substack{m=1\\m\neq a}}^{n-\lambda_1}(z_2-x_m)\prod\limits_{i=1}^{\mul((\lambda_1-1,\tilde{\lambda}))-1}(z_{3}-y_{n-1-i})\right).
\end{align*}
So the equality $LHS\vert_{z_1=x_a}=RHS\vert_{z_1=x_a}$ comes from 
\cref{prop 115}, with the variables shifted by $\x\rightarrow\x_{\hat{a}}$ and  $\z\rightarrow\sigma(\z)$.

For $b>2$, assume the statement holds for $(b-1)$. Setting $z_1:=x_a$ on both sides, \cref{prop113} gives
\begin{align*}
    LHS\vert_{z_1=x_a}&=\Sym_{((\lambda_1)^{b-1},\tilde{\lambda})}^{n-1}(\z;\x_{\hat{a}};\y)\prod\limits_{l=1}^{n-b+1}(x_a-y_l)\prod\limits_{i=2}^{b-1}(z_i-y_{n+1-b})\prod\limits_{i=b+2}^{b+\lambda_1-\tilde{\lambda}_1}(z_i-x_{n+1-\lambda_1})\\
    RHS\vert_{z_1=x_a}&=\partial^{\mathbf{z}}_b\left(\Sym_{((\lambda_1)^{b-2},\lambda_1-1,\tilde{\lambda})}^{n-1}(\z;\x_{\hat{a}};\y)\prod\limits_{l=1}^{n-b+1}(x_a-y_l)\prod\limits_{i=1}^{\mul((\lambda_1-1,\tilde{\lambda}))-1}(z_{b+1}-y_{n+1-b-i})\right).
\end{align*}
The equality $LHS\vert_{z_1=x_a}=RHS\vert_{z_1=x_a}$ comes from  the
induction hypothesis with the variables shifted by $\x\rightarrow\x_{\hat{a}}$ and  $\z\rightarrow\sigma(\z)$.
\end{proof}




\bibliographystyle{alpha}
\bibliography{bibliography}

\end{document}